\documentclass{article}

\usepackage{microtype}
\usepackage{graphicx}
\usepackage{subfigure}
\usepackage{booktabs} % for professional tables

\usepackage{lipsum}
\usepackage{cuted}
\usepackage{multirow}
\usepackage{colortbl}
\usepackage{makecell}
\usepackage{booktabs}
\usepackage{color}
\usepackage{tcolorbox}
\usepackage{tikz}
\usetikzlibrary{positioning, shapes.geometric}
\usepackage{hyperref}

\usepackage[accepted]{icml2023}

\usepackage{amsmath}
\usepackage{amssymb}
\usepackage{mathtools}
\usepackage{amsthm}

\usepackage[capitalize,noabbrev]{cleveref}
\makeatletter
\newcommand*{\rom}[1]{\expandafter\@slowromancap\romannumeral #1@}
\makeatother
\theoremstyle{plain}
\newtheorem{theorem}{Theorem}[section]
\newtheorem{proposition}[theorem]{Proposition}
\newtheorem{lemma}[theorem]{Lemma}
\newtheorem{corollary}[theorem]{Corollary}
\theoremstyle{definition}
\newtheorem{definition}[theorem]{Definition}
\newtheorem{assumption}[theorem]{Assumption}
\theoremstyle{remark}
\newtheorem{remark}[theorem]{Remark}

\usepackage[textsize=tiny]{todonotes}

\icmltitlerunning{Accelerated Stochastic Optimization Methods under Quasar-convexity}

\begin{document}

\twocolumn[
\icmltitle{Accelerated Stochastic Optimization Methods under Quasar-convexity}

% It is OKAY to include author information, even for blind
% submissions: the style file will automatically remove it for you
% unless you've provided the [accepted] option to the icml2023
% package.

% List of affiliations: The first argument should be a (short)
% identifier you will use later to specify author affiliations
% Academic affiliations should list Department, University, City, Region, Country
% Industry affiliations should list Company, City, Region, Country

% You can specify symbols, otherwise they are numbered in order.
% Ideally, you should not use this facility. Affiliations will be numbered
% in order of appearance and this is the preferred way.
\icmlsetsymbol{equal}{*}

\begin{icmlauthorlist}
\icmlauthor{Qiang Fu}{QF}
\icmlauthor{Dongchu Xu}{DCX}
\icmlauthor{Ashia Wilson}{AW}
%\icmlauthor{}{sch}
%\icmlauthor{}{sch}
\end{icmlauthorlist}

\icmlaffiliation{QF}{Sun Yat-sen University, Guangzhou, China}
\icmlaffiliation{DCX}{Harvard University, Cambridge, MA, USA}
\icmlaffiliation{AW}{MIT, Cambridge, MA, USA}

\icmlcorrespondingauthor{Qiang Fu}{fuqiang7@mail2.sysu.edu.cn}
\icmlcorrespondingauthor{Ashia Wilson}{ashia07@mit.edu}
% You may provide any keywords that you
% find helpful for describing your paper; these are used to populate
% the "keywords" metadata in the PDF but will not be shown in the document
\icmlkeywords{Machine Learning, ICML}

\vskip 0.3in
]

% this must go after the closing bracket ] following \twocolumn[ ...

% This command actually creates the footnote in the first column
% listing the affiliations and the copyright notice.
% The command takes one argument, which is text to display at the start of the footnote.
% The \icmlEqualContribution command is standard text for equal contribution.
% Remove it (just {}) if you do not need this facility.

\printAffiliationsAndNotice{}  % leave blank if no need to mention equal contribution
%\printAffiliationsAndNotice{\icmlEqualContribution} % otherwise use the standard text.

\begin{abstract}
Non-convex optimization plays a key role in a growing number of machine learning applications. This motivates the identification of specialized structure that enables sharper theoretical analysis. One such identified structure is quasar-convexity, a non-convex generalization of convexity that subsumes convex functions. Existing algorithms for minimizing quasar-convex functions in the stochastic setting  have either high complexity or slow convergence, which prompts us to derive a new class of stochastic methods for optimizing smooth quasar-convex functions. We demonstrate that our algorithms have fast convergence and outperform existing algorithms on several examples, including the classical problem of learning linear dynamical systems. We also present a unified analysis of our newly proposed algorithms and a previously studied deterministic algorithm.
\end{abstract}

\section{Introduction}
\label{submission}

Momentum is one of the most widely used techniques for speeding up the convergence rate of optimization methods. Many deterministic and stochastic momentum based algorithms have been proposed for optimizing (strongly) convex functions, e.g. accelerated gradient descent (AGD) \cite{Nesterov1983, nesterov2003introductory, beck2009}, accelerated stochastic gradient descent (ASGD) \cite{ghadimi2012optimal, ghadimi2016, kulunchakov2020}, accelerated stochastic variance reduced gradient (ASVRG) methods and their related variants \cite{nitanda2016, allen2017katyusha, kulunchakov2020}.

While much of our understanding of modern optimization algorithms relies on the ability to leverage the convexity of the objective function, a growing number of modern machine learning applications rely on non-convex optimization. Unfortunately, the theoretically guaranteed improvement for convex functions that accelerated algorithms have do not apply to many real-world scenarios.
For many smooth non-convex optimization problems, we only have guarantees for finding  stationary points instead of the global minimizer. However, some non-convex functions involved in several popular optimization problems such as low-rank matrix problems, deep learning and reinforcement learning, have special structure and exhibit convex-like properties \cite{ge2016matrix, bartlett2018deep, mei2020}, which makes it possible to find approximate global minimizers of these structured non-convex functions.

In this paper, we develop two accelerated stochastic optimization methods for optimizing quasar-convex functions. A quasar-convex function is parameterized by a constant $\gamma\in(0,1]$. $\gamma=1$ implies the function is star-convex, which is a relaxation of convexity \cite{nesterov2006cubic}. Quasar-convexity was first proposed in \citet{hardt}. They prove that the objective of learning linear dynamical systems is quasar-convex under several mild assumptions. \citet{zhou2019sgd} and \citet{kleinberg2018} also provide evidence to suggest that loss function of neural networks may conform to star-convexity in large neighborhoods of the minimizers. Several recent papers propose effective deterministic methods for minimizing $L$-smooth and $\gamma$-quasar-convex functions. While gradient descent (GD) and stochastic gradient descent (SGD) need $O(\gamma^{-1}\epsilon^{-1})$ and $O(\gamma^{-2}\epsilon^{-2})$ iterations to yield an $\epsilon$-approximate solution \citealt{guminov, gower2021}, the algorithms developed by \citet{guminov} and \citet{hinder} need $O(\gamma^{-1}\epsilon^{-1/2})$ iterations and the algorithm developed by~\citet{nesterov_primal} needs $O(\gamma^{-3/2}\epsilon^{-1/2})$ iterations. \citet{hinder} also introduce a new metric in terms of the total number of function and gradient evaluations. In order to compute an $\epsilon$-approximate solution, the method of \citet{hinder} requires $O(\gamma^{-1}\epsilon^{-1/2}\log(\gamma^{-1}\epsilon^{-1}))$ total evaluations for $\gamma$-quasar-convex functions and $O(\gamma^{-1}\kappa^{1/2}\log(\gamma^{-1}\kappa)\log(\gamma^{-1}\epsilon^{-1}))$\footnote{$\kappa\triangleq L/\mu$ is the condition number.} total evaluations for $\mu$-strongly $\gamma$-quasar-convex functions.

Many optimization problems in machine learning can be expressed in the following format
\begin{equation}
\label{problem}
    \underset{x\in\mathbb{R}^d}{\min} \left[f(x)=\frac{1}{n}\sum_{i=1}^n f_i(x)\right],
\end{equation}
which makes them particularly well-suited for stochastic optimization methods. When $n$ is large, applying deterministic algorithms will lead to high computational cost due to the full gradient and function value access required per iteration. Therefore, motivated by ASGD, ASVRG in the convex setting as well as the contributions of \citet{hinder}, we propose both a quasar-accelerated stochastic gradient descent (QASGD) method and a quasar-accelerated stochastic variance reduced gradient (QASVRG) method for solving \eqref{problem}, where the objective function $f$ is $L$-smooth and (strongly) quasar-convex. We also present a unified energy-based framework to analyze the convergence of these newly proposed accelerated algorithms, drawing inspiration from the unified analyses developed by \citet{wilson2021} and \citet{kulunchakov2020}.

Our principal contributions are three-fold.
\begin{itemize}
\item \textit{QASGD}: We introduce QASGD with momentum as the acceleration technique. Under a bounded gradient assumption~\ref{boundedgrad}, we prove that QASGD achieves convergence rates of $O\left(\frac{L}{t^2}+\frac{\sigma}{\gamma\sqrt{t}}+\frac{\epsilon}{2}\right)$ for general quasar-convex functions and $O\left((1+{\gamma^2}/{16})^{-t}+\frac{\sigma^2}{\gamma^2 t}\right)$ for strongly quasar-convex functions, where $\epsilon$ comes from a binary line search. We empirically demonstrate that on learning time-invariant dynamical systems, QASGD outperforms several existing proposed methods.
\item \textit{QASVRG}: We introduce QASVRG in a mini-batch setting, which is an extension of \citet{nitanda2016} to quasar-convexity with momentum as the acceleration technique. Variance reduction and mini-batches are employed to compute the stochastic gradient per iteration. Under an interpolation assumption \ref{interpolation} and a compactness assumption \ref{compactset}, QASVRG achieves an overall complexity\footnote{Here we use overall complexity to denote the total number of function and gradient evaluations} of
$\widetilde{O}\left(n+\min\left\{\frac{\kappa}{\gamma^2},\frac{n\sqrt{\kappa}}{\gamma}\right\}\right)$ and $\widetilde{O}\left(n+\min\left\{\frac{LR^2}{\gamma\epsilon},\frac{nR}{\gamma}\sqrt{\frac{L}{\epsilon}}\right\}\right)$
for strongly quasar-convex functions and general quasar-convex functions. We also propose an alternative scheme for strongly quasar-convex functions with different parameter choice (Option \rom{2} in Table \ref{parameterofQASVRG}), whose precise convergence rates are postponed to Theorem \ref{complexASVRG}. These two schemes have different dependency on $\kappa$ and $\epsilon$ and thus are suitable to different application scenarios. When $n$ is large, our complexity is significantly lower than the complexity of AGD in \citet{hinder}.
\item \textit{Lyapunov analysis}: We present a unified analysis for our proposed algorithms under quasar-convexity and smoothness using a standard Lyapunov argument. Additionally, we incorporate the AGD method proposed in \citet{hinder} in our energy-based framework, which we rename QAGD (quasar-accelerated gradient descent). Different from AGD in \citet{hinder}, QAGD admits the Bregman divergence which is more general than the Euclidean distance. %generalize the Euclidean distance to the Bregman distance for QAGD.
\end{itemize}
The remainder of this paper is organized as follows. Section 2 presents more details about quasar-convexity, related assumptions, and previously proposed methods. Section 3 presents the main algorithms of QAGD, QASGD and QASVRG for (strongly) quasar-convex functions and their convergence analysis. Section 4 describes our simulations verifying the effectiveness of our proposed algorithms.
\paragraph{Notation} The following notation is used throughout the paper: 
$D_h(x,y)\triangleq h(x)-h(y)-\langle\nabla h(y),x-y\rangle$ denotes the Bregman divergence between $x,y\in\mathbb{R}^d$, where h is an arbitrary $\bar{\mu}$-strongly convex function. $[n]\triangleq\{1,2,...,n\}$
$\log^+()\triangleq\max\{\log(),1\}$,
$\|\cdot\|\triangleq\|\cdot\|_2$, $\bar{a}_k\triangleq A_{k+1}-A_k$, $\bar{b}_k\triangleq B_{k+1}-B_k$, $\kappa\triangleq L/\bar{\mu}\mu$,
$\mathcal{E}_k\triangleq f(y_k)-f(x^*)$. 
$a\simeq b$ signifies $a=O(b)$.
$\langle,\rangle$ represents the inner product. $\mathcal{X}^*$ is the solution set of \eqref{problem} which we assume is not empty, and a point $x$ is an $\epsilon$-approximate solution if $f(x)-f(x^*)\leq\epsilon$ for $x^*\in\mathcal{X}^*$. $R$ denotes the upper bound of the initial distance such that $D_h(x^*,x_0)\leq R^2$. We assume $f(x^*)\geq 0$ without loss of generality. $f$ is $L$-smooth, if $\|\nabla f(x)-\nabla f(y)\|\leq L\|x-y\|$ for all $x,y\in\mathbb{R}^d$. $\mathcal{Q}_{\mu\gamma},\mathcal{F}_L$ respectively denote the set of $\mu$-strongly $\gamma$-quasar-convex functions and the set of $L$-smooth functions, and $\mathcal{Q}_{\mu\gamma}$ reduces to the set of $\gamma$-quasar-convex functions when $\mu=0$. We use $O(\cdot)$ to hide constants and $\widetilde{O}(\cdot)$ to hide logarithmic factors and constants.

\section{Background}
There has been growing interest in exploiting structure present in large classes of non-convex functions. One such structure is quasar-convexity and strong quasar-convexity, defined as follows. 
\begin{definition}[Quasar-convexity]
Let $\gamma\in(0,1]$ and let $x^*$ be a minimizer of the differentiable function $f:\mathbb{R}^d\rightarrow \mathbb{R}.$ A function is $\gamma$-quasar-convex with respect to $x^*$ if for all $x\in\mathbb{R}^d$,
\begin{equation}
\label{quasar-convex}
    f(x^*)\geq f(x)+\frac{1}{\gamma}\langle\nabla f(x),x^*-x\rangle.
\end{equation}
For $\mu>0$, a function is $\mu$-strongly $\gamma$-quasar-convex with respect to $x^*$ if for all $x\in\mathbb{R}^d$,
\begin{equation}
\label{strongly quasar-convex}
    f(x^*)\geq f(x)+\frac{1}{\gamma}\langle\nabla f(x),x^*-x\rangle+\frac{\mu}{2}\|x^*-x\|^2.
\end{equation}
\end{definition}
%Another standard assumption is $L$-smoothness, which implies the following relation:
%\begin{equation}
%\label{lsmooth}
%    f(y)\leq f(x)+\langle\nabla f(x),y-x\rangle+\frac{L}{2}\|y-x\|^2,
%\end{equation}
%for all $x,y\in\mathbb{R}^d$ and $L\geq 0$.
\subsection{Examples}
We introduce a classical example in \citet{hardt} of learning linear dynamical systems (LDS). Consider the following time-invariant linear dynamical system
\begin{subequations}
    \begin{align}
    h_{t+1}&=Ah_t+Bx_t\label{trueLDS1}\\
    y_t&=Ch_t+Dx_t+\xi_t,\label{trueLDS2}
    \end{align}
\end{subequations}
where $x_t\in\mathbb{R},y_t\in\mathbb{R}$ are the input and output of time $t$; $\xi_t$ is a random perturbation sampled i.i.d from a distribution; $h_t\in\mathbb{R}^d$ is the hidden state and $\Theta\triangleq(A,B,C,D)\in\mathbb{R}^{d\times d}\times\mathbb{R}^{d\times 1}\times\mathbb{R}^{1\times d}\times\mathbb{R}$ is the true parameter that we aim to learn. Assuming we have $N$ pairs of training examples $S=\{(x^{(1)},y^{(1)}),...,(x^{(N)},y^{(N)})\}$ where each input sequence $x\in\mathbb{R}^T$ is sampled from a distribution and $y$ is the corresponding output of the system above, we fit these training examples to the following model
\begin{subequations}
    \begin{align}
    \hat{h}_{t+1}&=\hat{A}\hat{h}_t+\hat{B}x_t\label{modelLDS1}\\
    \hat{y}_t&=\hat{C}\hat{h}_t+\hat{D}x_t,\label{modelLDS2}
    \end{align}
\end{subequations}
which is governed by $\Theta\triangleq(\hat{A},\hat{B},\hat{C},\hat{D})$. According to the training examples and the model system, we consider the following optimization problem
\begin{equation}
\label{learningLDS}
    \min \left\{F(\hat{\Theta})=\mathbb{E}_{\{x_t\},\{\xi_t\}}\left[\frac{1}{T}\sum_{t=1}^T\|\hat{y}_t-y_t\|^2\right]\right\}.
\end{equation}
\citet{hardt} demonstrate that the objective function $F(\hat{\Theta})$ is weakly smooth and quasar-convex with respect to $\Theta$ under some mild conditions.

We introduce another example of generalized linear models (GLM). Consider the following square loss minimization problem
\begin{equation}
\label{GLM}
\min\left\{f(w):=\mathbb{E}_{x\sim\mathcal{D}}\left[\frac{1}{2}(\sigma(w^{\mathsf{T}}x)-y)^2\right]\right\}
\end{equation}
where $\sigma(\cdot):\mathbb{R}\rightarrow\mathbb{R}$ is the link function; $x\in\mathbb{R}^d$ is i.i.d from $\mathcal{D}$ and there exists $w_*\in\mathbb{R}^d$ such that $y=\sigma({w_*}^{\mathsf{T}}x)$. The quasar-convex structure of $f(w)$ has been exploited in several literature \cite{foster2018uniform, ma2020local, wang2023continuized}.
 
\begin{table*}[t]%\large
%\vspace{-0.3cm}
\centering
%\renewcommand\arraystretch{1.5}
%\resizebox{\textwidth}{25mm}{%
%\setlength{\tabcolsep}{10mm}{
\begin{tabular}{llll}
\toprule
\textbf{Method} & \textbf{Assumptions} & \textbf{Complexity} \\ \hline
GD \cite{guminov}     &      $f\in\mathcal{F}_L$        &      $\widetilde{O}\left(\frac{nLR^2}{\gamma\epsilon}\right)$        \\ \hline
SGD \cite{gower2021}    &   \makecell[l]{$f\in\mathcal{F}_L$ \& \hyperlink{ERcondition}{ER Condition}\\ \\$f_i\in\mathcal{F}_L$ \& \hyperlink{interpolation}{Interpolation}}          &      \makecell[l]{$\widetilde{O}\left(\frac{(R^2+\gamma^2\lambda^2)^2}{\gamma^2\epsilon^2}\right)$\\$\widetilde{O}\left(\frac{LR^2}{\gamma^2\epsilon}\right)$}        \\ \hline
SGD \cite{jin2020convergence}    &    $f\in\mathcal{F}_L$ \& \hyperlink{boundedvar}{Bounded Variance}    &    $\widetilde{O}\left(\frac{LR^2}{\gamma\epsilon}+\frac{\bar{\sigma}^2R^2}{\gamma^2\epsilon^2}\right)$        \\ \hline
QAGD \cite{hinder}   &   $f\in\mathcal{F}_L$          &           $\widetilde{O}\left(\frac{nR}{\gamma}\sqrt{\frac{L}{\epsilon}}\right)$       \\ \hline
QASGD (Ours)  &   $f_i\in\mathcal{F}_L$ \& \hyperlink{boundedgrad}{Bounded Gradient}          &        $\widetilde{O}\left(R\sqrt{\frac{L}{\epsilon}}+ \frac{\sigma^2R^2}{\gamma^2\epsilon^2}\right)$         \\ \hline
QASVRG (Ours) &    $f_i\in\mathcal{F}_{L}$ \& \hyperlink{interpolation}{Interpolation} \& \hyperlink{compactset}{Compactness}         &            $\widetilde{O}\left(n+\min\left\{\frac{LR^2}{\gamma\epsilon},\frac{nR}{\gamma}\sqrt{\frac{L}{\epsilon}}\right\}\right)$     \\ \bottomrule
\end{tabular}
%}
\caption{Comparison between some existing methods and our methods when $f\in\mathcal{Q}_{0\gamma}$}
\label{summary}
\end{table*}
\subsection{Prior Deterministic Methods}
Several deterministic first-order methods have been developed to minimize $L$-smooth $\gamma$-quasar-convex functions.

\citet{guminov} prove that gradient descent achieves a convergence rate of $O(L/\gamma t)$. They also propose an accelerated algorithm achieving a convergence rate of $O(L/\gamma^2t^2)$. This algorithm, however, depends on a low-dimensional subspace optimization method at each iteration, which is possibly prohibitively expensive to perform.

\citet{hinder} propose a novel accelerated  gradient method achieving a convergence rate of $O((1-\gamma/\sqrt{2\kappa})^t)$ for strongly quasar-convex functions. Notably, when $\gamma=1$, this rate matches the convergence rate achieved by Nesterov's AGD for strongly convex functions. The method introduced by \citet{hinder} also achieves a convergence rate of $O(L/\gamma^2t^{2}+\epsilon/2)$ for general quasar-convex functions, which nearly matches the convergence rate achieved by Nesterov's AGD for convex functions when $\gamma=1$. An additional factor $\epsilon$ (which can be made arbitrarily small) appears in the convergence rate due to a binary line search subroutine introduced in order to search for the appropriate momentum parameters. Notably,  the momentum parameters classically chosen in accelerated gradient descent (\citealt{Nesterov1983}) are not guaranteed to perform well under quasar-convexity. 
Compared with the low-dimensional subspace method in \citet{guminov}, the binary line search in ~\citet{hinder}'s AGD achieves at most $O(\log(\gamma^{-1}\epsilon^{-1}))$ function and gradient evaluations, which can be considerably cheaper. Analogously, \citet{bu2020note} propose momentum-based accelerated algorithms relying on a subroutine but without complexity analysis of the subroutine. \citet{hinder} also establishes a worst case complexity lower bound of $\Omega(\gamma^{-1}\epsilon^{-1/2})$ for any deterministic first-order methods applied to quasar-convex functions, and their methods are optimal up to a logarithmic factor. The complexity of methods in \citet{guminov} conditionally matches this lower bound.
\subsection{Prior Stochastic Methods}
While deterministic accelerated methods for quasar-convex functions achieve fast convergence rates and near-optimal complexity, we focus on the development of stochastic methods to reduce the computational complexity when solving \eqref{problem}.
When the objective $f$ is $\gamma$-quasar-convex,
\citealt{hardt} show that SGD achieves a convergence rate of $O(\Gamma/\gamma^2t+\bar{\sigma}/\gamma\sqrt{t})$ under the assumptions of $\bar{\sigma}^2$-bounded variance and $\Gamma$-weak-smoothness.
\begin{tcolorbox}[frame empty, left = 1mm, right = 1mm, top = 1mm, bottom = 1mm]
\begin{assumption}[Bounded Variance]
\label{boundedvar}
\hypertarget{boundedvar}{
    Suppose $i$ is sampled i.i.d from $[n]$. For some constant $\bar{\sigma}$, we have
    \begin{equation*}
        \mathbb{E}_i\left[\|\nabla f_i(x)-\nabla f(x)\|^2\right]\leq \bar{\sigma}^2.
    \end{equation*}}
\end{assumption}
\end{tcolorbox}
The weak-smoothness assumption is milder than $L$-smoothness. \citet{gower2021} propose a stochastic gradient method achieving a convergence rate of $O(\lambda^2/\gamma\sqrt{t})$ under $L$-smoothness and Assumption \ref{ERcondition}.
\begin{tcolorbox}[frame empty, left = 1mm, right = 1mm, top = 1mm, bottom = 1mm]
\begin{assumption}[ER Condition]
\label{ERcondition}
\hypertarget{ERcondition}{
    Suppose $i$ is sampled i.i.d from $[n]$. For some constants $\rho$ and $\lambda$, we have
    \begin{equation*}
        \mathbb{E}_i\left[\|\nabla f_i(x)\|^2\right]\leq 4\rho(f(x)-f(x^*))+\|\nabla f(x)\|^2+2\lambda^2.
    \end{equation*}}
\end{assumption}
\end{tcolorbox}
Compared with \citet{hardt}, the smoothness assumption in \citet{gower2021} is stronger, but the assumption on the gradient estimate is weaker in a sense. Moreover, \citet{gower2021} demonstrate that this rate can be improved to $O(L/\gamma^2t)$ under Assumption~\ref{interpolation}.
Under smoothness and 
bounded variance, \citet{jin2020convergence} provides a sharper analysis of SGD compared with \citet{gower2021} and extends the analysis to the non-smooth setting.

There are several accelerated stochastic methods that can theoretically achieve better worst-case convergence rates than SGD when the objective is convex. In the convex setting, the objective function usually includes a regularizer term, which is convex lower semi-continuous and not necessarily smooth. \citet{ghadimi2016} and \citet{kulunchakov2020} propose proximal ASGD which  achieves convergence rates of $O(L/t^2+\sigma/\sqrt{t})$ and $O((1-1/\sqrt{\kappa})^t+\sigma^2/t)$ for general convex and strongly convex functions respectively under $L$-smoothness and the $\sigma^2$-bounded variance assumption. Variance reduction is a powerful technique to achieve a better convergence rate. \citet{allen2017katyusha} and \citet{kulunchakov2020} propose accelerated proximal SVRG with a convergence rate guarantee of $O((1-\min\{1/\sqrt{3\kappa n}, 1/\sqrt{2n}\})^t)$ and $O(Ln/t^2)$ for $L$-smooth (strongly) convex functions. \citet{nitanda2016} proposes accelerated mini-batch SVRG methods for minimizing (strongly) convex finite sum without regularizer. This algorithm is a multi-stage scheme achieving convergence rates of $\widetilde{O}(n+\min\{\kappa,n\sqrt{\kappa}\})$ and $\widetilde{O}(n+\min\{L/\epsilon,n\sqrt{L/\epsilon}\})$ for $L$-smooth (strongly) convex functions. By contrast, the methods they use to update the fixed anchor point of SVRG and control the variance are different.

\subsection{Motivation}
Inspired by the accelerated stochastic methods discussed above, we extend ASGD of \citet{kulunchakov2020} and AMSVRG of \citet{nitanda2016} to the (strongly) quasar-convex setting under different assumptions. In this subsection we will discuss these assumptions and how they are compared to prior sets of assumptions. Different from \citet{kulunchakov2020}, we do not consider random perturbations of the function value and gradient given $x^*$ may not be the global minimizer after perturbation. Furthermore, binary line search \cite{hinder} is incorporated into each of our proposed methods for finding the appropriate momentum parameters. %Furthermore, our algorithms can also be applied under convexity. 

For QASGD, our key assumption is the bounded gradient assumption, which is a frequently used assumption in the standard convergence analysis of SGD in the non-convex setting \cite{hazan2014beyond, rakhlin2011making, recht2011hogwild, nemirovski2009robust}. Due to the special structure of strongly quasar-convex functions whose gradient is not bounded, we generalize this assumption as follows.
\begin{tcolorbox}[frame empty, left = 1mm, right = 1mm, top = 1mm, bottom = 1mm]
\begin{assumption}[Bounded Gradient]
\label{boundedgrad}
\hypertarget{boundedgrad}{
    Suppose $f\in\mathcal{Q}_{\mu\gamma}$ and $i$ is sampled i.i.d from $[n]$. For some $\sigma\geq 0$ and $x^*\in\mathcal{X}^*$, we have
    \begin{equation*}
        \mathbb{E}_i\left[\|\nabla f_i(x)\|^2\right] \leq \sigma^2 + 2\mu^2\|x^*-x\|^2.
    \end{equation*}}
\end{assumption}
\end{tcolorbox}
 This assumption will reduce to the standard bounded gradient assumption under general quasar-convexity. Compared with ER condition, the bounded gradient assumption is stronger. Example \eqref{GLM} satisfies this assumption ($\mu=0$) if we choose the link functions to be logistic. For $\mu>0$, we consider a quasar-convex finite sum $f=\sum_{i=1}^nf_i$ where $x^*$ is the minimizer of $f$ and $\mathbb{E}_i[\|\nabla f_i\|^2]\leq\sigma^2$. Then $g(x)=f(x)+\frac{\mu}{2}\|x-x^*\|^2$ is strongly quasar-convex and satisfies this assumption. While some quasar-convex functions intrinsically do not satisfy Assumption \ref{boundedgrad}, it will hold in practice under Assumption \ref{compactset}, which was also proposed in \citet{https://doi.org/10.1002/asmb.538}, \citet{gurbuzbalaban2015globally} and \citet{nitanda2016} to analyze the incremental and stochastic methods. We summarize the relation of four assumptions above in Remark \ref{assumprelation}.
\begin{tcolorbox}[frame empty, left = 1mm, right = 1mm, top = 1mm, bottom = 1mm]
\begin{assumption}[Compactness]
\label{compactset}
\hypertarget{compactset}{
There exists a compact set $\mathcal{C}\subseteq\mathbb{R}^d$ containing iterates generated by some optimization algorithm.}
\end{assumption}
\end{tcolorbox}
\begin{tcolorbox}[frame empty, left = 1mm, right = 1mm, top = 1mm, bottom = 1mm]
\begin{remark}
\label{assumprelation}
The relation of Assumption \ref{ERcondition} (ER), Assumption \ref{boundedvar} (BV), Assumption \ref{boundedgrad} (BG) and Assumption \ref{compactset} (Compactness) is illustrated as follows.

\begin{tikzpicture}[node distance=20pt]
  \node[draw]                        (Compactness)   {\hyperlink{compactset}{Compactness}};
  \node[draw, right=of Compactness]                         (Bounded gradient)  {\hyperlink{boundedgrad}{BG}};
  \node[draw, right=of Bounded gradient]                         (Bounded variance)  {\hyperlink{boundedvar}{BV}};
  \node[draw, right=of Bounded variance]                        (ER condition)  {\hyperlink{ERcondition}{ER}};
  
  \draw[->] (Compactness)  -- (Bounded gradient);
  \draw[->] (Bounded gradient) -- (Bounded variance);
  \draw[->] (Bounded variance) -- (ER condition);
\end{tikzpicture}
\end{remark}
\end{tcolorbox}
For QASVRG, we will prove in the next section that the upper bound of the gradient variance introduced by \citet{nitanda2016} also upper bounds the gradient variance of quasar-convex functions (Proposition \ref{varianceupperbound}) provided that each $f_i$ in problem \eqref{problem} is $L_i$-smooth and satisfies the following interpolation assumption. 
\begin{tcolorbox}[frame empty, left = 1mm, right = 1mm, top = 1mm, bottom = 1mm]
\begin{assumption}[Interpolation]
\label{interpolation}
\hypertarget{interpolation}{
    There exists $x^*\in\mathcal{X}^*$ such that for all $i\in[n]$
    \begin{equation*}
        \min_{x\in\mathbb{R}^d}f_i(x)=f_i(x^*).
    \end{equation*}}
\end{assumption}
\end{tcolorbox}
The interpolation assumption is commonly observed in the over-parameterized machine learning models and has  attracted much attention recently \cite{zhou2019sgd, ma2018power, vaswani2019fast, gower2021}. If the model is sufficiently over-parameterized, it can interpolate the labelled training data completely. Particularly, example \eqref{learningLDS} satisfies the interpolation assumption when $\xi_t=0$, as $\Theta$ is also the global minimizer of the objective function generated by each training example. Example \eqref{GLM} also satisfies this assumption given that $y=\sigma(w_*^{\mathsf{T}}x)$ for each $x\sim \mathcal{D}$. Let $L=\max_i\{L_i\}$; we assume throughout that  each $f_i$ is $L$-smooth for brevity. Moreover, \citet{nitanda2016} only presents one parameter choice for both convex and strongly convex functions. In this paper, we provide two parameter choices (Option \rom{1} and \rom{2}) under strong quasar-convexity. Similarly, Option \rom{2} in Table \ref{parameterofQASVRG} is identical to the parameter choice of general quasar-convex functions. Option \rom{1} is a slightly different method from the direct extension of \citet{nitanda2016}'s AMSVRG. More technical comparison 
between these two parameter choices are provided in subsection \ref{convergenceanalysis}.
\section{Algorithms}
\begin{algorithm*}[ht]
\caption{$\left(A_k, B_k, \tilde{y}_0, t, \epsilon\right)$}
\label{algorithm_for_quasar}
\begin{algorithmic}[1]
\REQUIRE{$h$ satisfies $D_h(x,y)\geq\frac{\Bar{\mu}}{2}\|x-y\|^2$; $\tilde{f}\in\mathcal{F}_L$; $f\in\mathcal{Q}_{\mu\gamma}$.}
\STATE{Initialize $x_0=z_0=y_0=\tilde{y}_0$ and specify $\theta_k\triangleq(\nabla_k, \alpha_k, \beta_k, \rho_k, \tilde{f}, b, c, \tilde{\epsilon})$}
\FOR{$k=0,...,t-1$}
\STATE{$\tau_k\gets \text{Bisearch} (\tilde{f},y_k,z_k,b,c,\tilde{\epsilon})$} \quad \quad  ({\em search $\tau_k \in [0,1]$ satisfying $\tau_k\langle\nabla \tilde{f}(d_k),y_k-z_k\rangle-b\|d_k-z_k\|^2\leq c(\tilde{f}(y_k)-$ \\
\qquad\qquad\qquad\qquad\qquad\qquad\qquad\quad$ \tilde{f}(d_k))+\tilde{\epsilon}$ with $d_k:=\tau_k y_k+(1-\tau_k)z_k$; see Algorithm \ref{Bisearch} for more details})
\STATE{$x_{k+1}\gets(1-\tau_k)z_k+\tau_ky_k$} 
\qquad\qquad\qquad\qquad\qquad\qquad\qquad\quad\qquad\qquad\qquad\qquad\qquad\qquad\quad({\em coupling step})
\STATE{$z_{k+1}\gets
    \arg\underset{z\in\mathbb{R}^d}{\min}\left\{\langle\nabla_{k},z-z_k\rangle+\beta_k D_h(z,z_k)+\alpha_k D_h(z,x_{k+1})\right\}$}
\qquad\qquad\qquad\qquad\qquad
\,{(\em mirror descent step)}
\STATE{$y_{k+1}\gets \arg\underset{y\in\mathbb{R}^d}{\min}\left\{\rho_k\langle\nabla_{k},y-x_{k+1}\rangle+\frac{1}{2}\|y-x_{k+1}\|^2\right\}$}\qquad\qquad\qquad\qquad\qquad\qquad\qquad
{(\em gradient descent step)}
\ENDFOR
\OUTPUT $y_t$
\end{algorithmic}
\end{algorithm*}
\begin{algorithm*}[ht]
\caption{$\left(A_k, B_k, b_k, \tilde{y}_0, p, q, \epsilon\right)$}
\label{Multistagesvrg}
\begin{algorithmic}[1]
\REQUIRE{$D_h(x,y)\geq\frac{\Bar{\mu}}{2}\|x-y\|^2$; $f_i\in\mathcal{F}_L$; $f\in\mathcal{Q}_{\mu\gamma}$; $q\leq\frac{1}{4}$; $p\leq\frac{\gamma\bar{\mu}}{16}$}
\STATE{Initialize $y_0=\tilde{y}_0$}
\FOR{s=0,1,...}
\STATE{$
    t\gets\begin{cases}
    \sqrt{\frac{17LD_h(x^*,y_s)}{\gamma^2\bar{\mu}qf(y_s)}},\ \mu=0\ \text{or}\ \mu>0\ (\text{Opt \rom{2}})\\
    \log_{1+\frac{\gamma}{\sqrt{8\kappa}}}\left(\frac{2}{\gamma q}\right),\quad\mu>0\ (\text{Opt \rom{1}})
    \end{cases}\nonumber
$}
\STATE{$y_s\gets\text{Algorithm \ref{algorithm_for_quasar}}\left(A_k,B_k,y_0,\lceil t \rceil,\epsilon\right)$} \qquad\qquad ({\em specify $\nabla_k=\nabla f_{I_k}(x_{k+1})-\nabla f_{I_k}(y_0)+\nabla f(y_0)$ where $|I_k|=b_k$})
\STATE{$y_0\gets y_s$}
\ENDFOR
\OUTPUT $y_s$
\end{algorithmic}
\end{algorithm*}
In order to solve \eqref{problem}, QAGD, QASGD and QASVRG need to access the gradient or the gradient estimate from the oracle, which we denote $\nabla_k$. In this paper, we consider the following gradient (estimates):
\begin{itemize}
    \item \textbf{Full Gradient}: $\nabla_k = \nabla f(x_{k+1})$. Problem \eqref{problem} becomes deterministic.
    \item \textbf{Stochastic Gradient}: $\nabla_k=\nabla f_i(x_{k+1})$ with the index $i$ sampled i.i.d from $[n]$. We have $\mathbb{E}_i[\nabla_k]=\nabla f(x_{k+1}),$ where $\mathbb{E}_i$ denotes the expectation with respect to the index $i$.
    \item \textbf{Mini-batch SVRG}: $\nabla_k=\nabla f_{I_k}(x_{k+1})-\nabla f_{I_k}(\tilde{x})+\nabla f(\tilde{x}),$ where $\tilde{x}$ is the anchor point fixed per stage; $I_k=\{i_1,i_2,...,i_{b_k}\}$ is sampled i.i.d from $[n]$ with $f_{I_k}\triangleq\frac{1}{b_k}\sum_{j=1}^{b_k} f_{i_j}$. Batchsize $|I_k|=b_k$. We have $\mathbb{E}_{I_k}[\nabla_k]=\nabla f(x_{k+1}),$ where $\mathbb{E}_{I_k}$ denotes the expectation with respect to the mini-batch $I_k$.
\end{itemize}
\subsection{Quasar-accelerated Algorithms}
We introduce Algorithm \ref{algorithm_for_quasar} as a general framework incorporating QAGD, QASGD and a single stage of QASVRG with different parameter choices. Based on Algorithm \ref{algorithm_for_quasar}, we also introduce the multi-stage QASVRG as described in Algorithm \ref{Multistagesvrg}. Notably, we provide more information about Bisearch (line 3 of Algorithm \ref{algorithm_for_quasar}) in the Appendix including the whole algorithm and the corresponding complexity analysis obtained from \citet{hinder} (Algorithm \ref{Bisearch}, Lemma \ref{complexityofBi}). The guaranteed performance of our methods relies on the internal assumption \ref{Bregmanassump}. 
 Relation \eqref{st_h} requires $h$ is $\bar{\mu}$-strongly convex; relation \eqref{uniform-quasar-convex-of-f} is a generalization of $\mu$-strongly $\gamma$-quasar-convexity using Bregman divergence as the distance, which we will substitute \eqref{strongly quasar-convex} with in the following analysis. When $h=\frac{1}{2}\|\cdot\|^2$, this relation will be identical to \eqref{strongly quasar-convex}. Relation \eqref{func_value_distance} holds in the Euclidean setting given that $f$ is $\mu$-strongly $\gamma$-quasar-convex with respect to $x^*$ (\citet{hinder}, Corollary 1).
 \begin{tcolorbox}[frame empty, left = 1mm, right = 1mm, top = 1mm, bottom = 1mm]
\begin{assumption}
\label{Bregmanassump}
Suppose $f_i$ is differentiable for each $i\in[n]$. For some $\bar{\mu}>0$, $0<\gamma\leq 1$ and $\mu\geq 0$, and for all $x,y\in\mathbb{R}^d$, require
%\begin{equation}
\begin{align}
\label{st_h}
    &D_h(x,y)\geq\frac{\bar{\mu}}{2}\|x-y\|^2,\\
    \label{uniform-quasar-convex-of-f}
    &f(x^*)\geq f(x) \hspace{-2pt}+\hspace{-2pt} \frac{1}{\gamma}\langle\nabla f(x),x^*\hspace{-2pt}-\hspace{-2pt} x\rangle\hspace{-2pt}+\hspace{-2pt}\mu D_h(x^*,x),\\
    \label{func_value_distance}
    &f(x)\geq f(x^*)+\frac{\gamma\mu}{2-\gamma}D_h(x^*,x).
\end{align}
%\end{equation}
\end{assumption}
\end{tcolorbox}
\subsection{Convergence Analysis}
\label{convergenceanalysis}
We develop a unified analysis of QAGD, QASGD and QASVRG using the following Lyapunov function:
\begin{equation}
\label{lyap}
    E_k\triangleq A_k(f(y_k)-f(x^*))+B_k D_h(x^*,z_k),
\end{equation}
where $A_k$ and $B_k$ are positive non-decreasing sequences that the parameter choices shown in Table \ref{parameterofQAGD}, Table \ref{parameterofQASGD} and Table \ref{parameterofQASVRG} are highly related to. 
Based on the convergence rates derived by Lyapunov analysis, we deduce the complexity upper bound of each method in the Euclidean setting ($h=\frac{1}{2}\|\cdot\|^2$). Since the convergence results of QAGD have already been established in \citet{hinder}, we will not provide the convergence analysis of QAGD in this subsection. Instead, we make convergence analysis of QAGD in Lyapunov framework and obtain results matching \citet{hinder}. Relevant proofs and parameter choices are provided in Appendix~\ref{proofofconvergencerate} and \ref{proofofcomplexity}. 
\begin{theorem}[QASGD]
\label{convergencerateASGD}
    Suppose Assumption \ref{Bregmanassump} and Assumption \ref{boundedgrad} hold, $D_h(x^*,z_0)\leq R^2$, $f_i\in\mathcal{F}_L$ for all $i\in[n]$ and choose any $\Tilde{y}_0\in\mathbb{R}^d$. Then Algorithm \ref{algorithm_for_quasar} with the choices of $\nabla_k=\nabla f_i(x_{k+1})$ and $A_k, B_k, \theta_k$ specified in Table \ref{parameterofQASGD} satisfies 
\iffalse
    \begin{equation}
    \label{QASGDdiff}
        \mathbb{E}[E_{k+1}-E_k]\leq\begin{cases}
 {\sigma^2}/{2\bar{\mu}\beta_k^2}+{\bar{a}_k\epsilon}/{2}\quad & \mu=0,\\
            {\sigma^2B_k}/{\bar{\mu}\beta_k^2}\quad & \mu>0.
        \end{cases}
    \end{equation}
Summing both sides of \eqref{QASGDdiff}, we conclude the convergence rate as follows:
\fi
\iffalse
\begin{equation*}
    \mathbb{E}\left[f(y_t)-f(x^*)\right]\simeq\begin{cases}
    LR^2/t^2+\sigma R/\gamma\sqrt{t}+\epsilon/2\quad&\mu=0,\\
    (1+\mathcal{q})^{-t}E_0+\sigma^2/\gamma^2 t\quad&\mu>0,
\end{cases}\nonumber 
\end{equation*}
\fi
\begin{equation*}
    \mathbb{E}\left[\mathcal{E}_t\right]\simeq\left\{\begin{aligned}
    &\frac{LR^2}{t^2}+\frac{\sigma R}{\gamma\sqrt{t}}+\frac{\epsilon}{2}, &\mu=0,\\
    &\left(1+\min\left\{\frac{\gamma^2\bar{\mu}^2}{16},\frac{1}{2}\right\}\right)^{-t}E_0+\frac{\sigma^2}{\gamma^2 t}, &\mu>0.
\end{aligned}
\right.
\end{equation*}
\end{theorem}
\begin{corollary}
\label{complexASGD}
Consider QASGD under the same assumption in Theorem \ref{convergencerateASGD}. Then  the overall complexity of QASGD to achieve $\mathbb{E}\left[f(y_t)-f(x^*)\right]\leq\epsilon$ is upper bounded by
\begin{equation}
    \begin{aligned}
    &O\left(R\sqrt{\frac{L}{\epsilon}}\log^+\left(\frac{LR^2}{\gamma\epsilon}\right)+ \frac{\sigma^2R^2}{\gamma^2\epsilon^2}\log^+\left(\frac{LR^2}{\gamma\epsilon}\right)\right),\\
    &O\left(\frac{1}{\gamma^2}\log^+\left(\frac{\kappa^{3/4}}{\gamma}\right)\log\left(\frac{\mathcal{E}_0}{\gamma\epsilon}\right) + \frac{\sigma^2}{\gamma^2\epsilon}\log^+\left(\frac{\kappa^{2/3}}{\gamma\epsilon^{1/6}}\right)\right)
    \end{aligned}\nonumber
\end{equation}
for $\mu=0$ and $\mu>0$ respectively.
\end{corollary}
The following proposition shows the variance of the gradient estimate of QASVRG reduces as fast as the objective, which is a key technique in our proof to control the stochastic gradient variance. This proposition is also proposed in \citet{nitanda2016} where they assume $f_i$ is convex and smooth. In this paper, we circumvent the convexity of $f_i$ by using Assumption \ref{interpolation}. The proof of Proposition \ref{varianceupperbound} is postponed to Appendix \ref{proof of Varbound}. 
\begin{proposition}[Variance upper bound]
\label{varianceupperbound}
Suppose Assumption \ref{interpolation} holds, $\nabla_k=\nabla f_{I_k}(x_{k+1})-\nabla f_{I_k}(\tilde{x})+\nabla f(\tilde{x})$ and $f_i\in\mathcal{F}_L$ for each $i\in[n]$, then we obtain the following inequality
\begin{equation}
\begin{aligned}
    &\mathbb{E}_{I_{k}}\|\nabla_k-\nabla f(x_{k+1})\|^2\\
    &\quad\leq4L\frac{n-b_k}{b_k(n-1)}\left(f(x_{k+1})-f(x^*)+f(\tilde{x})-f(x^*)\right),\nonumber
\end{aligned}
\end{equation}
where $|I_k|=b_k$ and $\mathbb{E}_{I_k}$ denotes the expectation with respect to the mini-batch $I_k$.
\end{proposition}
\begin{theorem}[QASVRG (Single-stage)]
\label{convergencerateQASVRGsingle}
    Suppose Assumption \ref{Bregmanassump} and Assumption \ref{interpolation} hold, $D_h(x^*,z_0)\leq R^2$, $f_i\in\mathcal{F}_L$ for each $i\in[n]$ and choose any $\Tilde{y}_0\in\mathbb{R}^d$. Then Algorithm \ref{algorithm_for_quasar} with the choices of $\nabla_k=\nabla f_{I_k}(x_{k+1})-\nabla f_{I_k}(y_0)+\nabla f(y_0)$ and $A_k, B_k, b_k, \theta_k$ specified in Table \ref{parameterofQASVRG} satisfies
\begin{equation*}
         \mathbb{E}\left[\mathcal{E}_t \right]\simeq\left\{\begin{aligned}
         &\frac{LR^2}{\gamma^2t^2}+\left(\frac{p}{\gamma}+\epsilon\right)f(y_0), &\mu=0,\\
    &\left(1+\frac{\gamma}{\sqrt{8\kappa}}\right)^{-t}E_0+\frac{\mathcal{E}_0}{2}, &\mu>0\ ({\text{Opt}\ \rom{1}}),\\
    &\frac{LR^2}{\gamma^2t^2}+\left(\frac{p}{\gamma}+\epsilon\right)f(y_0), &\mu>0\ ({\text{Opt}\ \rom{2}}),
\end{aligned}
\right.
\end{equation*}
where $p\leq\frac{\gamma\bar{\mu}}{16}$ is user specified.
\end{theorem}
\iffalse
\begin{theorem}[QASVRG (Multi-stage)]
\label{convergencerateQASVRGmulti}
    Suppose Assumption \ref{compactset} holds. Under the same setting in Theorem \ref{convergencerateQASVRGsingle}, Algorithm \ref{Multistagesvrg} satisfies
    \begin{equation*}
    \mathbb{E}\left[f(y_s)-f(x^*)\right]\simeq\begin{cases}
    (q+8p/\gamma\bar{\mu}+\epsilon)^{s+1}f(y_0)\quad&\mu=0,\\
    (q+1/2)^{s+1}\mathcal{E}_0\quad&\mu>0.
\end{cases}
\end{equation*}
\end{theorem}
\fi
Under Assumption \ref{compactset}, $\{D_h(x^*,y_s)\}$ can be uniformly bounded by some constant if $\{y_s\}$ generated by Algorithm \ref{Multistagesvrg} are restricted to a compact set. Thus we can hide the Bregman divergence inside $O(\cdot)$ and $\widetilde{O}(\cdot)$. Note that we only need this assumption when $\mu=0$. According to Theorem \ref{convergencerateQASVRGsingle}, single-stage QASVRG is biased which means that an $\epsilon$-approximate solution can not be generated via single-stage QASVRG. For instance, the bias is upper bounded by $O\left(\left(\frac{p}{\gamma}+\epsilon\right)f(y_0)\right)$ when $\mu=0$, but the expectation of the optimality gap $\mathcal{E}_t$ can shrink at each stage with small $p$ and $q$ when $t=\Omega\left(\sqrt{\frac{LR^2}{\gamma^2f(y_0)}}\right)$. Consequently we need $O(\log(1/\epsilon))$ stages to generate an $\epsilon$-approximate solution.
\begin{corollary}
\label{complexASVRG}
Under Assumption \ref{compactset} and the same assumptions in Theorem \ref{convergencerateQASVRGsingle}, the overall complexity required for Algorithm \ref{Multistagesvrg} to achieve $\mathbb{E}\left[f(y_{s})-f(x^*)\right]\leq\epsilon$ is upper bounded by 
\begin{equation*}
    \begin{aligned}
    &O\left(\hspace{-.1cm}\left(n+\frac{nLR^2}{\gamma\epsilon n+\gamma\sqrt{\epsilon LR^2}}\log^+\left(\frac{L^{1/2}R}{q^{1/2}\gamma\epsilon^{9/14}}\right)\hspace{-.1cm}\right)\log\left(\frac{1}{\epsilon}\right)\hspace{-.1cm}\right),\\
    &O\left(\hspace{-.1cm}\left(n\hspace{-.05cm}+\frac{n{\kappa}}{\gamma^2 n+\gamma\sqrt{\kappa}}\log\left(\frac{2}{\gamma q}\right)\hspace{-.05cm}\log^+\hspace{-.1cm}\left(\frac{\kappa^{7/6}}{\gamma}\right)\hspace{-.1cm}\right)\log\left(\frac{1}{\epsilon}\right)\hspace{-.1cm}\right),\\
    &O\left(\hspace{-.1cm}\left(n+\frac{n\kappa}{\gamma^2n+\gamma^{3/2}\sqrt{\kappa}}\log^+\left(\frac{L^{1/2}R}{q^{1/2}\gamma\epsilon^{9/14}}\right)\hspace{-.1cm}\right)\log\left(\frac{1}{\epsilon}\right)\hspace{-.1cm}\right)
    \end{aligned}
\end{equation*}
for $\mu=0$ and $\mu>0$ (the last two bounds) respectively, where $q\in(0,1/4]$.
\end{corollary}
\paragraph{Proof Sketch}
The method we use to derive the convergence rates and complexity for each algorithm is unified. We  take the difference between each Lyapunov stage, and then take the conditional expectation to obtain the upper bound on $\mathbb{E}[E_{k+1}-E_k]$. Finally we sum over $k$ to conclude the convergence rates and a subsequent iteration complexity for each algorithm. Combining this complexity with that of Bisearch, we conclude the overall complexity in each corollary. For more details, see Appendix~\ref{proofofconvergencerate} and \ref{proofofcomplexity}.

The last two bounds of QASVRG correspond to two different choices of parameters in Table \ref{parameterofQASVRG} (Option \rom{1} and \rom{2}). The complexity bound derived with Option \rom{1} has a more unfavorable dependency on $\kappa$ while the complexity bound derived with Option \rom{2} has a more unfavorable dependency on $\epsilon$. This suggests Option \rom{1} performs better on well-conditioned problems e.g. $\kappa\epsilon<1$ and Option \rom{2} performs better on ill-conditioned problems e.g. $\kappa\epsilon\gg 1$.

In the complexity bounds of QASGD and QASVRG, extra logarithmic factors are included, which comes from Bisearch. \citet{hinder} prove that the complexity of Bisearch is at most a logarithmic factor given that the function involved in this subroutine is $L$-smooth. In the stochastic setting, where the functions involved are single $f_i$ or a mini-batch of $f_{I_k}$, we need to assume $f_i\in\mathcal{F}_L$ for all $i$ due to the uniform sampling.

We summarize our methods and some existing methods in Table \ref{summary}, including their corresponding assumptions and complexity upper bounds. To summarize, both QASGD and QASVRG achieve better complexity upper bounds than QAGD when $n$ is large, and QASGD enjoys a faster convergence rate and lower complexity than SGD under a stronger assumption. While QASVRG has the potential to be more computationally expensive than SGD due to the full gradient and function value access once a stage, it enjoys a theoretically faster convergence rate than SGD.
\begin{figure*}[t]
\centering 
\subfigure{
\label{Dataset1.1}
\includegraphics[width=5.2cm,height =3.2cm]%[width=4.2cm,height = 2.8cm]
{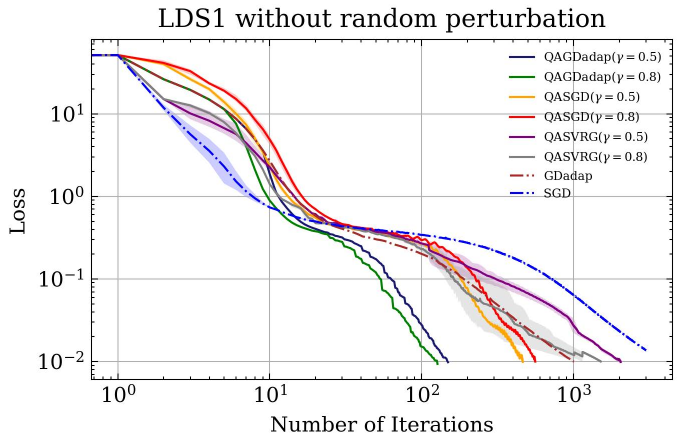}}\subfigure{
\label{Dataset1.2}
\includegraphics[width=5.2cm,height =3.2cm]{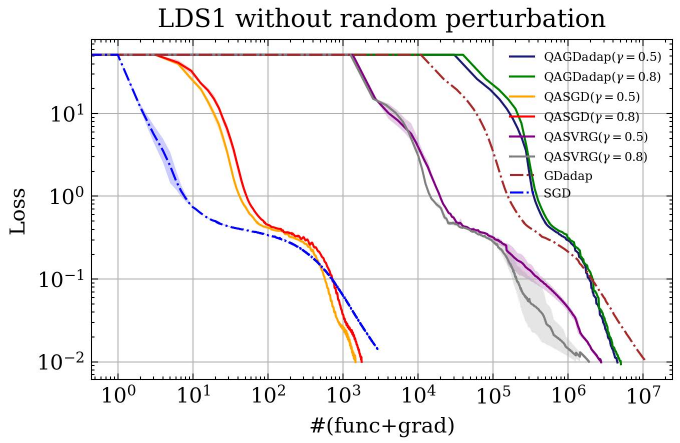}}
\subfigure{
\label{Dataset1.3}
\includegraphics[width=5.2cm,height =3.2cm]{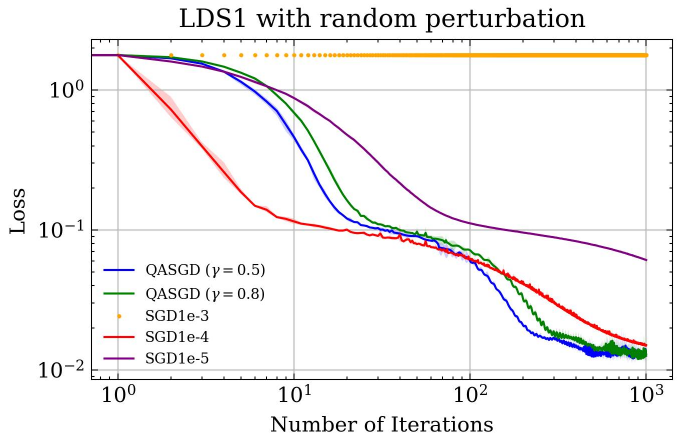}}
\subfigure{
\label{Dataset2.1}
\includegraphics[width=5.2cm,height =3.2cm]{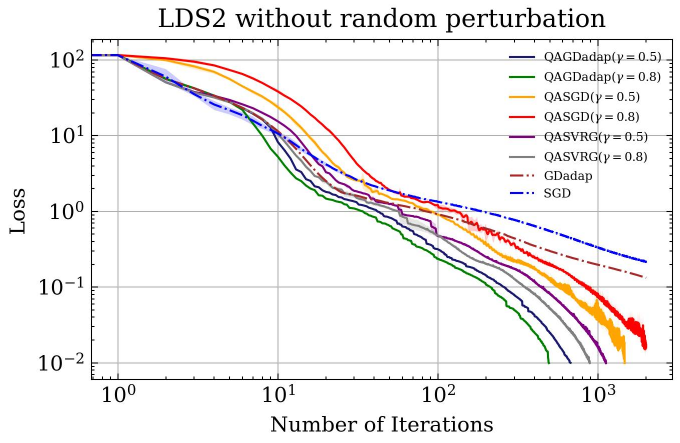}}\subfigure{
\label{Dataset2.2}
\includegraphics[width=5.2cm,height =3.2cm]{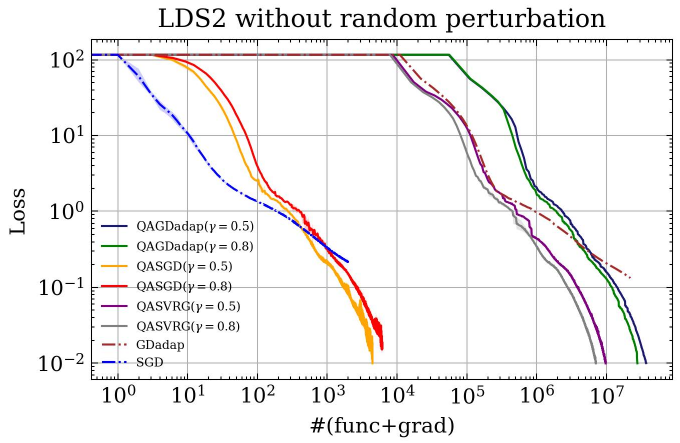}}
\subfigure{
\label{Dataset2.3}
\includegraphics[width=5.2cm,height =3.2cm]{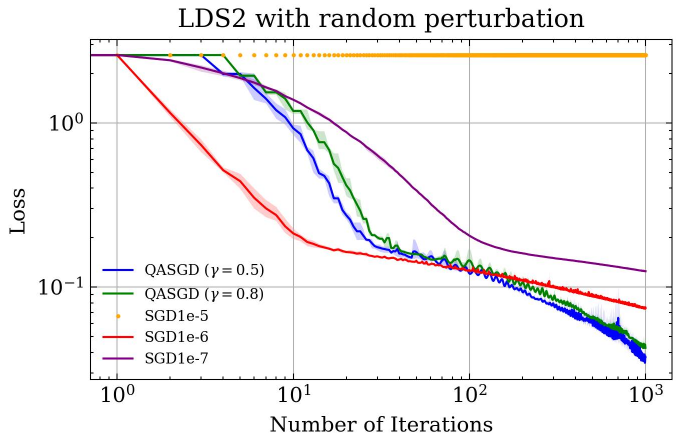}}
\subfigure{
\label{Dataset3.1}
\includegraphics[width=5.2cm,height =3.2cm]{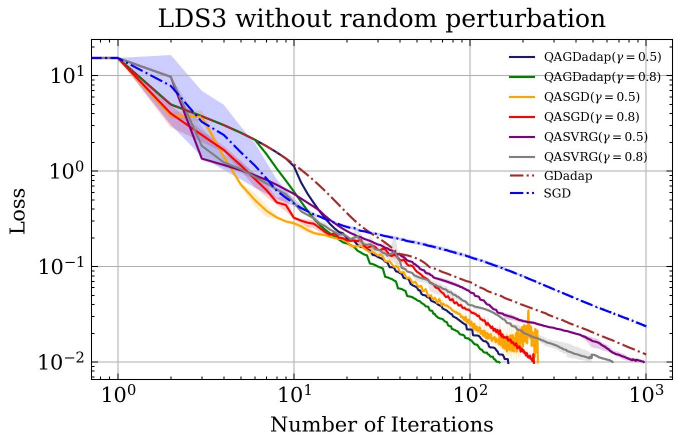}}\subfigure{
\label{Dataset3.2}
\includegraphics[width=5.2cm,height =3.2cm]{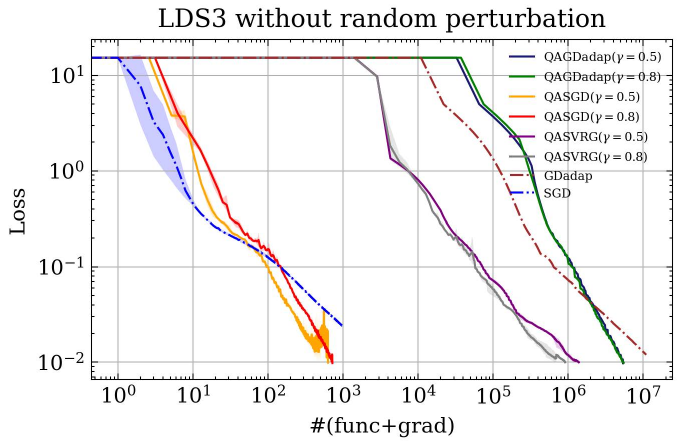}}
\subfigure{
\label{Dataset3.3}
\includegraphics[width=5.2cm,height =3.2cm]{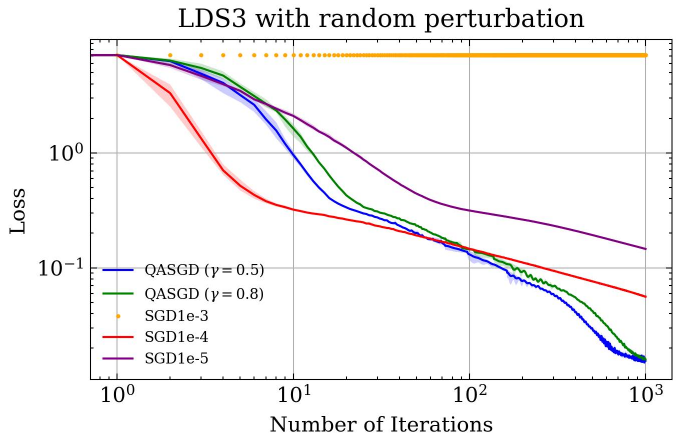}}
\caption{Evaluation on three different LDS instances. We choose $\epsilon=10^{-2}$, the stepsize to be $5\times10^{-5},1\times10^{-6},1\times10^{-4}$ for SGD, $L=1\times 10^6,1\times 10^8, 1\times 10^5$ for QASGD and $L=3\times 10^4,1\times 10^6,1\times 10^4$ for QASVRG in LDS1, LDS2 and LDS3. The flat line in the third column means the loss blows up to infinity with this choice of stepsize.}
\label{LDS}
\end{figure*}
\section{Simulations}
\label{simulation}
In this section, we evaluate our methods on example \eqref{learningLDS} in the Euclidean setting using synthetic dataset where each input sequence $x^{(i)}\sim\mathcal{N}(0,1)$ coordinate-wise. Different from \citet{hardt}, we generate $N$ training examples and random perturbations $\xi_t$ before training instead of generating fresh data and random perturbations at each iteration. Thus example \eqref{learningLDS} can be reformulated as
\begin{equation*}
    \min\left\{F(\hat{\Theta})=\frac{1}{N}\sum_{i=1}^N\left[\frac{1}{T}\sum_{t=1}^T\left\|\hat{y}_t^{(i)}-y_t^{(i)}\right\|^2\right]\right\},
\end{equation*}
where the superscript $(i)$ represents that the output is generated using $i^{\text{th}}$ training data $(x^{(i)},y^{(i)})$. Similar to \citet{hardt}, the actual objective in our experiments is $F(\hat{\Theta})=\frac{1}{N}\sum_{i=1}^N\left[\frac{1}{T-T_1}\sum_{t>T_1}\|\hat{y}_t^{(i)}-y_t^{(i)}\|^2\right]$, where $T_1=T/4$.
We generate the true dynamical system and data the same way as in \citet{hardt} using parameters $N=5000,d=20,T=500$. Following \citet{hinder}, we generate the initial iterate $(\hat{A}_0,\hat{C}_0,\hat{D}_0)$ by perturbing the parameters of the true system and keep the spectral radius of $\hat{A}_0$ strictly less than 1. We choose the value of random seed to be in $\{0,12,24,36,48\}$ for generating five true LDS instances and their initialization. We only present simulation results of $\{0,24,48\}$ in this sections and the remaining results are provided in section \ref{additionalsimul}. Note that $\hat{B}$ is not a trainable parameter since $B$ is known. As is described in \citet{hardt}, it is intractable to calculate the precise value of quasar-convexity parameter $\gamma$ of LDS objective or even estimate it. Thus we evaluate our methods with $\gamma\in\{0.5,0.8\}$. We observe that the imprecise $\gamma$ does not affect the performance of all methods involved. As is shown in Table \ref{parameterofQASGD}, $\eta$ is involved in the parameter choice of QASGD, which we choose to be $\min\left\{\frac{1}{L}, \frac{\sqrt{2}\gamma\|z_0\|}{\sigma(t+1)^{3/2}}\right\}$ in our simulations. While $\sigma$ in Assumption \ref{boundedgrad} relies on the compact set and initialization, we observe that choosing $\sigma=1$ is a robust choice for all of our simulations by evaluating QASGD on three LDS instances with $\sigma\in\{1,10,10^2,10^3\}$ (See Figure \ref{LDSadditionalsigma} in Appendix \ref{additionalsimul} for more details). According to the analysis in \citet{hardt}, $F(\hat{\Theta})$ is $L$-weakly smooth, and it is still unknown whether $F(\hat{\Theta})$ is $L$-smooth. Given that the parameter choice of our methods involves $L$, we fine-tune the value of $L$ for QASGD and QASVRG and choose the best stepsize for SGD by extensive grid search in each instance. We use the adaptive stepsize for QAGD and GD the same way as in \citet{hinder}. We consider the random noise $\xi_t\sim\mathcal{N}(0,10^{-2})$ or $\xi_t=0$ perturbing the output of the true systems. If $\xi_t\sim\mathcal{N}(0,10^{-2})$, the interpolation assumption will be violated since $\Theta$ is no longer the global minimizer of the objective generated by each training example. Thus we only evaluate SGD and QASGD in this case. In 
Algorithm \ref{Multistagesvrg}, it may be difficult to calculate $t$ especially when $L$ is unknown. Therefore we can specify $t$ to be relatively large (we choose $t=10^4$) and use an appropriate restart scheme in Algorithm \ref{algorithm_for_quasar} to boost the performance of QASVRG. Following \citet{nitanda2016}, when the relation $\langle\nabla_k,y_{k+1}-y_k\rangle>0$ holds, we break Algorithm \ref{algorithm_for_quasar} to return $y_s$ and start the next stage. Since we generate the initial iterate with $\rho(\hat{A}_0)<1$, we don't use gradient clipping or projection proposed in \citet{hardt} during training. We generate the error bar in Figure \ref{LDS} by averaging the results obtained from running each stochastic algorithm three times and choose the maximum and minimum value pointwise to be the upper bar and lower bar.

The simulation results in Figure \ref{LDS} validate our methods and show the superiority of our methods in terms of the convergence speed and the overall complexity. In addition, both QASGD and QASVRG are robust to the random sampling of the stochastic gradient. Code is available at \href{https://github.com/QiangFu09/Stochastic-quasar-convex-acceleration}{https://github.com/QiangFu09/Stochastic-quasar-convex-acceleration}. There is an interesting phenomenon in our simulations. While the convergence rate of QASGD matches the rate of SGD when $t$ is large, Figure \ref{LDS} still shows the substantial superiority of QASGD over QASVRG. In fact, QASGD enjoys a convergence rate of $O\left(\frac{1}{t^2}+\frac{1}{\sqrt{t}}+\frac{\epsilon}{2}\right)$ indicating rapid initial phase where $O\left(\frac{1}{t^2}\right)$ dominates the convergence. We speculate that QASGD in most of our simulations does not escape the initial phase and thus enjoys a fast convergence. The simulation results above lead to a reconsideration about whether we need the bounded gradient assumption in QASGD or not, as the objective of LDS does not satisfy this assumption but the performance of QASGD on LDS is still favorable and robust. We believe that this assumption is used in the theoretical analysis but may not be necessary.
\section{Disscussion}
In this paper, we propose QASGD and QASVRG achieving fast convergence and low complexity under their corresponding assumptions. We present our algorithms in a unified framework using a single energy-based analysis to establish the convergence rates and complexity of QAGD, QASGD and QASVRG. We close with a brief discussion of some possible future work.

First, we introduced the bounded gradient assumption for QASGD, but it remains to be seen whether we can weaken this assumption to some extent. Given that \citet{gower2021} and \citet{jin2020convergence} establish the convergence of SGD for quasar-convex functions under the ER condition and bounded variance assumption respectively,
it would be of interest to see whether it is possible to apply these weaker assumptions in QASGD.

Second, QAGD are proven near-optimal in \citet{hinder} based on the lower bound they establish for first-order deterministic methods. In future work we hope to establish a worst case complexity lower bound for first-order stochastic methods applied to quasar-convex functions under different assumptions. We expect such bounds will prove that our methods are nearly optimal as well.

Moreover, a higher order method usually leads to a better convergence rate. \citet{nesterov2006cubic} propose the cubic regularized Newton method to optimize star-convex functions ($\gamma=1$). Therefore, we are also interested in whether it is possible to use higher order methods to improve the convergence of our methods.

Finally, we hope to exploit more applications of quasar-convex functions in machine learning.

% Acknowledgements should only appear in the accepted version.

\section*{Acknowledgements}
We would like to thank our anonymous reviewers for constructive comments, thank Kyurae Kim (University of Pennsylvania) for helpful discussions on citations and thank Oliver Hinder (University of Pittsburgh) and Nimit Sohoni (Stanford University) for sharing codes of LDS.

% In the unusual situation where you want a paper to appear in the

\bibliography{ref}
\bibliographystyle{icml2023}

%%%%%%%%%%%%%%%%%%%%%%%%%%%%%%%%%%%%%%%%%%%%%%%%%%%%%%%%%%%%%%%%%%%%%%%%%%%%%%%
%%%%%%%%%%%%%%%%%%%%%%%%%%%%%%%%%%%%%%%%%%%%%%%%%%%%%%%%%%%%%%%%%%%%%%%%%%%%%%%
% APPENDIX
%%%%%%%%%%%%%%%%%%%%%%%%%%%%%%%%%%%%%%%%%%%%%%%%%%%%%%%%%%%%%%%%%%%%%%%%%%%%%%%
%%%%%%%%%%%%%%%%%%%%%%%%%%%%%%%%%%%%%%%%%%%%%%%%%%%%%%%%%%%%%%%%%%%%%%%%%%%%%%%
\newpage
\appendix
\onecolumn
\section{Helpful Lemmas and Assumptions}
\begin{lemma}[Three-point identity]
\label{Three-point identity}
For all $x\in \text{dom }h$ and $y,z\in \text{int}(\text{dom }h)$
\begin{equation}
\label{TPI}
    D_h(x,y)-D_h(x,z)=-\langle\nabla h(y)-\nabla h(z),x-y\rangle-D_h(y,z),
\end{equation}
where $D_h(x,y)=h(x)-h(y)-\langle\nabla h(y),x-y\rangle$.
\end{lemma}
\begin{lemma}[Fenchel-Young inequality]
For all $x,y\in\mathbb{R}^d$ and $p\neq 0$, we have
\begin{equation}
\label{fenchelyoung}
    \langle x,y \rangle +\frac{1}{p}\|y\|^p \geq-\frac{p-1}{p}\|x\|_*^{\frac{p}{p-1}},
\end{equation}
where $\|\cdot\|_*$ denotes the conjugate norm of $\|\cdot\|$.
\end{lemma}
\begin{lemma}[\citet{hinder}, Lemma 2]
\label{Existenceofalpha}
Let $f:\mathbb{R}^d\rightarrow\mathbb{R}$ be differentiable and let $y,z\in\mathbb{R}^d$. For $\tau\in\mathbb{R}$ define $x_{\tau}\triangleq\tau y+(1-\tau)z$. For any $c\geq 0$ there exists $\tau\in [0,1]$ such that
\begin{equation}
    \tau\langle \nabla f(x_{\tau}),y-z\rangle \leq c\left(f(y)-f(x_{\tau})\right).
    \label{EOGa}
\end{equation}
\end{lemma}
\iffalse
\begin{proof}
Define $g(\tau)\triangleq f(x_{\tau})$. Then for any $\tau\in\mathbb{R}$ we have $g'(\tau)=\langle\nabla f(x_{\tau}),y-z\rangle$. Consequently, \eqref{EOGa} is equivalent to the condition
\begin{equation}
    \tau g'(\tau)\leq c[g(1)-g(\tau)].
    \label{Equivalentcond}
\end{equation}
If $g'(1)< 0$, inequality \eqref{Equivalentcond} trivially holds at $\tau=1$;
If $g(0)< g(1)$, inequality \eqref{Equivalentcond} trivially holds at $\tau=0$;
If $g'(1)\geq0$ and $g(0)\geq g(1)$, there exists an $\tau\in (0,1)$ such that $g'(\alpha)=0$ and $g(\tau)\leq g(1)$. Thus, inequality \eqref{Equivalentcond} holds.
\end{proof}
\fi
In fact, we can slacken condition \eqref{EOGa} to some extent, and we can find an appropriate momentum parameter $\tau$ satisfying
\begin{equation}
    \tau\langle\nabla f(x_{\tau}),y-z\rangle-\tau^2b\|y-z\|^2\leq c\left(f(y)-f(x_{\tau})\right)+\tilde{\epsilon},
    \label{slackEOGa}
\end{equation}
for $b,c,\tilde{\epsilon}\geq 0$. Existence of $\tau$ satisfying condition \eqref{EOGa} implies the existence of $\tau$ satisfying condition \eqref{slackEOGa}.
\begin{lemma}
    If $f:\mathbb{R}^d\rightarrow\mathbb{R}$ is $L$-smooth, then for any $x,y\in\mathbb{R}^d$
    \begin{equation}
    \label{lsmooth}
        f(y)\leq f(x)+\langle\nabla f(x),y-x\rangle+\frac{L}{2}\|y-x\|^2.
    \end{equation}
\end{lemma}
Now We introduce an algorithm to effectively search the "good" $\tau$ for any $L$-smooth functions.
\begin{lemma}[\citet{hinder}, C.2]
\label{complexityofBi}
For $L$-smooth $f:\mathbb{R}^d\rightarrow\mathbb{R}$, $x,v\in\mathbb{R}^n$ and scalars $b,c,\tilde{\epsilon}\geq0$, Algorithm \eqref{Bisearch} computes $\alpha\in[0,1]$ satisfying \eqref{slackEOGa} with at most
$$6+3\left\lceil\log_2^+\left((4+c)\min\left\{\frac{2L^3}{b^3},\frac{L\|y-z\|^2}{2\tilde{\epsilon}}\right\}\right)\right\rceil$$
function and gradient evaluations, where $\log^+(\cdot)\overset{\Delta}{=}\max\{\log(\cdot),1\}$.
\end{lemma}
\begin{algorithm}[H]
\caption{Bisearch($f,y,z,b,c,\tilde{\epsilon}$,[guess])}
\label{Bisearch}
\begin{algorithmic}[1]
\REQUIRE{$f$ is $L$-smooth; $z,y\in\mathbb{R}^d$; $c\geq 0$; $\text{"guess"}\in[0,1]$ if provided. "guess" can be the momentum parameter under convexity or other.}
\STATE{Define $g(\alpha)\overset{\Delta}{=}f(\tau y+(1-\tau)z)$ and $p\overset{\Delta}{=}b\|z-y\|^2$}
\STATE{\textbf{if} guess provided \textbf{and} $cg(\text{guess})+\text{guess}\cdot (g'(\text{guess})-\text{guess}\cdot p)\leq cg(1)$ \textbf{then return} guess}
\STATE{\textbf{if} $g'(1)\leq p+\tilde{\epsilon}$ \textbf{then return} 1;}
\STATE{\textbf{else if} $c=0$ or $g(0)\leq g(1)+\tilde{\epsilon}/c$ \textbf{then return} 0;}
\STATE{$\delta\gets1-g'(1)/L\|z-y\|^2$}
\STATE{\textbf{lo}$\gets0$, \textbf{hi}$\gets\delta$, $\tau\gets\delta$}
\WHILE{$cg(\tau)+\tau( g'(\tau)-\tau p)>cg(1)+\tilde{\epsilon}$}
\STATE{$\tau\gets(\textbf{lo}+\textbf{hi})/2$}
\STATE{\textbf{if} $g(\tau)\leq g(\delta)$ \textbf{then} $\textbf{hi}\gets\tau$}
\STATE{\textbf{else} \textbf{lo}$\gets\tau$}
\ENDWHILE
\OUTPUT $\tau$
\end{algorithmic}
\end{algorithm}
%%%%%%%%%%%%%%%%%%%%%%%%%%%%%%%%%%%%%%%%%%%%%%%%%%%%%%%%%%%%%%%%%%%%%%%%%%%%%%%
%%%%%%%%%%%%%%%%%%%%%%%%%%%%%%%%%%%%%%%%%%%%%%%%%%%%%%%%%%%%%%%%%%%%%%%%%%%%%%%
\begin{proposition}
\label{lowerboundofkappa}
Based on Assumption \ref{Bregmanassump}, we have $\kappa\geq\frac{\gamma}{2-\gamma}$, where $\kappa=\frac{L}{\bar{\mu}\mu}$ and $\mu>0$.
\end{proposition}
\begin{proof}
\begin{equation*}
\frac{\gamma\mu}{2-\gamma}D_h(x^*,x)\leq f(x)-f(x^*)\leq\frac{L}{2}\|x^*-x\|^2\leq\frac{L}{\bar{\mu}}D_h(x^*,x)
\end{equation*}
Thus, we have $\frac{\gamma\mu}{2-\gamma}\leq\frac{L}{\bar{\mu}}$. Furthermore, we have $\sqrt{\kappa}\geq\sqrt{\frac{\gamma}{2-\gamma}}\geq\sqrt{\gamma^2}=\gamma$.
\end{proof}

\begin{lemma}
\label{varupperbound}
Suppose Assumption \ref{interpolation} holds, and each $f_i\in\mathcal{F}_L$. If $\nabla_k=\nabla f_i(x_{k+1})-\nabla f_i(y_{k})+\frac{1}{n}\sum_{i=1}^nf_i(y_k)$, we can obtain
\begin{equation}
    \mathbb{E}_i\left[\|\nabla_k-\nabla f(x_{k+1})\|^2\right]\leq 4L\left(f(x_{k+1})-f(x^*)+f(y_k)-f(x^*)\right)
\end{equation}
\end{lemma}
\begin{proof}
Since $f$ is $L$-smooth, for any $x,y\in\mathbb{R}^d$ we have
\begin{equation}
    f(x)\leq f(y)+\langle\nabla f(y),x-y\rangle+\frac{L}{2}\|x-y\|^2\triangleq g(x)\nonumber
\end{equation}
Let $\nabla g(\tilde{x})=0$, and $\tilde{x}=y-\frac{1}{L}\nabla f(y)$ is the minimizer of $g(x)$. And we have
\begin{equation}
    f(x^*)\leq f(x)\leq g(\tilde{x})= f(y)-\frac{1}{2L}\|\nabla f(y)\|^2\Rightarrow \|\nabla f(y)\|^2\leq 2L(f(y)-f(x^*))\nonumber
\end{equation}
Since $\nabla f(x^*)=0$, we have
\begin{equation}
    \|\nabla f(y)-\nabla f(x^*)\|\leq 2L(f(y)-f(x^*))
    \label{varbound}
\end{equation}
By \eqref{varbound}, we can upper bound $\mathbb{E}\left[\|\nabla_k-\nabla f(x_{k+1})\|^2\right]$ as follows
\begin{equation*}
    \begin{aligned}
    \mathbb{E}_i\left[\|\nabla_k-\nabla f(x_{k+1})\|^2\right]&=\mathbb{E}_i\left[\|\nabla f_i(x_{k+1})-\nabla f_i(y_{k})-\mathbb{E}_i[\nabla f_i(x_{k+1})-\nabla f_i(y_{k})]\|^2\right]\\
    &\leq\mathbb{E}_i\left[\|\nabla f_i(x_{k+1})-\nabla f_i(y_{k})\|^2\right]\\
    &=\mathbb{E}_i\left[\|\nabla f_i(x_{k+1})-\nabla f(x^*)+\nabla f(x^*)-\nabla f_i(y_{k})\|^2\right]\\
    &\leq2\mathbb{E}_i\left[\|\nabla f_i(x_{k+1})-\nabla f(x^*)\|^2\right]+2\mathbb{E}_i\left[\|\nabla f(x^*)-\nabla f_i(y_{k})\|^2\right]\\
    &\leq4L\mathbb{E}_i[f_i(x_{k+1})-f_i(x^*)]+4L\mathbb{E}_i[f_i(y_{k})-f_i(x^*)]\\
    &=4L\left(f(x_{k+1})-f(x^*)+f(y_k)-f(x^*)\right),
    \end{aligned}
\end{equation*}
where the first inequality uses the relation $\mathbb{E}\left[\|X-\mathbb{E}[X]\|^2\right]\leq\mathbb{E}\left[\|X\|^2\right]$ for all random variable $X$; the second inequaity uses the relation $\|a+b\|^2\leq 2\|a\|^2+2\|b\|^2$.
\end{proof}
\begin{lemma}
\label{Exptrans}
Let $\{\xi_i\}_{i=1}^n$ be a set of vectors in $\mathbb{R}^d$ and $\bar{\xi}$ denote an average of $\{\xi_i\}_{i=1}^n$. Let I denote a uniform random variable representing a subset of $\{1,2,...,n\}$ with its size equal to $b$. Then, it follows that,
\begin{equation}
    \mathbb{E}_I\left\|\frac{1}{b}\underset{i\in I}{\sum}\xi_i-\bar{\xi}\right\|^2=\frac{n-b}{b(n-1)}\mathbb{E}_i\|\xi_i-\bar{\xi}\|^2.\nonumber
\end{equation}
\end{lemma}
We orient our readers to the supplementary materials of \citet{nitanda2016} for proof details of the above lemma.
\section{Proof of Proposition \ref{varianceupperbound}}
\label{proof of Varbound}
Based on Lemma \ref{varupperbound} and Lemma \ref{Exptrans}, we prove \ref{varianceupperbound}.
\begin{proof}
Let $\nabla_k^i=\nabla f_i(x_{k+1})-\nabla f_i(\tilde{x})+\nabla f(\tilde{x})$ and $\nabla_k=\frac{1}{b_k}\underset{i\in I_k}{\sum}\nabla_k^i$. Using Lemma \eqref{Exptrans}, we have
\begin{equation}
\label{varupperboundproof}
    \mathbb{E}_{I_{k}}\|\nabla_k-\nabla f(x_{k+1})\|^2=\frac{1}{b}\frac{n-b_k}{n-1}\mathbb{E}_i\left[\|\nabla_k^i-\nabla f(x_{k+1})\|^2\right]\leq4L\frac{n-b_k}{b_k(n-1)}\left(f(x_{k+1})-f(x^*)+f(\tilde{x})-f(x^*)\right),
\end{equation}
where the inequality follows from Lemma \eqref{varupperbound}.
\end{proof}
\section{Proofs of Convergence Rates}
\label{proofofconvergencerate}
\subsection{Proof of QAGD}
\label{proofQAGD}
\begin{table}[H]
\begin{center}
\renewcommand\arraystretch{1.5}
    \begin{tabular}{|c|}
    \hline
    \rowcolor{black!20}
    \textbf{QAGD}\\
    \hline
    \rowcolor{black!10}
    \textbf{$\gamma$-quasar-convex} $(\mu=0)$\\
    \hline
    $A_k=\frac{\bar{\mu}\gamma^2}{4L}(k+1)^2$, $B_k=1$\\
    $(\alpha_k, \beta_k, \rho_k, \tilde{f},b,c,\tilde{\epsilon})\gets\left(0,\frac{\gamma}{\bar{a}_k},\frac{1}{L}, f, 0, \frac{\gamma A_k}{\bar{a}_k}, \frac{\gamma\epsilon}{2}\right)$\\
    \hline
    \rowcolor{black!10}
    \textbf{$\mu$-strongly $\gamma$-quasar-convex} $(\mu>0)$\\
    \hline
    $A_k=(1+{\gamma}/2{\sqrt{\kappa}})^{k}$, $B_k=\mu A_k$\\
    $(\alpha_k, \beta_k, \rho_k, \tilde{f},b,c,\tilde{\epsilon})\gets\left(\gamma\mu, \frac{\gamma\mu B_k}{\bar{b}_k}, \frac{1}{L}, f, \frac{\gamma\bar{\mu}\mu}{2}, \frac{\gamma A_k}{\bar{a}_k},0\right)$\\
    \hline
    \end{tabular}
    \caption{Parameter choices for QAGD}
    \label{parameterofQAGD}
    \end{center}
\end{table}
we begin with Algorithm \ref{algorithm_for_quasar}, $\nabla_k=\nabla f(x_{k+1})$ and parameters specified in Table \ref{parameterofQAGD} using Lyapunov function \eqref{lyap}:

\textbf{Case 1:} $\mu=0$
For $\nabla_k=\nabla f(x_{k+1})$, we begin with Algorithm \ref{algorithm_for_quasar} using Lyapunov function \eqref{lyap}:
\begin{equation}
    \begin{aligned}
    E_{k+1}-E_k
    &\overset{\eqref{TPI}}{=}-\left\langle{\nabla h(z_{k+1})-\nabla h(z_k)},x^*-z_{k+1}\right\rangle-D_h(z_{k+1},z_k)+A_{k+1}(f(y_{k+1})-f(x^*))-A_k(f(y_k)-f(x^*))\\
    &=\frac{1}{\beta_k}\left\langle\nabla f(x_{k+1}),x^*-z_{k+1}
    \right\rangle-D_h(z_{k+1},z_k)+{A_{k+1}}(f(y_{k+1})-f(x^*))-{A_k}(f(y_k)-f(x^*))\\
    &\overset{\eqref{st_h}}{\leq}\frac{1}{\beta_k}\left\langle\nabla f(x_{k+1}),x^*-z_{k}
    \right\rangle+\frac{1}{\beta_k}\left\langle\nabla f(x_{k+1}),z_k-z_{k+1}
    \right\rangle-\frac{\bar{\mu}}{2}\|z_{k+1}-z_k\|^2\\
    &\quad+{A_{k+1}}(f(y_{k+1})-f(x^*))-{A_k}(f(y_k)-f(x^*))\\
    &\overset{\eqref{fenchelyoung}}{\leq}\frac{1}{\beta_k}\left\langle\nabla f(x_{k+1}),x^*-z_{k}
    \right\rangle+\frac{1}{2\bar{\mu}\beta_k^2}\|\nabla f(x_{k+1})\|^2+{A_{k+1}}(f(y_{k+1})-f(x^*))-{A_k}(f(y_k)-f(x^*))\\
    &=\frac{1}{\beta_k}\langle\nabla f(x_{k+1}),x^*-x_{k+1}\rangle+\frac{\tau_k}{\beta_k}\langle\nabla f(x_{k+1}),y_k-z_k\rangle+\frac{1}{2\bar{\mu}\beta_k^2}\|\nabla f(x_{k+1})\|^2+A_{k+1}(f(y_{k+1})-f(x_{k+1}))\\
    &\quad+(A_{k+1}-A_k)(f(x_{k+1})-f(x^*))+A_k(f(x_{k+1})-f(y_k))\\
    &\leq\frac{\gamma}{\beta_k}(f(x^*)-f(x_{k+1}))+\frac{1}{\beta_k}(c(f(y_k)-f(x_{k+1}))+\tilde{\epsilon})+\frac{1}{2\bar{\mu}\beta_k^2}\|\nabla f(x_{k+1})\|^2+A_{k+1}(f(y_{k+1})-f(x_{k+1}))\\
    &\quad+(A_{k+1}-A_k)(f(x_{k+1})-f(x^*))+A_k(f(x_{k+1})-f(y_k))\nonumber
    \end{aligned}
\end{equation}
The second inequality follows from the Fenchel-Young inequality, and the last inequality follows from the quasar-convexity of $f$ and \eqref{slackEOGa} (Lemma \ref{Existenceofalpha}). With the choice of parameter summarized in Table \ref{parameterofQAGD}, we obtain the following bound:
\begin{equation}
    \begin{aligned}
    E_{k+1}-E_k &\leq\frac{\gamma}{\beta_k}(f(x^*)-f(x_{k+1}))+\frac{1}{\beta_k}(c(f(y_k)-f(x_{k+1}))+\tilde{\epsilon})+\frac{1}{2\bar{\mu}\beta_k^2}\|\nabla f(x_{k+1})\|^2+A_{k+1}(f(y_{k+1})-f(x_{k+1}))\\
    &\quad+(A_{k+1}-A_k)(f(x_{k+1})-f(x^*))+A_k(f(x_{k+1})-f(y_k))\\
    &=\frac{1}{2\bar{\mu}\beta_k^2}\|\nabla f(x_{k+1})\|^2+A_{k+1}(f(y_{k+1})-f(x_{k+1}))+\frac{(A_{k+1}-A_k)}{2}\epsilon\\
    &\overset{\eqref{lsmooth}}{\leq}\left(\frac{1}{2\bar{\mu}\beta_k^2}-\frac{A_{k+1}}{2L}\right)\|\nabla f(x_{k+1})\|^2+\frac{(A_{k+1}-A_k)}{2}\epsilon\nonumber
    \end{aligned}
\end{equation}
The last inequality follows from \eqref{lsmooth}. In fact, \eqref{Lsmoothdescent} is implied by the gradient descent with $\rho_k=1/L$ and $L$-smoothness.
\begin{equation}
\label{Lsmoothdescent}
    f(y_{k+1})-f(x_{k+1})\leq\langle\nabla f(x_{k+1}),y_{k+1}-x_{k+1}\rangle+\frac{L}{2}\|y_{k+1}-x_{k+1}\|^2=-\frac{1}{2L}\|\nabla f(x_{k+1})\|^2
\end{equation}
With the choice of $\beta_k$ and $A_k$, we have
\begin{equation}
    \frac{1}{2\bar{\mu}\beta_k^2}=\frac{\bar{a}_k^2}{2\bar{\mu}\gamma^2}=\frac{\bar{\mu}\gamma^2}{32L^2}(2k+3)^2\leq\frac{\bar{\mu}\gamma^2}{8L^2}(k+2)^2=\frac{A_{k+1}}{2L}\nonumber
\end{equation}
Thus, we obtain the final bound:
\begin{equation}
\begin{aligned}
    E_{k+1}-E_k&\leq\left(\frac{1}{2\bar{\mu}\beta_k^2}-\frac{A_{k+1}}{2L}\right)\|\nabla f(x_{k+1})\|^2+\frac{(A_{k+1}-A_k)}{2}\epsilon\leq\frac{(A_{k+1}-A_k)}{2}\epsilon\nonumber
\end{aligned}
\end{equation}
By summing both sides of the inequality above, we obtain
\begin{equation*}
\begin{aligned}
    f(y_t)-f(x^*)&\leq A_t^{-1}\left(A_0(f(y_0)-f(x^*))+D_h(x^*,z_0)\right)+\frac{\epsilon}{2}\\
    &\overset{\eqref{lsmooth}}{\leq}A_t^{-1}\left(\frac{\bar{\mu}\gamma^2}{8}\|y_0-x^*\|^2+D_h(x^*,z_0)\right)+\frac{\epsilon}{2}\\
    &\overset{\eqref{st_h}}{\leq}A_t^{-1}\left(\left(\frac{\gamma^2}{4}+1\right)D_h(x^*,z_0)\right)+\frac{\epsilon}{2}\\
    &\leq\frac{8LD_h(x^*,z_0)}{\gamma^2\bar{\mu}(t+1)^2}+\frac{\epsilon}{2}\leq \frac{8LR^2}{\gamma^2\bar{\mu}(t+1)^2}+\frac{\epsilon}{2}
\end{aligned}
\end{equation*}
\textbf{Case 2:} $\mu>0$
\begin{equation*}
    \begin{aligned}
    E_{k+1}-E_k &\overset{\eqref{TPI}}{=} \bar{b}_k D_h(x^*,z_{k+1})-B_k\langle\nabla h(z_{k+1})-\nabla h(z_k),x^*-z_{k+1}\rangle-B_k D_h(z_{k+1},z_k)\\
    &\quad + A_{k+1}(f(y_{k+1})-f(x^*))-A_k(f(y_{k})-f(x^*))\\
    &=\bar{b}_k D_h(x^*,z_{k+1})+\frac{\alpha_k B_k}{\beta_k}\langle\nabla h(z_{k+1})-\nabla h(x_{k+1}), x^*-z_{k+1}\rangle+\frac{B_k}{\beta_k}\langle\nabla f(x_{k+1}), x^*-z_{k+1}\rangle\\
    &\quad-B_k D_h(z_{k+1},z_k)+A_{k+1}(f(y_{k+1})-f(x^*))-A_k(f(y_{k})-f(x^*))\\
    &\overset{\eqref{TPI}}{=} \left(\bar{b}_k-\frac{\alpha_k B_k}{\beta_k}\right)D_h(x^*,z_{k+1})+\frac{\alpha_k B_k}{\beta_k}\left(D_h(x^*,x_{k+1})-D_h(z_{k+1},x_{k+1})\right)+\frac{B_k}{\beta_k}\langle\nabla f(x_{k+1}),x^*-z_k\rangle\\
    &\quad+\frac{B_k}{\beta_k}\langle\nabla f(x_{k+1}),z_k-z_{k+1}\rangle-B_k D_h(z_{k+1},z_k)+A_{k+1}(f(y_{k+1})-f(x^*))-A_k(f(y_{k})-f(x^*))\\
    &\overset{\eqref{st_h}}{\leq}\frac{\alpha_k B_k}{\beta_k}D_h(x^*,x_{k+1})-\frac{\bar{\mu}\alpha_k B_k}{2\beta_k}\|z_{k+1}-x_{k+1}\|^2+\frac{B_k}{\beta_k}\langle\nabla f(x_{k+1}),x^*-z_k\rangle\\
    &\quad+\frac{B_k}{\beta_k}\langle\nabla f(x_{k+1}),z_k-z_{k+1}\rangle-\frac{\bar{\mu}B_k}{2}\|z_{k+1}-z_k\|^2+A_{k+1}(f(y_{k+1})-f(x^*))-A_k(f(y_{k})-f(x^*))\\
    &\leq \frac{\alpha_k B_k}{\beta_k}D_h(x^*,x_{k+1})-\frac{\bar{\mu}\alpha_k B_k}{2\beta_k}\|z_{k+1}-x_{k+1}\|^2+\frac{B_k}{\beta_k}\langle\nabla f(x_{k+1}),x^*-x_{k+1}\rangle\\
    &\quad+\frac{B_k}{\beta_k}\left(c(f(y_k)-f(x_{k+1}))+b\|x_{k+1}-z_k\|^2\right)+\frac{B_k}{\beta_k}\langle\nabla f(x_{k+1}),z_k-z_{k+1}\rangle-\frac{\bar{\mu}B_k}{2}\|z_{k+1}-z_k\|^2\\
    &\quad+A_{k+1}(f(y_{k+1})-f(x^*))-A_k(f(y_{k})-f(x^*))\\
    &\leq \frac{\alpha_kB_k}{\beta_k}D_h(x^*,x_{k+1})+\frac{B_k}{\beta_k}\langle\nabla f(x_{k+1}),z_k-z_{k+1}\rangle+\left(\frac{\bar{\mu}\alpha_kB_k}{2\beta_k}-\frac{\bar{\mu}B_k}{2}\right)\|z_{k+1}-z_k\|^2\\
    &\quad+\frac{B_k}{\beta_k}\langle\nabla f(x_{k+1}),x^*-x_{k+1}\rangle+A_k(f(y_k)-f(x_{k+1}))+A_{k+1}(f(y_{k+1})-f(x^*))-A_k(f(y_{k})-f(x^*))
    \end{aligned}
\end{equation*}
The first equality and the third equality follows from Lemma \ref{Three-point identity}; the second equality follows from mirror descent.
The first inequality follows from the strong convexity of $h$, and the second inequality follows from \eqref{slackEOGa}
(Lemma \ref{Existenceofalpha}). With the choice of $\alpha_k$ and $\beta_k$, we have ${\alpha_kB_k}/{\beta_k}=\bar{b}_k$, which explains the first inequality. Moreover, with the choice of $B_k$ and Observation \ref{lowerboundofkappa}, we have
\begin{equation*}
    \frac{\alpha_k}{\beta_k}=\frac{\bar{b}_k}{B_k}=\frac{\gamma}{2\sqrt{\kappa}}\leq\frac{\gamma}{2}\sqrt{\frac{2-\gamma}{\gamma}}\leq\frac{\gamma}{2}\sqrt{\frac{1}{\gamma^2}}=\frac{1}{2}.
\end{equation*}
Combined with the initial bound and the relation above, we obtain the following bound:
\begin{equation*}
    \begin{aligned}
    E_{k+1}-E_k&\leq\frac{\alpha_kB_k}{\beta_k}D_h(x^*,x_{k+1})+\frac{B_k}{\beta_k}\langle\nabla f(x_{k+1}),z_k-z_{k+1}\rangle+\left(\frac{\bar{\mu}\alpha_kB_k}{2\beta_k}-\frac{\bar{\mu}B_k}{2}\right)\|z_{k+1}-z_k\|^2\\
    &\quad+\frac{B_k}{\beta_k}\langle\nabla f(x_{k+1}),x^*-x_{k+1}\rangle+A_k(f(y_k)-f(x_{k+1}))+A_{k+1}(f(y_{k+1})-f(x^*))-A_k(f(y_{k})-f(x^*))\\
    &\leq\frac{\alpha_kB_k}{\beta_k}D_h(x^*,x_{k+1})+\frac{B_k}{\beta_k}\langle\nabla f(x_{k+1}),z_k-z_{k+1}\rangle-\frac{\bar{\mu}B_k}{4}\|z_{k+1}-z_k\|^2+\frac{B_k}{\beta_k}\langle\nabla f(x_{k+1}),x^*-x_{k+1}\rangle\\
    &\quad+A_{k+1}(f(y_{k+1})-f(x_{k+1}))+\bar{a}_k(f(x_{k+1})-f(x^*))\\
    &\overset{\eqref{fenchelyoung}\eqref{uniform-quasar-convex-of-f}}{\leq}\frac{\alpha_kB_k}{\beta_k}D_h(x^*,x_{k+1})+\frac{B_k}{\bar{\mu}\beta_k^2}\|\nabla f(x_{k+1})\|^2+\frac{\gamma B_k}{\beta_k}(f(x^*)-f(x_{k+1})-\mu D_h(x^*,x_{k+1}))\\
    &\quad+A_{k+1}(f(y_{k+1})-f(x_{k+1}))+\bar{a}_k(f(x_{k+1})-f(x^*))\\
    &\leq\frac{B_k}{\bar{\mu}\beta_k^2}\|\nabla f(x_{k+1})\|^2+A_{k+1}(f(y_{k+1})-f(x_{k+1}))\\
    &\overset{\eqref{Lsmoothdescent}}{\leq}\left(\frac{B_k}{\bar{\mu}\beta_k^2}-\frac{A_{k+1}}{2L}\right)\|\nabla f(x_{k+1})\|^2\\
    &\leq 0
    \end{aligned}
\end{equation*}
The third inequality follows from the Fenchel-Young inequality of $h$ and the strong quasar-convexity of $f$. The fifth inequality follows from $L$-smoothness of $f$. With the choice of $B_k,\beta_k$ and $A_k$, we have
\begin{equation*}
    \frac{B_k}{\bar{\mu}\beta_k^2}=\frac{(\gamma/2\sqrt{\kappa})^2A_k}{\gamma^2\bar{\mu}\mu}=\frac{A_k}{4L}\leq\frac{A_{k+1}}{2L},
\end{equation*}
which explains the last inequality. Therefore, we obtain
\begin{equation*}
    f(y_t)-f(x^*)\leq \left(1+\frac{\gamma}{2\sqrt{\kappa}}\right)^{-t}(f(y_0)-f(x^*)+\mu D_h(x^*,z_0))=\left(1+\frac{\gamma}{2\sqrt{\kappa}}\right)^{-t}E_0
\end{equation*}
\subsection{Proof of Theorem \ref{convergencerateASGD}}
\label{proofQASGD}
\begin{table}[H]
\begin{center}
\renewcommand\arraystretch{1.5}
    \begin{tabular}{|c|}
    \hline
    \rowcolor{black!20}
    \textbf{QASGD}\\
    \hline
    \rowcolor{black!10}
    \textbf{$\gamma$-quasar-convex} $(\mu=0)$\\
    \hline
    $A_k=\eta(k+1)^2$, $\eta=\min\left(\frac{\bar{\mu}}{L},\sqrt{\frac{4D_h(x^*,z_0)}{\sigma^2}}\frac{\gamma}{(t+1)^{3/2}}\right)$,
    $B_k=1$\\
    $(\alpha_k, \beta_k, \rho_k, \tilde{f},b,c,\tilde{\epsilon})\gets\left(0,\frac{\gamma}{\bar{a}_k}, 0, f_i, 0, \frac{\gamma A_k}{\bar{a}_k}, \frac{\gamma\epsilon}{2}\right)$\\
    \hline
    \rowcolor{black!10}
    \textbf{$\mu$-strongly $\gamma$-quasar-convex} $(\mu>0)$\\
    \hline
    \textbf{Step 1}: Choose $A_k,B_k,\theta_k$ as follows until convergence\\
    \hline
    $A_k=\left(1+\min\left\{{\gamma^2\bar{\mu}^2}/{16},{1}/{2}\right\}\right)^{k}$, $B_k=\mu A_k$\\
    $(\alpha_k, \beta_k, \rho_k, \tilde{f},b,c,\tilde{\epsilon})\gets\left(\frac{\gamma\mu}{2}, \frac{\gamma\mu B_k}{2\bar{b}_k}, 0, f_i, \frac{\gamma\bar{\mu}\mu}{4}, \frac{\gamma A_k}{2\bar{a}_k},0\right)$\\
    \hline
    \textbf{Step 2}: Restart and choose $A_k,B_k,\theta_k$ as follows\\
    \hline
    $A_k=\frac{\gamma^2\bar{\mu}^2}{36}\left(k+\max\{{48}/{\gamma^2\bar{\mu}^2},5\}\right)^2$, $B_k=\mu A_k$\\
    $(\alpha_k, \beta_k, \rho_k, \tilde{f},b,c,\tilde{\epsilon})\gets\left(\frac{\gamma\mu}{2}, \frac{\gamma\mu B_k}{2\bar{b}_k}, 0, f_i, \frac{\gamma\bar{\mu}\mu}{4}, \frac{\gamma A_k}{2\bar{a}_k},0\right)$\\
    \hline
    \end{tabular}
    \caption{Parameter choices for QASGD}
    \label{parameterofQASGD}
    \end{center}
\end{table}
Note that $y_k$ is identical to $x_k$ since $\rho_k=0$. Thus we can substitute $y_k$ in Lyapunov functions \eqref{lyap} with $x_k$. 
We begin with Algorithm \ref{algorithm_for_quasar}, $\nabla_k=\nabla f_i(x_{k+1})$ and parameters specified in Table \ref{parameterofQASGD} using Lyapunov function \eqref{lyap}:

\textbf{Case 1:} $\mu=0$
\begin{equation}
    \begin{aligned}
     E_{k+1}-E_k&=-\left\langle\nabla h(z_{k+1})-\nabla h(z_k),x^*-z_{k+1}\right\rangle-D_h(z_{k+1},z_k)+A_{k+1}(f(x_{k+1})-f(x^*))-A_k(f(x_{k})-f(x^*))\\
     &=\frac{1}{\beta_k}\left\langle\nabla_k,x^*-z_{k+1}\right\rangle-D_h(z_{k+1},z_k)+A_{k+1}(f(x_{k+1})-f(x^*))-A_k(f(x_{k})-f(x^*))\\
     &\overset{\eqref{st_h}}{\leq}\frac{1}{\beta_k}\left\langle\nabla_k,x^*-z_k\right\rangle+\frac{1}{\beta_k}\left\langle\nabla_k,z_k-z_{k+1}\right\rangle-\frac{\bar{\mu}}{2}\|z_{k+1}-z_k\|^2+A_{k+1}(f(x_{k+1})-f(x^*))\\
     &\quad-A_k(f(x_{k})-f(x^*))\\
     &\leq\frac{1}{\beta_k}\left\langle\nabla_k,x^*-z_k\right\rangle+\frac{1}{2\bar{\mu}\beta_k^2}\|\nabla_k\|^2+A_{k+1}(f(x_{k+1})-f(x^*))-A_k(f(x_{k})-f(x^*))\\
     &=\frac{1}{\beta_k}\left\langle\nabla_k,x^*-x_{k+1}\right\rangle+\frac{\tau_k}{\beta_k}\left\langle\nabla_k,x_k-z_k\right\rangle+\frac{1}{2\bar{\mu}\beta_k^2}\|\nabla_k\|^2+\mathcal{A}_k\\
     &\leq\frac{1}{\beta_k}\left\langle\nabla_k,x^*-x_{k+1}\right\rangle+\frac{1}{\beta_k}(c(f_{i}(x_k)-f_{i}(x_{k+1}))+\tilde{\epsilon})+\frac{1}{2\bar{\mu}\beta_k^2}\|\nabla_k\|^2+\mathcal{A}_k\nonumber
    \end{aligned}
\end{equation}
where $\mathcal{A}_k\triangleq \bar{a}_k(f(x_{k+1})-f(x^*))+A_k(f(x_{k+1})-f(x_k))$. The second equality follows from mirror descent; the first inequality follows from the strong convexity of $h$; the second inequality follows from the Fenchel-Young inequality, and the last inequality follows from \eqref{slackEOGa} (Lemma \ref{Existenceofalpha}). We take the expectation of both sides of the initial bound and obtain
\begin{equation*}
    \begin{aligned}
    \mathbb{E}[E_{k+1}-E_k]&\leq \frac{1}{\beta_k}\left\langle\nabla f(x_{k+1}),x^*-x_{k+1}\right\rangle+\frac{c}{\beta_k}(f(x_k)-f(x_{k+1}))+\frac{1}{2\bar{\mu}\beta_k^2}\mathbb{E}\left[\|\nabla_k\|^2\right]+\frac{\bar{a}_k}{2}\epsilon\\
     &\quad+\bar{a}_k(f(x_{k+1})-f(x^*))+A_k(f(x_{k+1})-f(x_k))\\
     &\leq\frac{\gamma}{\beta_k}(f(x^*)-f(x_{k+1}))+A_k(f(x_k)-f(x_{k+1}))+\frac{1}{2\bar{\mu}\beta_k^2}\mathbb{E}\left[\|\nabla_k\|^2\right]+\frac{\bar{a}_k}{2}\epsilon\\
     &\quad+\bar{a}_k(f(x_{k+1})-f(x^*))+A_k(f(x_{k+1})-f(x_k))\\
     &\leq\frac{1}{2\bar{\mu}\beta_k^2}\mathbb{E}\left[\|\nabla_k\|^2\right]+\frac{\bar{a}_k}{2}\epsilon
    \end{aligned}
\end{equation*}
By summing both sides of the inequality above, we obtain
\begin{equation*}
\begin{aligned}
    \mathbb{E}[f(x_t)-f(x^*)]&\leq A_t^{-1}\left(A_0(f(x_0)-f(x^*))+D_h(x^*,z_0)+\sum_{k=0}^{t-1}\frac{1}{2\bar{\mu}\beta_k^2}\mathbb{E}\left[\|\nabla_k\|^2\right]\right)+\frac{\epsilon}{2}\\
    &\leq A_t^{-1}\left(2D_h(x^*,z_0)+\sum_{k=0}^{t-1}\frac{1}{2\bar{\mu}\beta_k^2}\mathbb{E}\left[\|\nabla_k\|^2\right]\right)+\frac{\epsilon}{2}\\
    &\leq\frac{2D_h(x^*,z_0)}{\eta(t+1)^2}+\frac{\sigma^2\eta}{\gamma^2\bar{\mu}}(t+1)+\frac{\epsilon}{2}\\
    &\leq\frac{2LD_h(x^*,z_0)}{\bar{\mu}(t+1)^2}+\frac{2\sigma}{\gamma\bar{\mu}}\sqrt{\frac{D_h(x^*,z_0)}{t+1}}+\frac{\epsilon}{2}\leq\frac{2LR^2}{\bar{\mu}(t+1)^2}+\frac{2\sigma R}{\gamma\bar{\mu}\sqrt{t+1}}+\frac{\epsilon}{2}
\end{aligned}
\end{equation*}
\textbf{Case 2:} $\mu>0$
\begin{equation*}
    \begin{aligned}
    E_{k+1}-E_k &= \bar{b}_k D_h(x^*,z_{k+1})-B_k\langle\nabla h(z_{k+1})-\nabla h(z_k),x^*-z_{k+1}\rangle-B_k D_h(z_{k+1},z_k)\\
    &\quad + \bar{a}_k(f(x_{k+1})-f(x^*))+A_k(f(x_{k+1})-f(x_k))\\
    &=\bar{b}_k D_h(x^*,z_{k+1})+\frac{\alpha_k B_k}{\beta_k}\langle\nabla h(z_{k+1})-\nabla h(x_{k+1}), x^*-z_{k+1}\rangle+\frac{B_k}{\beta_k}\langle\nabla_k, x^*-z_{k+1}\rangle\\
    &\quad-B_k D_h(z_{k+1},z_k)+\bar{a}_k(f(x_{k+1})-f(x^*))+A_k(f(x_{k+1})-f(x_k))\\
    &\overset{\eqref{TPI}}{=} \left(\bar{b}_k-\frac{\alpha_k B_k}{\beta_k}\right)D_h(x^*,z_{k+1})+\frac{\alpha_k B_k}{\beta_k}\left(D_h(x^*,x_{k+1})-D_h(z_{k+1},x_{k+1})\right)+\frac{B_k}{\beta_k}\langle\nabla_k,x^*-z_k\rangle\\
    &\quad+\frac{B_k}{\beta_k}\langle\nabla_k,z_k-z_{k+1}\rangle-B_k D_h(z_{k+1},z_k)+\bar{a}_k(f(x_{k+1})-f(x^*))+A_k(f(x_{k+1})-f(x_k))
    \end{aligned}
\end{equation*}
The second equality follows from mirror descent. With the choice of $\beta_k,\alpha_k$ and Observation \ref{lowerboundofkappa}, we have $\bar{b}_k=\alpha_kB_k/\beta_k$ and $\alpha_k/\beta_k\leq 1/2$. Therefore, We obtain the following bound:
\begin{equation*}
    \begin{aligned}
    E_{k+1}-E_k&\leq\frac{\alpha_k B_k}{\beta_k}D_h(x^*,x_{k+1})-\frac{\bar{\mu}\alpha_k B_k}{2\beta_k}\|z_{k+1}-x_{k+1}\|^2+\frac{B_k}{\beta_k}\langle\nabla_k,x^*-x_{k+1}\rangle+\frac{\tau_kB_k}{\beta_k}\langle\nabla_k,x_k-z_k\rangle\\
    &\quad+\frac{B_k}{\beta_k}\langle\nabla_k,z_k-z_{k+1}\rangle-\frac{\bar{\mu}B_k}{2}\|z_{k+1}-z_k\|^2+\mathcal{A}_k\\
    &\leq\frac{\alpha_k B_k}{\beta_k}D_h(x^*,x_{k+1})-\frac{\bar{\mu}\alpha_k B_k}{2\beta_k}\|z_{k+1}-x_{k+1}\|^2+\frac{B_k}{\beta_k}\langle\nabla_k,x^*-x_{k+1}\rangle+\frac{B_k}{\beta_k}(c(f_i(x_k)-f_i(x_{k+1}))\\
    &\quad+b\|z_k-x_{k+1}\|^2)+\frac{B_k}{\beta_k}\langle\nabla_k,z_k-z_{k+1}\rangle-\frac{\bar{\mu}B_k}{2}\|z_{k+1}-z_k\|^2+\mathcal{A}_k\\
    &\leq\frac{\alpha_k B_k}{\beta_k}D_h(x^*,x_{k+1})-\frac{\bar{\mu}B_k}{4}\|z_{k+1}-z_k\|^2+\frac{B_k}{\beta_k}\langle\nabla_k,z_k-z_{k+1}\rangle+\frac{B_k}{\beta_k}\langle\nabla_k,x^*-x_{k+1}\rangle\\
    &\quad+A_k(f_i(x_k)-f_i(x_{k+1}))+\mathcal{A}_k\\
    &\leq\frac{\alpha_k B_k}{\beta_k}D_h(x^*,x_{k+1})+\frac{B_k}{\bar{\mu}\beta_k^2}\|\nabla_k\|^2+\frac{B_k}{\beta_k}\langle\nabla_k,x^*-x_{k+1}\rangle+\mathcal{A}_k+A_k(f_i(x_k)-f_i(x_{k+1}))
    \end{aligned}
\end{equation*}
where $\mathcal{A}_k\triangleq \bar{a}_k(f(x_{k+1})-f(x^*))+A_k(f(x_{k+1})-f(x_k))$. The second inequality follows from \eqref{slackEOGa} (Lemma \ref{Existenceofalpha}); the second inequality follows from the triangle inequality and the last inequality follows from the Fenchel-Young inequality. We take the expectation of both sides of the inequality above and obtain
\begin{equation*}
    \begin{aligned}
    \mathbb{E}[E_{k+1}-E_k]&\leq\frac{\alpha_k B_k}{\beta_k}D_h(x^*,x_{k+1})+\frac{B_k}{\bar{\mu}\beta_k^2}\mathbb{E}\left[\|\nabla_k\|^2\right]+\frac{B_k}{\beta_k}\langle\nabla f(x_{k+1}),x^*-x_{k+1}\rangle+\bar{a}_k(f(x_{k+1})-f(x^*))\\
    &\leq\frac{\alpha_k B_k}{\beta_k}D_h(x^*,x_{k+1})+\frac{B_k}{\bar{\mu}\beta_k^2}\mathbb{E}\left[\|\nabla_k\|^2\right]+\frac{\gamma B_k}{\beta_k}(f(x^*)-f(x_{k+1})-\mu D_h(x^*,x_{k+1}))\\
    &\quad+\bar{a}_k(f(x_{k+1})-f(x^*))\\
    &\leq-\frac{\gamma\mu B_k}{2\beta_k}D_h(x^*,x_{k+1})+\frac{2\mu^2B_k}{\bar{\mu}\beta_k^2}\|x^*-x_{k+1}\|^2+\frac{\sigma^2B_k}{\bar{\mu}\beta_k^2}\\
    &\leq\left(\frac{4\mu^2B_k}{\bar{\mu}^2\beta_k^2}-\frac{\gamma\mu B_k}{2\beta_k}\right)D_h(x^*,x_{k+1})+\frac{\sigma^2B_k}{\bar{\mu}\beta_k^2}
    \end{aligned}
\end{equation*}
Next We prove that $4\mu^2B_k/\bar{\mu}^2\beta_k^2\leq\gamma\mu B_k/2\beta_k$. It suffices to prove that $4\mu/\bar{\mu}^2\beta_k\leq\gamma/2$.
\begin{equation*}
    \frac{4\mu}{\bar{\mu}^2\beta_k}=\frac{8\bar{a}_k}{\gamma\bar{\mu}^2A_k}\leq\frac{8}{\gamma\bar{\mu}^2}\frac{\gamma^2\bar{\mu}^2}{16}=\frac{\gamma}{2}
\end{equation*}
Thus we obtain the final bound:
\begin{equation*}
    \mathbb{E}[E_{k+1}-E_k]\leq\frac{\sigma^2B_k}{\bar{\mu}\beta_k^2}
\end{equation*}
By summing both sides of the inequality above and using the notation $\mathcal{q}\triangleq\min\{\gamma^2\bar{\mu}^2/16,1/2\}$, we obtain
\begin{equation*}
    \begin{aligned}
    \mathbb{E}[f(x_t)-f(x^*)]&\leq A_t^{-1}\left(f(x_0)-f(x^*)+\mu D_h(x^*,z_0)+\sum_{k=0}^{t-1}\frac{\sigma^2B_k}{\bar{\mu}\beta_k^2}\right)\\
    &\leq A_t^{-1}\left(f(x_0)-f(x^*)+\mu D_h(x^*,z_0)+\sum_{k=0}^{t-1}\frac{4\sigma^2(A_{k+1}-A_k)^2}{\gamma^2\bar{\mu}\mu A_k}\right)\\
    &\leq A_t^{-1}\left(f(x_0)-f(x^*)+\mu D_h(x^*,z_0)+\sum_{k=0}^{t-1}\frac{4\sigma^2q^2A_k}{\gamma^2\bar{\mu}\mu}\right)\\
    &\leq A_t^{-1}\left(f(x_0)-f(x^*)+\mu D_h(x^*,z_0)+\frac{4\sigma^2qA_t}{\gamma^2\bar{\mu}\mu}\right)\\
    &=(1+\mathcal{q})^{-t}E_0+\frac{\bar{\mu}\sigma^2}{4\mu}
    \end{aligned}
\end{equation*}
In Step 1, the algorithm is run until convergence and we let $x_0$ be the last iterate of the Step 1. Assume $\mathbb{E}[f(x_0)-f(x^*)+\mu D_h(x^*,z_0)]\leq\frac{\bar{\mu}\sigma^2}{4\mu}$, and we restart the algorithm using the parameters in Step 2 and the notation $m=\max\left\{\frac{48}{\gamma^2\bar{\mu}^2},5\right\}$. Then We obtain
\begin{equation*}
    \begin{aligned}
    \mathbb{E}[f(x_t)-f(x^*)]&\leq A_t^{-1}\left(f(x_0)-f(x^*)+\mu D_h(x^*,z_0)+\sum_{k=0}^{t-1}\frac{4\sigma^2(A_{k+1}-A_k)^2}{\gamma^2\bar{\mu}\mu A_k}\right)\\
    &\leq A_t^{-1}\left(\frac{\bar{\mu}\sigma^2}{4\mu}+\sum_{k=0}^{t-1}\frac{\bar{\mu}\sigma^2(2k+2m+1)^2}{9\gamma^2\mu (k+m)^2}\right)\\
    &\leq A_t^{-1}\left(\frac{\bar{\mu}\sigma^2}{4\mu}+\frac{\bar{\mu}\sigma^2t}{\gamma^2\mu}\right)\\
    &\leq \frac{9\sigma^2}{\gamma^2\bar{\mu}\mu(t+m)^2}+\frac{36\sigma^2}{\gamma^2\bar{\mu}\mu(t+m)}
    \end{aligned}
\end{equation*}
Additionally, we need to verify that the parameters in Step 2 satisfy two essential relations, $\alpha_k/\beta_k\leq 1/2$ and $4\mu/\bar{\mu}^2\beta_k^2\leq \gamma/2$, which are key to obtaining the final bound of $E_{k+1}-E_k$.
\begin{equation*}
    \frac{\alpha_k}{\beta_k}=\frac{A_{k+1}-A_k}{A_k}=\frac{2k+2m+1}{(k+m)^2}\leq\frac{2m+1}{m^2}\leq\frac{11}{25}\leq\frac{1}{2}
\end{equation*}
\begin{equation*}
    \frac{4\mu}{\bar{\mu}^2\beta_k^2}\leq\frac{8(A_{k+1}-A_k)}{\gamma\bar{\mu}^2A_k}\leq\frac{8}{\gamma\bar{\mu}^2}\frac{2m+1}{m^2}\leq\frac{8}{\gamma\bar{\mu}^2}\frac{3}{m}\leq\frac{8}{\gamma\bar{\mu}^2}\frac{\gamma^2\bar{\mu}^2}{16}=\frac{\gamma}{2}
\end{equation*}
\subsection{Proof of Theorem \ref{convergencerateQASVRGsingle}}
\label{proofQASVRG}
\begin{table}[H]
\begin{center}
\renewcommand\arraystretch{1.5}
    \begin{tabular}{|c|}
    \hline
    \rowcolor{black!20}
    \textbf{QASVRG}\\
    \hline
    \rowcolor{black!10}
    \textbf{$\gamma$-quasar-convex} $(\mu=0)$\\
    \hline
    $A_k=\frac{\gamma^2\bar{\mu}}{16L}(k+1)^2$, $B_k=1$\\
    Batchsize $b_k=\left\lceil\frac{\gamma\bar{\mu}n(2k+3)}{2(n-1)p+\gamma\bar{\mu}(2k+3)}\right\rceil,p\leq\frac{\gamma\bar{\mu}}{16}$\\
    $(\alpha_k, \beta_k, \rho_k, \tilde{f},b,c,\tilde{\epsilon})\gets\left(0,\frac{\gamma}{2\bar{a}_k},\frac{1}{L}, f_{I_k}, 0, \frac{\gamma A_k}{2\bar{a}_k}, \frac{\gamma\epsilon f(y_0)}{2}\right)$\\
    \hline
    \rowcolor{black!10}
    \textbf{$\mu$-strongly $\gamma$-quasar-convex} $(\mu>0)$\\
    \hline
    \hypertarget{option1}{\textbf{Option} \rom{1}}\\
    \hline
    $A_k=(1+{\gamma}/{\sqrt{8\kappa}})^{k}$, $B_k=\mu A_k$\\
    Batchsize $b_k=\left\lceil\frac{8n(\sqrt{8\kappa}+\gamma)}{\gamma(n-1)+8(\sqrt{8\kappa}+\gamma)}\right\rceil$\\
    $(\alpha_k, \beta_k, \rho_k, \tilde{f},b,c,\tilde{\epsilon})\gets\left(\frac{\gamma\mu}{2}, \frac{\gamma\mu B_k}{2\bar{b}_k}, \frac{1}{L}, f_{I_k}, \frac{\gamma\bar{\mu}\mu}{4}, \frac{\gamma A_k}{2\bar{a}_k},0\right)$\\
    \hline
    \hypertarget{option2}{\textbf{Option} \rom{2}}\\
    \hline
    $A_k=\frac{\gamma^2\bar{\mu}}{16L}(k+1)^2$, $B_k=1$\\
    Batchsize $b_k=\left\lceil\frac{\gamma\bar{\mu}n(2k+3)}{2(n-1)p+\gamma\bar{\mu}(2k+3)}\right\rceil, p\leq\frac{\gamma\bar{\mu}}{16}$\\
    $(\alpha_k, \beta_k, \rho_k, \tilde{f},b,c,\tilde{\epsilon})\gets\left(0,\frac{\gamma}{2\bar{a}_k},\frac{1}{L}, f_{I_k}, 0, \frac{\gamma A_k}{2\bar{a}_k}, \frac{\gamma\epsilon f(y_0)}{2}\right)$\\
    \hline
    \end{tabular}
    \caption{Parameter choices for QASVRG}
    \label{parameterofQASVRG}
    \end{center}
\end{table}
We begin with Algorithm \ref{algorithm_for_quasar}, $\nabla_k=\nabla f_{I_k}(x_{k+1})-\nabla f_{I_k}(y_0)+\nabla f(y_0)$ and parameters specified in Table \ref{parameterofQASVRG} using Lyapunov function \eqref{lyap}:

\textbf{case 1:} $\mu=0$
\begin{equation}
    \begin{aligned}
     E_{k+1}-E_k&=-\left\langle\nabla h(z_{k+1})-\nabla h(z_k),x^*-z_{k+1}\right\rangle-D_h(z_{k+1},z_k)+A_{k+1}(f(y_{k+1})-f(x^*))-A_k(f(y_{k})-f(x^*))\\
     &=\frac{1}{\beta_k}\left\langle\nabla_k,x^*-z_{k+1}\right\rangle-D_h(z_{k+1},z_k)+A_{k+1}(f(y_{k+1})-f(x^*))-A_k(f(y_{k})-f(x^*))\\
     &\leq\frac{1}{\beta_k}\left\langle\nabla_k,x^*-z_k\right\rangle+\frac{1}{2\bar{\mu}\beta_k^2}\|\nabla_k\|^2+A_{k+1}(f(y_{k+1})-f(x^*))-A_k(f(y_{k})-f(x^*))\\
     &=\frac{1}{\beta_k}\left\langle\nabla_k,x^*-x_{k+1}\right\rangle+\frac{\tau_k}{\beta_k}\left\langle\nabla_k,y_k-z_k\right\rangle+\frac{1}{2\bar{\mu}\beta_k^2}\|\nabla_k\|^2+A_{k+1}(f(y_{k+1})-f(x_{k+1}))\\
     &\quad+(A_{k+1}-A_k)(f(x_{k+1})-f(x^*))+A_k(f(x_{k+1})-f(y_k))\\
     &\leq\frac{1}{\beta_k}\left\langle\nabla_k,x^*-x_{k+1}\right\rangle+\frac{1}{\beta_k}\left(c(f_{I_k}(y_k)-f_{I_k}(x_{k+1}))+\tilde{\epsilon}\right)+\frac{1}{2\bar{\mu}\beta_k^2}\|\nabla_k\|^2+A_{k+1}(f(y_{k+1})-f(x_{k+1}))\\
     &\quad +(A_{k+1}-A_k)(f(x_{k+1})-f(x^*))+A_k(f(x_{k+1})-f(y_k))+\frac{\tau_k}{\beta_k}\left\langle\nabla f(y_0)-\nabla f_{I_k}(y_0),y_k-z_k\right\rangle\nonumber
    \end{aligned}
\end{equation}
The first equality follows from the mirror descent; the third equality follows from the momentum step; the first inequality follows from the Fenchel-Young inequality, and the last inequality follows from \eqref{slackEOGa} (Lemma \ref{Existenceofalpha}). Then we take the expectation with respect to the history of random variable $I_{i_j}$ with its size equal to $b_i$, where $i=0,1,...,t-1$ and $j=1,2,...,\tbinom{n}{b_i}$, and we obtain
\begin{equation}
    \begin{aligned}
    \mathbb{E}[E_{k+1}-E_k]
    &\leq \frac{1}{\beta_k}\left\langle\nabla f(x_{k+1}),x^*-x_{k+1}\right\rangle+\frac{c}{\beta_k}(f(y_k)-f(x_{k+1}))+\frac{1}{2\bar{\mu}\beta_k^2}\mathbb{E}_{I_k}\left[\|\nabla_k\|^2\right]+\bar{a}_k(f(x_{k+1})-f(x^*))\\
     &\quad+A_{k+1}\mathbb{E}\left[f(y_{k+1})-f(x_{k+1})\right]+A_k(f(x_{k+1})-f(y_k))+\bar{a}_kf(y_0)\epsilon\\
    &\leq\bar{a}_k(f(x^*)-f(x_{k+1}))+\frac{1}{2\bar{\mu}\beta_k^2}\mathbb{E}_{I_k}\left[\|\nabla_k-\nabla f(x_{k+1})\|^2\right]+A_{k+1}\mathbb{E}\left[f(y_{k+1})-f(x_{k+1})\right]\\
    &\quad+\frac{1}{2\bar{\mu}\beta_k^2}\|\nabla f(x_{k+1})\|^2+\bar{a}_kf(y_0)\epsilon\nonumber
    \end{aligned}
\end{equation}
The first inequality follows from $\mathbb{E}\left[\|\nabla_k\|^2\right]=\mathbb{E}_{I_k}\left[\|\nabla_k\|^2\right]$, as $\nabla_k$ corresponds to the mini-batch in the $k^{\text{th}}$ iteration; the second inequality follows from the quasar-convexity of $f$ and the relation $\mathbb{E}_{I_k}\left[\|\nabla_k\|^2\right]=\mathbb{E}_{I_k}\left[\|\nabla_k-\nabla f(x_{k+1})\|^2\right]+\|\nabla f(x_{k+1})\|^2$. By $L$-smoothness and the gradient descent in Algorithm \ref{algorithm_for_quasar}, we have the following relation:
\begin{equation}
\begin{aligned}
\label{fy-fx}
    \mathbb{E}[f(y_{k+1})-f(x_{k+1})]&\leq\mathbb{E}[\langle\nabla f(x_{k+1}),y_{k+1}-x_{k+1}\rangle]+\frac{L}{2}\mathbb{E}\left[\|y_{k+1}-x_{k+1}\|^2\right]\\
    &=-\frac{1}{L}\|\nabla f(x_{k+1})\|^2+\frac{1}{2L}\mathbb{E}\left[\|\nabla_k\|^2\right]\\
    &=-\frac{1}{2L}\|\nabla f(x_{k+1})\|^2+\frac{1}{2L}\mathbb{E}\left[\|\nabla_k-\nabla f(x_{k+1})\|^2\right]
\end{aligned}
\end{equation}
Using the relation above, we obtain the following bound:
\begin{equation*}
    \begin{aligned}
    \mathbb{E}[E_{k+1}-E_k]&\leq\bar{a}_k(f(x^*)-f(x_{k+1}))+\frac{1}{2\bar{\mu}\beta_k^2}\mathbb{E}_{I_k}\left[\|\nabla_k-\nabla f(x_{k+1})\|^2\right]+A_{k+1}\mathbb{E}\left[f(y_{k+1})-f(x_{k+1})\right]\\
    &\quad+\frac{1}{2\bar{\mu}\beta_k^2}\|\nabla f(x_{k+1})\|^2+\bar{a}_kf(y_0)\epsilon\\
    &\leq \bar{a}_k(f(x^*)-f(x_{k+1}))+\left(\frac{1}{2\bar{\mu}\beta_k^2}+\frac{A_{k+1}}{2L}\right)\mathbb{E}_{I_k}\left[\|\nabla_k-\nabla f(x_{k+1})\|^2\right]\\
    &\quad+\left(\frac{1}{2\bar{\mu}\beta_k^2}-\frac{A_{k+1}}{2L}\right)\|\nabla f(x_{k+1})\|^2+\bar{a}_kf(y_0)\epsilon\\
    &\leq\bar{a}_k(f(x^*)-f(x_{k+1}))+\frac{A_{k+1}}{L}\mathbb{E}_{I_k}\left[\|\nabla_k-\nabla f(x_{k+1})\|^2\right]+\bar{a}_kf(y_0)\epsilon\\
    &\overset{\eqref{varupperboundproof}}{\leq}\bar{a}_k(f(x^*)-f(x_{k+1}))+4A_{k+1}\delta_k(f(x_{k+1})-f(x^*))+4A_{k+1}\delta_k(f(y_0)-f(x^*))+\bar{a}_kf(y_0)\epsilon
    \end{aligned}
\end{equation*}
We use the notation $\delta_k\triangleq\frac{n-b_k}{b_k(n-1)}$. The third inequality follows from the relation $1/\bar{\mu}\beta_k^2\leq A_{k+1}/L$. Actually, with the choice of $\beta_k$ and $A_k$, we have
\begin{equation*}
    \frac{1}{\bar{\mu}\beta_k^2}=\frac{4\bar{a}_k^2}{\bar{\mu}\gamma^2}=\frac{4\eta^2(2k+3)^2}{\bar{\mu}\gamma^2}\leq\frac{16\eta^2(k+2)^2}{\bar{\mu}\gamma^2}=\frac{A_{k+1}}{L},
\end{equation*}
where $\eta=\gamma^2\bar{\mu}/16L$. The last inequality follows from Proposition 3.1 \eqref{varupperboundproof}. Next we prove the relation $4L\delta_k/\beta_k\leq p$. It suffices to prove $\delta_k\leq\gamma p/8L\bar{a}_k$.
\begin{equation*}
    \delta_k=\frac{n-b_k}{b_k(n-1)}\leq\frac{2np}{2(n-1)p+\gamma\bar{\mu}(2k+3)}\frac{2(n-1)p+\gamma\bar{\mu}(2k+3)}{\gamma\bar{\mu}n(2k+3)}=\frac{2p}{\gamma\bar{\mu}(2k+3)}=\frac{\gamma p}{8L\bar{a}_k}
\end{equation*}
Thus we have $4A_{k+1}\delta_k\leq pA_{k+1}\beta_k/L$, by which we obtain the final bound of $\mathbb{E}[E_{k+1}-E_k]$.
\begin{equation*}
    \begin{aligned}
    \mathbb{E}[E_{k+1}-E_k]&\leq\bar{a}_k(f(x^*)-f(x_{k+1}))+4A_{k+1}\delta_k(f(x_{k+1})-f(x^*))+4A_{k+1}\delta_k(f(y_0)-f(x^*))+\bar{a}_kf(y_0)\epsilon\\
    &\leq\left(\bar{a}_k-\frac{pA_{k+1}\beta_k}{L}\right)(f(x^*)-f(x_{k+1}))+\frac{pA_{k+1}\beta_k}{L}(f(y_0)-f(x^*))+\bar{a}_kf(y_0)\epsilon\\
    &\leq\frac{8p\bar{a}_k}{\gamma\bar{\mu}}(f(y_0)-f(x^*))+\bar{a}_kf(y_0)\epsilon\\
    &\leq\frac{8p\bar{a}_k}{\gamma\bar{\mu}}f(y_0)+\bar{a}_kf(y_0)\epsilon
    \end{aligned}
\end{equation*}
The third inequality follows from two relations, $\bar{a}_k\geq pA_{k+1}\beta_k/L$ and $A_{k+1}\beta_k/L\leq8\bar{a}_k/\gamma\bar{\mu}$ and the last inequality follows from $f(x^*)\geq0$. With the choice of $A_k$ and $\beta_k$, we have
\begin{equation*}
    \frac{pA_{k+1}\beta_k}{L}=\frac{\gamma pA_{k+1}}{2L\bar{a}_k}\leq\frac{\gamma^2\bar{\mu}(k+2)^2}{32L(2k+3)}\leq\frac{\gamma^2\bar{\mu}(2k+3)}{32L}=\frac{\bar{a}_k}{2}\leq\bar{a}_k
\end{equation*}
\begin{equation*}
    \frac{A_{k+1}\beta_k}{L}=\frac{\gamma A_{k+1}}{2L\bar{a}_k}=\frac{\gamma(k+2)^2}{2L(2k+3)}\leq\frac{\gamma(2k+3)}{2L}=\frac{8\bar{a}_k}{\gamma\bar{\mu}}
\end{equation*}
Summing both sides of the inequality about the final bound, we obtain
\begin{equation*}
\begin{aligned}
    \mathbb{E}[f(y_t)-f(x^*)]&\leq A_t^{-1}\left(A_0(f(y_0)-f(x^*))+D_h(x^*,z_0)\right)+A_t^{-1}\sum_{k=0}^{t-1}\frac{8p\bar{a}_k}{\gamma\bar{\mu}}f(y_0)+f(y_0)\epsilon\\
    &\leq\frac{17LD_h(x^*,z_0)}{\gamma^2\bar{\mu}(t+1)^2}+\left(\frac{8p}{\gamma\bar{\mu}}+\epsilon\right)f(y_0)\\
    &\leq\frac{17LR^2}{\gamma^2\bar{\mu}(t+1)^2}+\left(\frac{8p}{\gamma\bar{\mu}}+\epsilon\right)f(y_0)
\end{aligned}
\end{equation*}
When $t\geq\left\lceil\sqrt{\frac{17LD_h(x^*,z_0)}{\gamma^2\bar{\mu}qf(y_0)}}\right\rceil$, we have
\begin{equation*}
    \mathbb{E}[f(y_t)-f(x^*)]\leq\left(q+\frac{8p}{\gamma\bar{\mu}}+\epsilon\right)f(y_0),
\end{equation*}
where $q+8p/\gamma\bar{\mu}+\epsilon<1$ with the choice of $q$ and $p$. Next we conclude the convergence rate of global stages. Suppose $y_{s}$ is the output of stage s. If $t\geq\left\lceil\sqrt{\frac{17LD_h(x^*,y_s)}{\gamma^2\bar{\mu}qf(y_s)}}\right\rceil$ at each stage, we have
\begin{equation*}
    \mathbb{E}[f(y_{s})-f(x^*)]\leq\left(q+\frac{8p}{\gamma\bar{\mu}}+\epsilon\right)^{s+1}f(y_0)
\end{equation*}
\textbf{Case 2: }$\mu>0$ (Option \rom{1})
Using the notation $d_{k+1}=A_{k+1}(f(y_{k+1})-f(x^*))-A_k(f(y_k)-f(x^*))$, we obtain
\begin{equation}
    \begin{aligned}
    E_{k+1}-E_k&=\bar{b}_kD_h(x^*,z_{k+1})-B_k\langle\nabla h(z_{k+1})-\nabla h(z_k),x^*-z_{k+1}\rangle-B_kD_h(z_{k+1},z_k)+d_{k+1}\\
    &{\leq}\bar{b}_kD_h(x^*,z_{k+1})+\frac{\alpha_kB_k}{\beta_k}\langle\nabla h(z_{k+1})-\nabla h(x_{k+1}),x^*-z_{k+1}\rangle+\frac{B_k}{\beta_k}\langle\nabla_k,x^*-z_{k+1}\rangle\\
    &\quad-\frac{\bar{\mu}B_k}{2}\|z_{k+1}-z_k\|^2+d_{k+1}\\
    &=\left(\bar{b}_k -\frac{\alpha_k B_k}{\beta_k}\right)D_h(x^*,z_{k+1})+\frac{\alpha_k B_k}{\beta_k}(D_h(x^*,x_{k+1})-D_h(z_{k+1},x_{k+1}))+\frac{B_k}{\beta_k}\langle\nabla_k, x^*-x_{k+1}\rangle\\
    &\quad +\frac{B_k(1-\tau_k)}{\beta_k}\langle\nabla_k, y_k-z_k\rangle+\frac{B_k}{\beta_k}\langle\nabla_k, z_k-z_{k+1}\rangle-\frac{\bar{\mu}B_k}{2}\|z_{k+1}-z_k\|^2+d_{k+1}\\
    &\leq\frac{\alpha_k B_k}{\beta_k}D_h(x^*,x_{k+1})-\frac{\bar{\mu}\alpha_kB_k}{2\beta_k}\|z_{k+1}-x_{k+1}\|^2+\frac{B_k}{\beta_k}\langle\nabla_k, x^*-x_{k+1}\rangle+\frac{B_k}{\beta_k}\langle\nabla_k, z_k-z_{k+1}\rangle\\
    &\quad-\frac{\bar{\mu}B_k}{2}\|z_{k+1}-z_k\|^2+\frac{B_k}{\beta_k}\left(c(f_{I_k}(y_k)-f_{I_k}(x_{k+1}))+b\|x_{k+1}-z_k\|^2\right)+d_{k+1}+\xi_k\\
    &\leq\frac{\alpha_kB_k}{\beta_k}D_h(x^*,x_{k+1})+\frac{B_k}{\beta_k}\langle\nabla_k,x^*-x_{k+1}\rangle+\frac{B_k}{\bar{\mu}\beta_k^2}\|\nabla_k\|^2+\frac{cB_k}{\beta_k}(f_{I_k}(y_k)-f_{I_k}(x_{k+1}))+d_{k+1}+\xi_k\nonumber
    \end{aligned}
\end{equation}
Note that $\xi_k=\frac{B_k}{\beta_k}\langle\nabla f(y_0)-\nabla f_{I_k}(y_0),x_{k+1}-z_k\rangle$, and $\mathbb{E}\left[\xi_k\right]=0$. The first inequality follows from the $\mu$-strong convexity of function $h$; the second inequality follows from \eqref{slackEOGa} (Lemma \ref{Existenceofalpha}), and the last inequality follows from the triangle inequality with the relation $\alpha_k/\beta_k \leq 1/2$. Then we take the expectation with respect to the history of random variable $I_{i_j}$ with its size equal to $b_i$, where $i=0,1,...,t-1$ and $j=1,2,...,\tbinom{n}{b_i}$.
\begin{equation}
    \begin{aligned}
    \mathbb{E}\left[E_{k+1}-E_k\right]&=\frac{\alpha_k B_k}{\beta_k}D_h(x^*,x_{k+1})+\frac{B_k}{\beta_k}\langle\nabla f(x_{k+1}),x^*-x_{k+1}\rangle+\frac{B_k}{\bar{\mu}\beta_k^2}\mathbb{E}_{I_k}\left[\|\nabla_k\|^2\right]+\frac{cB_k}{\beta_k}(f(y_k)-f(x_{k+1}))\\
    &\quad+\bar{a}_k(f(x_{k+1})-f(x^*))+A_k(f(x_{k+1})-f(y_k))+A_{k+1}\mathbb{E}\left[f(y_{k+1})-f(x_{k+1})\right]\\
    &\leq\frac{(\alpha_k-\gamma\mu)B_k}{\beta_k}D_h(x^*,x_{k+1})+\frac{\gamma B_k}{\beta_k}(f(x^*)-f(x_{k+1}))+\frac{cB_k}{\beta_k}(f(y_k)-f(x_{k+1}))\\
    &\quad+\frac{B_k}{\bar{\mu}\beta_k^2}\mathbb{E}_{I_k}\left[\|\nabla_k\|^2\right]+\bar{a}_k(f(x_{k+1})-f(x^*))+A_k(f(x_{k+1})-f(y_k))+A_{k+1}\mathbb{E}\left[f(y_{k+1})-f(x_{k+1})\right]\\
    &\leq\left(\frac{\gamma B_k}{\beta_k}-\bar{a}_k\right)(f(x^*)-f(x_{k+1}))+\frac{B_k}{\bar{\mu}\beta_k^2}\mathbb{E}_{I_k}\left[\|\nabla_k-\nabla f(x_{k+1})\|^2\right]+\frac{B_k}{\bar{\mu}\beta_k^2}\|\nabla f(x_{k+1})\|^2\\
    &\quad+\frac{A_{k+1}}{2L}\mathbb{E}_{I_k}\left[\|\nabla_k-\nabla f(x_{k+1})\|^2\right]-\frac{A_{k+1}}{2L}\|\nabla f(x_{k+1})\|^2\\
    &\leq\left(\frac{\gamma B_k}{\beta_k}-\bar{a}_k\right)(f(x^*)-f(x_{k+1}))+\frac{A_{k+1}}{L}\mathbb{E}_{I_k}\left[\|\nabla_k-\nabla f(x_{k+1})\|^2\right]\\
    &\leq\bar{a}_k(f(x^*)-f(x_{k+1}))+4A_{k+1}\delta_k(f(x_{k+1})-f(x^*))+4A_{k+1}\delta_k(f(y_0)-f(x^*))\nonumber
    \end{aligned}
\end{equation}
The first inequality follows from the uniform quasar-convexity of function $f$; the second inequality follows from \eqref{fy-fx}; the third inequality follows from the relation $B_k/\bar{\mu}\beta_k^2\leq A_{k+1}/2L$, and the last inequality follows from Proposition 3.1 \eqref{varupperboundproof}. We testify the correctness of relations above. With the choice of $A_k,B_k,\alpha_k,\beta_k$ and Observation \ref{lowerboundofkappa}, we have
\begin{equation*}
    \frac{\alpha_k}{\beta_k}=\frac{\bar{b}_k}{B_k}=\frac{\gamma}{\sqrt{8\kappa}}\leq\frac{\gamma}{\sqrt{8}}\sqrt{\frac{2-\gamma}{\gamma}}\leq\frac{\gamma}{\sqrt{8}}\sqrt{\frac{1}{\gamma^2}}\leq\frac{1}{2}.
\end{equation*}
\begin{equation*}
\frac{B_k}{\bar{\mu}\beta_k^2}=\frac{4(A_{k+1}-A_k)^2}{\gamma^2\bar{\mu}\mu A_k}=\frac{4A_k(\gamma/\sqrt{8\kappa})^2}{\gamma^2\bar{\mu}\mu}=\frac{A_k}{2L}\leq\frac{A_{k+1}}{2L}
\end{equation*}
Next we prove $4A_{k+1}\delta_k\leq \bar{a}_k/2$. It suffices to prove $\delta_k\leq\bar{a}_k/8A_{k+1}$.
\begin{equation*}
    \delta_k=\frac{n-b_k}{b_k(n-1)}\leq\frac{\gamma}{8(\sqrt{8\kappa}+\gamma)}=\frac{\bar{a}_k}{8A_{k+1}}
\end{equation*}
Thus we obtain the final bound of $\mathbb{E}[E_{k+1}-E_k]$.
\begin{equation*}
\begin{aligned}
    \mathbb{E}[E_{k+1}-E_k]&\leq\bar{a}_k(f(x^*)-f(x_{k+1}))+4A_{k+1}\delta_k(f(x_{k+1})-f(x^*))+4A_{k+1}\delta_k(f(y_0)-f(x^*))\\
    &\leq\frac{\bar{a}_k}{2}(f(y_0)-f(x^*))
\end{aligned}
\end{equation*}
Summing both sides of the inequality above, we obtain
\begin{equation*}
\begin{aligned}
    \mathbb{E}[f(y_t)-f(x^*)]&\leq \left(1+\frac{\gamma}{\sqrt{8\kappa}}\right)^{-t}(f(y_0)-f(x^*)+\mu D_h(x^*,z_0))+\frac{1}{2}(f(y_0)-f(x^*))\\
    &\leq\left(1+\frac{\gamma}{\sqrt{8\kappa}}\right)^{-t}\frac{2}{\gamma}(f(y_0)-f(x^*))+\frac{1}{2}(f(y_0)-f(x^*))
\end{aligned}
\end{equation*}
When $t\geq\log_{1+\gamma/\sqrt{8\kappa}}(2/\gamma q)$, we have
\begin{equation*}
    \mathbb{E}[f(y_t)-f(x^*)]\leq\left(q+\frac{1}{2}\right)(f(y_0)-f(x^*)),
\end{equation*}
where $q+1/2<1$ with the choice of $q$. Next we conclude the convergence rate of global stages. Suppose $y_{s}$ is the output of stage $s$. If $t\geq\log_{1+\gamma/\sqrt{8\kappa}}(2/\gamma q)$ at each stage, we have
\begin{equation*}
    \mathbb{E}[f(y_{s})-f(x^*)]\leq\left(q+\frac{1}{2}\right)^{s+1}(f(y_0)-f(x^*))=\left(q+\frac{1}{2}\right)^{s+1}\mathcal{E}_0.
\end{equation*}
For $\mu>0$ (Option \rom{2}), our analysis is identical to case 1.
\section{Proofs of Complexity}
\label{proofofcomplexity}
In this section, we analyze the overall complexity of QASGD and QASVRG. As the complexity of QAGD has been analyzed in \citet{hinder}, we will not provide the related proof. Analogously, we apply the same metric as \citet{hinder} propose, which is the total number of function and gradient evaluations. Roughly speaking, the overall complexity is the multiplication of the number of iterations and the number of function and gradient evaluations per iteration. We consider the worst case of Bisearch where $0,1$ and "guess" do not meet our conditions, and we need to do line search at each iteration. In this situation, we access the gradient (estimate) in Bisearch that can be directly used in the subsequent updates at each iteration, i.e., no additional access to $\nabla_k$ is required in the mirror descent step and the gradient descent step. The number of $f_i$ involved in Bisearch affects the complexity. While QAGD needs all $f_i$ in Bisearch, QASGD and QASVRG only need a single $f_i$ and a mini-batch of $f_i$ respectively. Next we present the analysis of the complexity of Bisearch per iteration for QASGD and QASVRG in the Euclidean setting where $D_h(x,y)=\frac{1}{2}\|x-y\|^2$, and present the proof of Theorem 3.5 and Theorem 3.6.  Lemma \ref{complexityofBi} implies that Bisearch needs $O\left(\log^+\left((1+c)\min\left\{\frac{L\|y_k-z_k\|^2}{\tilde{\epsilon}},\frac{L^3}{b^3}\right\}\right)\right)$ function and gradient evaluations per iteration for a single function. As $\tilde{\epsilon}$ and $b$ can not be simultaneously non-zero, we have two different situations corresponding to different complexity. 
\begin{equation*}
    O\left(\log^+\left((1+c)\min\left\{\frac{L\|y_k-z_k\|^2}{\tilde{\epsilon}},\frac{L^3}{b^3}\right\}\right)\right)=\begin{cases}
    O\left(\log^+\left((1+c)\frac{L^3}{b^3}\right)\right)\quad &\mu>0,\\
    O\left(\log^+\left((1+c)\frac{L\|y_k-z_k\|^2}{\tilde{\epsilon}}\right)\right)\quad&\mu=0.
\end{cases}
\end{equation*}
When $\mu=0$, the key to our proof is to bound $\|y_k-z_k\|$.
\subsection{Proof of Corollary \ref{complexASGD}}
\textbf{Case 1: }$\mu=0$

By the proof of 3.2, we have $\mathbb{E}[E_{k+1}-E_k]\leq\frac{\sigma^2}{2\beta_k^2}+\frac{\bar{a}_k\epsilon}{2}$. Assuming $\|x^*-z_0\|\leq R$, we have the following relation.
\begin{equation*}
    \begin{aligned}
    \frac{1}{2}\mathbb{E}\|x^*-z_k\|^2&\leq A_0(f(x_0)-f(x^*))+\frac{1}{2}\|x^*-z_0\|^2+\sum_{j=0}^{k-1}\frac{\sigma^2}{2\beta_j^2}+\frac{A_k\epsilon}{2}\\
    &\leq\|x^*-z_0\|^2+\sum_{j=0}^{k-1}\frac{\sigma^2\eta^2}{2\gamma^2}(2j+3)^2+\frac{\eta\epsilon}{2}(k+1)^2\\
    &\leq 5R^2+\frac{(k+1)^2}{2L}\epsilon
    \end{aligned}
\end{equation*}
The third inequality follows from $\eta=\min\left(\frac{1}{L},\sqrt{\frac{2\|x^*-z_0\|^2}{\sigma^2}}\frac{\gamma}{(t+1)^{3/2}}\right)$. Combining the analysis above, we have $\mathbb{E}\|x^*-z_k\|^2\leq 10R^2+\frac{(k+1)^2}{L}\epsilon$ and $\mathbb{E}\|x^*-z_k\|\leq\sqrt{\mathbb{E}\|x^*-z_k\|^2}\leq\sqrt{10R^2+\frac{(k+1)^2}{L}\epsilon}$ by Jensen's inequality. Thus we obtain
\begin{equation*}
    \mathbb{E}\|z_k-z_{k-1}\|=\mathbb{E}\left\|\frac{1}{\beta_{k-1}}\nabla_{k-1}\right\|\leq\mathbb{E}\|x^*-z_k\|+\mathbb{E}\|x^*-z_{k-1}\|\leq 2\sqrt{10R^2+\frac{(k+1)^2}{L}\epsilon}.
\end{equation*}
As the stepsize of gradient descent is $0$, $y_k$ is identical to $x_k$, and we solely need to bound $\|x_k-z_k\|$. By the definition of $z_k$ and $x_k$, we have
\begin{equation*}
    \begin{aligned}
    \mathbb{E}\|x_k-z_k\|&=\mathbb{E}\left\|(1-\tau_{k-1})z_{k-1}+\tau_{k-1}x_{k-1}-z_{k-1}+\frac{1}{\beta_{k-1}}\nabla_{k-1}\right\|\\
    &=\mathbb{E}\left\|\tau_{k-1}(x_{k-1}-z_{k-1})+\frac{1}{\beta_{k-1}}\nabla_{k-1}\right\|\\
    &\leq \tau_{k-1}\mathbb{E}\|x_{k-1}-z_{k-1}\|+\mathbb{E}\left\|\frac{1}{\beta_{k-1}}\nabla_{k-1}\right\|\\
    &\leq \mathbb{E}\|x_{k-1}-z_{k-1}\|+2\sqrt{10R^2+\frac{(k+1)^2}{L}\epsilon}\\
    &\leq \mathbb{E}\|x_{k-1}-z_{k-1}\|+2\sqrt{10}R+2(k+1)\sqrt{\frac{\epsilon}{L}}
    \end{aligned}
\end{equation*}
The last inequality follows from the relation $\sqrt{a^2+b^2}\leq a+b$ for $a,b\geq  0$. By the proof of Theorem 3.2, we have
\begin{equation*}
    \mathbb{E}[f(x_t)-f(x^*)]\leq \frac{LR^2}{(t+1)^2}+\frac{2\sigma }{\gamma}\frac{R}{\sqrt{2(t+1)}}\leq\frac{LR^2+\frac{\sqrt{2}\sigma R}{\gamma}}{\sqrt{t+1}}.
\end{equation*}
Suppose $k\leq k_{\max}=\left\lfloor\frac{4\left(LR^2+\frac{\sqrt{2}\sigma R}{\gamma}\right)^2}{\epsilon^2}\right\rfloor$, and we obtain
\begin{equation*}
\begin{aligned}
    \mathbb{E}\|x_k-z_k\|&\leq 2\sqrt{10}Rk+\sqrt{\frac{\epsilon}{L}}k(k+3)\\
    &\leq 2\sqrt{10}Rk+4\sqrt{\frac{\epsilon}{L}}k^2\\
    &\leq 8\sqrt{10}R \frac{\left(LR^2+\frac{\sqrt{2}\sigma R}{\gamma}\right)^2}{\epsilon^2}+\frac{64\left(LR^2+\frac{\sqrt{2}\sigma R}{\gamma}\right)^4}{\sqrt{L}\epsilon^{\frac{7}{2}}}\\
    &\leq 8\sqrt{10}R\frac{(LR^2+\sqrt{2}\sigma R)^2}{\gamma^2\epsilon^2}+\frac{64\left(LR^2+{\sqrt{2}\sigma R}\right)^4}{\sqrt{L}\gamma^4\epsilon^{\frac{7}{2}}}\\
    &=O\left(\frac{L^{7/2}R^8}{\gamma^4\epsilon^{7/2}}\right)
\end{aligned}
\end{equation*}
The first inequality follows from the relation $k+3\leq 4k$ for $k\geq 1$. By Markov's inequality, $P_r(\|x_k-z_k\|\geq k_{\max}\mathbb{E}\|x_k-z_k\|)\leq\frac{1}{k_{\max}}$, which implies $\|x_k-z_k\|\leq k_{\max}\mathbb{E}\|x_k-z_k\|\leq O\left(\frac{L^{11/2}R^{12}}{\gamma^6\epsilon^{11/2}}\right)$ with probability at least $1-\frac{1}{k_{\max}}$. Thus we have $\frac{L\|x_k-z_k\|^2}{\tilde{\epsilon}}\leq O\left(\frac{L^{12}R^{24}}{\gamma^{13}\epsilon^{12}}\right)$. Besides, we have $c+1=\frac{\gamma(k+1)^2}{2k+3}+1\leq\gamma k+2\leq\frac{4(LR^2+\sqrt{2}\sigma R)^2}{\gamma\epsilon^2}+2=O\left(\frac{L^2R^4}{\gamma\epsilon^2}\right)$. Then we obtain the following upper bound of the term inside $\log^+()$:
\begin{equation*}
    \begin{aligned}
    (1+c)\min\left\{\frac{L\|x_k-z_k\|^2}{\tilde{\epsilon}},\frac{L^3}{b^3}\right\}&=(1+c)\frac{L\|x_k-z_k\|^2}{\tilde{\epsilon}}\\
    &\leq O\left(\frac{L^{14}R^{28}}{\gamma^{14}\epsilon^{14}}\right)
    \end{aligned}
\end{equation*}
Thus we have $O\left(\log^+\left((1+c)\min\left\{\frac{L\mathbb{E}\|y_k-z_k\|^2}{\tilde{\epsilon}},\frac{L^3}{b^3}\right\}\right)\right)\leq O(\log^{+}(LR^2\gamma^{-1}\epsilon^{-1}))$.
In the case of $\mu=0$, QASGD needs $O\left(\sqrt{\frac{LR^2}{\epsilon}}+\frac{\sigma^2R^2}{\gamma^2\epsilon^2}\right)$ iterations to generate an $\epsilon$-approximate solution under expectation. Therefore, the overall complexity of QASGD is upper bounded by $(\mu=0)$ is $O\left(\left(\sqrt{\frac{LR^2}{\epsilon}}+\frac{\sigma^2R^2}{\gamma^2\epsilon^2}\right)\log^+\left(\frac{LR^2}{\gamma\epsilon}\right)\right)$ with high probability.

\textbf{Case 2: }$\mu>0$

As $\tilde{\epsilon}=0$ and $b>0$, we have $O\left(\log^+\left((1+c)\min\left\{\frac{L\|y_k-z_k\|^2}{\tilde{\epsilon}},\frac{L^3}{b^3}\right\}\right)\right)=O\left(\log^+\left((1+c)\frac{L^3}{b^3}\right)\right)$. In Step 1, $b=\frac{\gamma\mu}{4}$ and $c=\frac{8}{\gamma}$. Then we have
\begin{equation*}
    (1+c)\frac{L^3}{b^3}\leq\frac{9}{\gamma}\frac{64L^3}{\gamma^3\mu^3}=\frac{576\kappa^3}{\gamma^4}
\end{equation*}
Thus we have $O\left(\log^+\left((1+c)\min\left\{\frac{L\mathbb{E}\|y_k-z_k\|^2}{\tilde{\epsilon}},\frac{L^3}{b^3}\right\}\right)\right)=O\left(\log^+\left(\frac{\kappa^{3/4}}{\gamma}\right)\right)$. By the proof of Theorem 3.2, QASGD needs $O\left(\frac{1}{\gamma^2}\log\left(\frac{f(x_0)-f(x^*))}{\gamma\epsilon}\right)\right)$ iterations in Step 1, and the complexity of Step 1 is $O\left(\frac{1}{\gamma^2}\log\left(\frac{f(x_0)-f(x^*))}{\gamma\epsilon}\right)\log^+\left(\frac{\kappa^{3/4}}{\gamma}\right)\right)$. In step 2, $b=\frac{\gamma\mu}{4}$ and $c=\frac{\gamma(k+48/\gamma^2)^2}{2(2k+96/\gamma^2+1)}\leq\frac{\gamma(k+48/\gamma^2)}{2}\leq\frac{k+48}{2\gamma}$. We can upper bound the convergence rate of Step 2:
\begin{equation*}
    \mathbb{E}[f(x_t)-f(x^*)]\leq \frac{9\sigma^2}{\gamma^2\mu(t+48/\gamma^2)^2}+\frac{36\sigma^2}{\gamma^2\mu(t+48/\gamma^2)}\leq\frac{45\sigma^2}{\gamma^2\mu (t+48/\gamma^2)}.
\end{equation*}
Suppose $k\leq\left\lfloor\frac{90\sigma^2}{\gamma^2\mu\epsilon}\right\rfloor$. We have $c+1\leq\frac{C_3}{\gamma^3\epsilon}$, where $C_3=\frac{45\sigma^2}{\mu}+25$.
\begin{equation*}
    (1+c)\frac{L^3}{b^3}\leq\frac{C_3}{\gamma^3\epsilon}\frac{64\kappa^3}{\gamma^3}=\frac{64C_3\kappa^3}{\gamma^6\epsilon}
\end{equation*}
Thus $O\left(\log^+\left((1+c)\min\left\{\frac{L\|y_k-z_k\|^2}{\tilde{\epsilon}},\frac{L^3}{b^3}\right\}\right)\right)=O\left(\log^+\left(\frac{\kappa^{2/3}}{\gamma\epsilon^{1/6}}\right)\right)$. QASGD needs $O\left(\frac{\sigma^2}{\gamma^2\epsilon}\right)$ iterations in Step 2, and the complexity of Step 2 is $O\left(\frac{\sigma^2}{\gamma^2\epsilon}\log^+\left(\frac{\kappa^{2/3}}{\gamma\epsilon^{1/6}}\right)\right)$. In summary, the overall complexity of QASGD $(\mu>0)$ is upper bounded by $O\left(\frac{1}{\gamma^2}\log\left(\frac{f(x_0)-f(x^*))}{\gamma\epsilon}\right)\log^+\left(\frac{\kappa^{3/4}}{\gamma}\right)+\frac{\sigma^2}{\gamma^2\epsilon}\log^+\left(\frac{\kappa^{2/3}}{\gamma\epsilon^{1/6}}\right)\right)$.
\subsection{Proof of Corollary \ref{complexASVRG}}
\textbf{Case 1: }$\mu=0$

By the proof of Theorem 3.3, we have $\mathbb{E}[E_{k+1}-E_k]\leq\bar{a}_k\left(\frac{1}{2}+\epsilon\right)f(y_0)$. Assuming $\|x^*-z_0\|\leq R$, we have the following relation.
\begin{equation*}
    \begin{aligned}
    \frac{1}{2}\mathbb{E}\|x^*-z_k\|^2&\leq A_0(f(y_0)-f(x^*))+\frac{1}{2}\|x^*-z_0\|^2+A_k\left(\frac{1}{2}+\epsilon\right)f(y_0)\\
    &\leq \|z_0-x^*\|^2+\frac{3\gamma^2}{32L}(k+1)^2f(y_0)+\frac{\gamma^2}{16L}(k+1)^2f(y_0)\epsilon\\
    &\leq R^2+\frac{3\gamma^2}{32L}(k+1)^2f(y_0)+\frac{\gamma^2}{16L}(k+1)^2f(y_0)\epsilon
    \end{aligned}
\end{equation*}
Thus we have $\mathbb{E}\|x^*-z_k\|^2\leq2R^2+\frac{3\gamma^2}{16L}(k+1)^2f(y_0)+\frac{\gamma^2}{8L}(k+1)^2f(y_0)\epsilon$ and by Jensen's inequality
\begin{equation*}
\begin{aligned}
    \mathbb{E}\|\nabla_{k-1}\|=\beta_{k-1}\mathbb{E}\|z_k-z_{k-1}\|&\leq\beta_{k-1}\mathbb{E}(\|x^*-z_k\|+\|x^*-z_{k-1}\|)\\
    &\leq\frac{16L}{\gamma(2k+1)}\sqrt{2R^2+\frac{3\gamma^2}{16L}(k+1)^2f(y_0)+\frac{\gamma^2}{8L}(k+1)^2f(y_0)\epsilon}.
    \end{aligned}
\end{equation*}
By the definition of $y_k$ and $z_k$, we have
\begin{equation*}
    \begin{aligned}
    \mathbb{E}\|y_k-z_k\|&=\mathbb{E}\left\|x_k-\frac{1}{L}\nabla_{k-1}-z_{k-1}+\frac{1}{\beta_{k-1}}\nabla_{k-1}\right\|\\
    &=\mathbb{E}\left\|(1-\tau_{k-1})z_{k-1}+\tau_{k-1}y_{k-1}-\frac{1}{L}\nabla_{k-1}-z_{k-1}+\frac{1}{\beta_{k-1}}\nabla_{k-1}\right\|\\
    &\leq\tau_{k-1}\mathbb{E}\|y_{k-1}-z_{k-1}\|+\left|\frac{1}{\beta_{k-1}}-\frac{1}{L}\right|\mathbb{E}\|\nabla_{k-1}\|\\
    &\leq\mathbb{E}\|y_{k-1}-z_{k-1}\|+\left(\frac{1}{\beta_{k-1}}+\frac{1}{L}\right)\mathbb{E}\|\nabla_{k-1}\|\\
    &\leq\mathbb{E}\|y_{k-1}-z_{k-1}\|+\frac{2k+9}{8L}\mathbb{E}\|\nabla_{k-1}\|\\
    &\leq\mathbb{E}\|y_{k-1}-z_{k-1}\|+\frac{8}{\gamma}\sqrt{2R^2+\frac{3\gamma^2}{16L}(k+1)^2f(y_0)+\frac{\gamma^2}{8L}(k+1)^2f(y_0)\epsilon}\\
    &\leq\mathbb{E}\|y_{k-1}-z_{k-1}\|+\frac{8\sqrt{2}R}{\gamma}+2\sqrt{3}(k+1)\sqrt{\frac{f(y_0)}{L}}+2\sqrt{2}(k+1)\sqrt{\frac{f(y_0)\epsilon}{L}}
    \end{aligned}
\end{equation*}
Suppose $f(y_0)\geq\epsilon$ and $k\leq \left\lfloor\sqrt{\frac{17LR^2}{2\gamma^2qf(y_0)}}\right\rfloor\leq k_{\max}=\left\lfloor\sqrt{\frac{17LR^2}{2\gamma^2q\epsilon}}\right\rfloor$, and we obtain
\begin{equation*}
    \begin{aligned}
    \mathbb{E}\|y_k-z_k\|&\leq\frac{8\sqrt{2}Rk}{\gamma}+\sqrt{3}k(k+3)\sqrt{\frac{f(y_0)}{L}}+\sqrt{2}k(k+3)\sqrt{\frac{f(y_0)\epsilon}{L}}\\
    &\leq\frac{8\sqrt{2}Rk}{\gamma}+4\sqrt{3}k^2\sqrt{\frac{f(y_0)}{L}}+4\sqrt{2}k^2\sqrt{\frac{f(y_0)\epsilon}{L}}\\
    &\leq\frac{8\sqrt{17q}R^2}{\gamma^2q}\sqrt{\frac{L}{f(y_0)}}+\frac{34\sqrt{3}R^2}{\gamma^2q}\sqrt{\frac{L}{f(y_0)}}+\frac{34\sqrt{2}R^2}{\gamma^2 q}\sqrt{\frac{L\epsilon}{f(y_0)}}\\
    &\leq\frac{108R^2L^{1/2}}{\gamma^2\epsilon^{1/2}q}+\frac{34\sqrt{2}R^2L^{1/2}}{\gamma^2q}\\
    &\leq\frac{176R^2L^{1/2}}{\gamma^2\epsilon^{1/2}q}=O\left(\frac{L^{1/2}R^2}{q\gamma^2\epsilon^{1/2}}\right)
    \end{aligned}
\end{equation*}
By Markov's inequality, $P_r(\|y_k-z_k\|\geq k_{\max}\mathbb{E}\|y_k-z_k\|)\leq\frac{1}{k_{\max}}$, which implies $\|y_k-z_k\|\leq k_{\max}\mathbb{E}\|x_k-z_k\|\leq O\left(\frac{LR^{3}}{q^{3/2}\gamma^3\epsilon}\right)$ with probability at least $1-\frac{1}{k_{\max}}$. Thus we have $\frac{L\|y_k-z_k\|^2}{\tilde{\epsilon}}\leq O\left(\frac{L^3R^6}{q^3\gamma^7\epsilon^4}\right)$. Besides, we have $c+1=\frac{\gamma(k+1)^2}{2(2k+3)}+1\leq\frac{\gamma k}{2}+\frac{3}{2}\leq\frac{R}{2}\sqrt{\frac{17L}{2q\epsilon}}+\frac{3}{2}=O\left(\frac{L^{1/2}R}{q^{1/2}\epsilon^{1/2}}\right)$. Then we obtain the following upper bound of the term inside $\log^+()$:
\begin{equation*}
    \begin{aligned}
    (1+c)\min\left\{\frac{L\|y_k-z_k\|^2}{\tilde{\epsilon}},\frac{L^3}{b^3}\right\}&=(1+c)\frac{L\|y_k-z_k\|^2}{\tilde{\epsilon}}\leq O\left(\frac{L^{7/2}R^7}{q^{7/2}\gamma^7\epsilon^{9/2}}\right)
    \end{aligned}
\end{equation*}
Thus we have $O\left(\log^+\left((1+c)\min\left\{\frac{L\|y_k-z_k\|^2}{\tilde{\epsilon}},\frac{L^3}{b^3}\right\}\right)\right)\leq O\left(\log^{+}\left(\frac{L^{1/2}R}{q^{1/2}\gamma\epsilon^{9/14}}\right)\right)$ with high probability. As we need to access the full gradient and function value evaluated at $y_0$ per stage and the gradient and function value of mini-batch to calculate SVRG and $\tilde{\epsilon}$, the overall complexity of QASVRG $(\mu=0)$ to generate an $\epsilon$-approximate solution is
\begin{equation*}
\begin{aligned}
    &O\left(\left(2n+\sum_{k=0}^{t-1}b_k\log^{+}\left(\frac{L^{1/2}R}{q^{1/2}\gamma\epsilon^{9/14}}\right)\right)\log\left(\frac{1}{\epsilon}\right)\right)\\
    &=O\left(\left(2n+\sum_{k=0}^{t-1}\frac{\gamma n(2k+3)}{2(n-1)p+\gamma(2k+3)}\log^{+}\left(\frac{L^{1/2}R}{q^{1/2}\gamma\epsilon^{9/14}}\right)\right)\log\left(\frac{1}{\epsilon}\right)\right)\\
    &\leq O\left(\left(2n+\frac{\gamma nt(2t+1)}{2(n-1)p+\gamma(2t+1)}\log^{+}\left(\frac{L^{1/2}R}{q^{1/2}\gamma\epsilon^{9/14}}\right)\right)\log\left(\frac{1}{\epsilon}\right)\right)\\
    &\leq O\left(\left(n+\frac{nLR^2}{\gamma\epsilon n+\gamma\sqrt{\epsilon LR^2}}\log^{+}\left(\frac{L^{1/2}R}{q^{1/2}\gamma\epsilon^{9/14}}\right)\right)\log\left(\frac{1}{\epsilon}\right)\right)
\end{aligned}
\end{equation*}
where $t=\left\lceil\sqrt{\frac{17L\|x^*-z_0\|^2}{2\gamma^2qf(y_0)}}\right\rceil\leq\sqrt{\frac{17LR^2}{2\gamma^2qf(y_0)}}$ is the maximum number of iterations per stage, and $O(\log(\epsilon^{-1}))$ is the number of stages. Note that $D_h(x^*,y_s)$ is uniformly bounded by $R^2$ under Assumption \ref{compactset}, which is in the bound above.

\textbf{Case 2: }$\mu>0$

For Option \rom{2}, we have $t=\left\lceil\sqrt{\frac{17L\|x^*-z_0\|^2}{2\gamma^2qf(y_0)}}\right\rceil\leq\left\lceil\sqrt{\frac{17L\|x^*-z_0\|^2}{2\gamma^2q\mathcal{E}_0}}\right\rceil\leq\sqrt{\frac{17L(2-\gamma)}{\gamma^3q\mu}}\leq\sqrt{\frac{34\kappa}{\gamma^3q}}$ using the last relation in Assumption \ref{Bregmanassump}. Thus the overall complexity of QASVRG (Option \rom{2}) to generate an $\epsilon$-approximate solution is
\begin{equation*}
    \begin{aligned}
    &O\left(\left(2n+\sum_{k=0}^{t-1}b_k\log^{+}\left(\frac{L^{1/2}R}{q^{1/2}\gamma\epsilon^{9/14}}\right)\right)\log\left(\frac{1}{\epsilon}\right)\right)\\
    &=O\left(\left(2n+\sum_{k=0}^{t-1}\frac{\gamma n(2k+3)}{2(n-1)p+\gamma(2k+3)}\log^{+}\left(\frac{L^{1/2}R}{q^{1/2}\gamma\epsilon^{9/14}}\right)\right)\log\left(\frac{1}{\epsilon}\right)\right)\\
    &\leq O\left(\left(2n+\frac{\gamma nt(2t+1)}{2(n-1)p+\gamma(2t+1)}\log^{+}\left(\frac{L^{1/2}R}{q^{1/2}\gamma\epsilon^{9/14}}\right)\right)\log\left(\frac{1}{\epsilon}\right)\right)\\
    &\leq O\left(\left(n+\frac{n\kappa}{\gamma^2 n+\gamma^{3/2}\sqrt{\kappa}}\log^{+}\left(\frac{L^{1/2}R}{q^{1/2}\gamma\epsilon^{9/14}}\right)\right)\log\left(\frac{1}{\epsilon}\right)\right)
    \end{aligned}
\end{equation*}
where $O(\log(\epsilon^{-1}))$ is the number of stages. For Option \rom{1}, $\tilde{\epsilon}=0,b=\frac{\gamma\mu}{4}$ and $c=\sqrt{2\kappa}$. Then we obtain the following relation:
\begin{equation*}
    (1+c)\min\left\{\frac{L\|y_k-z_k\|^2}{\tilde{\epsilon}},\frac{L^3}{b^3}\right\}=(1+c)\frac{L^3}{b^3}=(1+\sqrt{2\kappa})\frac{64L^3}{\gamma^3\mu^3}=(1+\sqrt{2\kappa})\frac{64\kappa^3}{\gamma^3}.
\end{equation*}
Thus we have $O\left(\log^+\left((1+c)\min\left\{\frac{L\|y_k-z_k\|^2}{\tilde{\epsilon}},\frac{L^3}{b^3}\right\}\right)\right)=O(\log^{+}(\kappa^{7/6}\gamma^{-1}))$. At each stage, QASVRG (Option \rom{2}) is run until $(1+\gamma/\sqrt{8\kappa})^{-t}\leq \frac{\gamma q}{2}$. Let $(1+\gamma/\sqrt{8\kappa})^{-t}\leq e^{-\frac{\gamma t}{\sqrt{8\kappa}+\gamma}}\leq\frac{\gamma q}{2}$, and we have $t\geq\frac{\sqrt{8\kappa}+\gamma}{\gamma}\log\left(\frac{2}{\gamma q}\right)$, and $t=\left\lceil\log_{1+\frac{\gamma}{\sqrt{8\kappa}}}\left(\frac{2}{\gamma q}\right)\right\rceil\leq\frac{\sqrt{8\kappa}+\gamma}{\gamma}\log\left(\frac{2}{\gamma q}\right)\leq\frac{5\sqrt{\kappa}}{\gamma}\log\left(\frac{2}{\gamma q}\right)$. Thus the overall complexity of QASVRG (Option \rom{1}) to generate an $\epsilon$-approximate solution is 
\begin{equation*}
    \begin{aligned}
    O\left(\left(n+\sum_{k=0}^{t-1}b_k\log^{+}\left(\frac{\kappa^{7/6}}{\gamma}\right)\right)\log\left(\frac{1}{\epsilon}\right)\right)&=O\left(\left(n+\sum_{k=0}^{t-1}\frac{8n(\sqrt{8\kappa}+\gamma)}{\gamma (n-1)+8(\sqrt{8\kappa}+\gamma)}\log^{+}\left(\frac{\kappa^{7/6}}{\gamma}\right)\right)\log\left(\frac{1}{\epsilon}\right)\right)\\
    &= O\left(\left(n+\frac{40nt\sqrt{\kappa}}{\gamma (n-1)+8(\sqrt{8\kappa}+\gamma)}\log^{+}\left(\frac{\kappa^{7/6}}{\gamma}\right)\right)\log\left(\frac{1}{\epsilon}\right)\right)\\
    &\leq O\left(\left(n+\frac{n\kappa}{\gamma^2 n+\gamma\sqrt{\kappa}}\log\left(\frac{2}{\gamma q}\right)\log^{+}\left(\frac{\kappa^{7/6}}{\gamma}\right)\right)\log\left(\frac{1}{\epsilon}\right)\right)
    \end{aligned}
\end{equation*}
where $O(\log(\epsilon^{-1}))$ is the number of stages.
\section{Theoretical Extension}
We analyze QASGD under a restrictive condition: strong growth condition (SGC), which is formally formulated in the following. This condition has been proposed in \citet{schmidt2013fast}, \citet{vaswani2019fast}, and \citet{gower2021}. \citet{schmidt2013fast} derive optimal convergence rates for SGD under SGC for convex and strongly convex functions.
\begin{assumption}[SGC]
\label{SGC}
    Suppose $i$ is sampled i.i.d from $[n]$. For some constant $\rho$ and $x^*\in\mathcal{X}^*$, we have
    \begin{equation*}
        E_i\left[\|\nabla f_i(x)\|^2\right]\leq \rho \|\nabla f(x)\|^2.
    \end{equation*}
\end{assumption}
If $\nabla f(x)=0$, then $\nabla f_i(x)=0$ under SGC, which implies the interpolation assumption. We derive better convergence rates for QASGD under SGC for $f\in\mathcal{Q}_{\mu\gamma}$.
\begin{theorem}[QASGD under SGC]
\label{convergencerateAGD}
Suppose Assumption \ref{Bregmanassump} and Assumption \ref{SGC} hold, $D_h(x^*,z_0)\leq R^2$, $f\in\mathcal{F}_L$, and choose any $\tilde{y}_0\in\mathbb{R}^d$. Then Algorithm \ref{algorithm_for_quasar} with the choices of $\nabla_k=\nabla f_i(x_{k+1})$ and $A_k, B_k, \theta_k$ specified in Table \ref{parameterofQASGDSGC} satisfies
\begin{equation}
\label{QAGDdiff}
    \mathbb{E}[E_{k+1}-E_k]\leq\left\{\begin{aligned}
    &\frac{\Bar{a}_k\epsilon}{2}, &\mu=0,\\
    &0, &\mu>0.
\end{aligned}
\right.
\end{equation}
Summing both sides of \eqref{QAGDdiff}, we conclude the convergence rate as follows:
\begin{equation}
    \mathbb{E}[f(y_t)-f(x^*)]\simeq\left\{\begin{aligned}
    &\frac{L\rho R^2}{\gamma^2t^2}+\frac{\epsilon}{2}, &\mu=0,\\
    &\left(1+\frac{\gamma}{2\rho\sqrt{\kappa}}\right)^{-t}E_0, &\mu>0.
\end{aligned}
\right.
\end{equation}
\end{theorem}
\begin{table}[H]
\begin{center}
\renewcommand\arraystretch{1.5}
    \begin{tabular}{|c|}
    \hline
    \rowcolor{black!20}
    \textbf{QASGD under SGC}\\
    \hline
    \rowcolor{black!10}
    \textbf{$\gamma$-quasar-convex} $(\mu=0)$\\
    \hline
    $A_k=\frac{\bar{\mu}\gamma^2}{4\rho L}(k+1)^2$, $B_k=1$\\
    $(\alpha_k, \beta_k, \rho_k, \tilde{f},b,c,\tilde{\epsilon})\gets\left(0,\frac{\gamma}{\bar{a}_k},\frac{1}{\rho L}, f, 0, \frac{\gamma A_k}{\bar{a}_k}, \frac{\gamma\epsilon}{2}\right)$\\
    \hline
    \rowcolor{black!10}
    \textbf{$\mu$-strongly $\gamma$-quasar-convex} $(\mu>0)$\\
    \hline
    $A_k=(1+{\gamma}/2\rho{\sqrt{\kappa}})^{k}$, $B_k=\mu A_k$\\
    $(\alpha_k, \beta_k, \rho_k, \tilde{f},b,c,\tilde{\epsilon})\gets\left(\gamma\mu, \frac{\gamma\mu B_k}{\bar{b}_k}, \frac{1}{\rho L}, f, \frac{\gamma\bar{\mu}\mu}{2}, \frac{\gamma A_k}{\bar{a}_k},0\right)$\\
    \hline
    \end{tabular}
    \caption{Parameter choices for QASGD under SGC}
    \label{parameterofQASGDSGC}
    \end{center}
\end{table}
\begin{proof}
\begin{equation}
    \begin{aligned}
    E_{k+1}-E_k
    &\overset{\eqref{TPI}}{=}-\left\langle{\nabla h(z_{k+1})-\nabla h(z_k)},x^*-z_{k+1}\right\rangle-D_h(z_{k+1},z_k)+A_{k+1}(f(y_{k+1})-f(x^*))-A_k(f(y_k)-f(x^*))\\
    &=\frac{1}{\beta_k}\left\langle\nabla_k,x^*-z_{k+1}
    \right\rangle-D_h(z_{k+1},z_k)+{A_{k+1}}(f(y_{k+1})-f(x^*))-{A_k}(f(y_k)-f(x^*))\\
    &\overset{\eqref{st_h}}{\leq}\frac{1}{\beta_k}\left\langle\nabla_k,x^*-z_{k}
    \right\rangle+\frac{1}{\beta_k}\left\langle\nabla_k,z_k-z_{k+1}
    \right\rangle-\frac{\bar{\mu}}{2}\|z_{k+1}-z_k\|^2\\
    &\quad+{A_{k+1}}(f(y_{k+1})-f(x^*))-{A_k}(f(y_k)-f(x^*))\\
    &\overset{\eqref{fenchelyoung}}{\leq}\frac{1}{\beta_k}\left\langle\nabla_k,x^*-z_{k}
    \right\rangle+\frac{1}{2\bar{\mu}\beta_k^2}\|\nabla_k\|^2+{A_{k+1}}(f(y_{k+1})-f(x^*))-{A_k}(f(y_k)-f(x^*))\\
    &=\frac{1}{\beta_k}\langle\nabla_k,x^*-x_{k+1}\rangle+\frac{\tau_k}{\beta_k}\langle\nabla_k,y_k-z_k\rangle+\frac{1}{2\bar{\mu}\beta_k^2}\|\nabla_k\|^2+A_{k+1}(f(y_{k+1})-f(x_{k+1}))\\
    &\quad+(A_{k+1}-A_k)(f(x_{k+1})-f(x^*))+A_k(f(x_{k+1})-f(y_k))\nonumber
    \end{aligned}
\end{equation}
\begin{equation*}
\begin{aligned}
    \mathbb{E}[E_{k+1}-E_k]&\leq \frac{\gamma}{\beta_k}(f(x^*)-f(x_{k+1}))+\frac{1}{\beta_k}(c(f(y_k)-f(x_{k+1}))+\tilde{\epsilon})+\frac{1}{2\bar{\mu}\beta_k^2}\mathbb{E}\left[\|\nabla_k\|^2\right]\\
    &\quad+A_{k+1}\mathbb{E}[f(y_{k+1})-f(x_{k+1})]+(A_{k+1}-A_k)(f(x_{k+1})-f(x^*))+A_k(f(x_{k+1})-f(y_k))\\
    &\leq\frac{\rho}{2\bar{\mu}\beta_k^2}\|\nabla f(x_{k+1})\|^2+A_{k+1}\mathbb{E}[f(y_{k+1})-f(x_{k+1})]+\frac{(A_{k+1}-A_k)}{2}\epsilon\\
    &\overset{\eqref{lsmooth}}{\leq}\left(\frac{\rho}{2\bar{\mu}\beta_k^2}-\frac{A_{k+1}}{2\rho L}\right)\|\nabla f(x_{k+1})\|^2+\frac{(A_{k+1}-A_k)}{2}\epsilon\leq\frac{A_{k+1}-A_k}{2}\epsilon
\end{aligned}
\end{equation*}
\textbf{Case 2:} $\mu>0$
\begin{equation*}
    \begin{aligned}
    E_{k+1}-E_k &\overset{\eqref{TPI}}{=} \bar{b}_k D_h(x^*,z_{k+1})-B_k\langle\nabla h(z_{k+1})-\nabla h(z_k),x^*-z_{k+1}\rangle-B_k D_h(z_{k+1},z_k)+ A_{k+1}(f(y_{k+1})-f(x^*))\\
    &\quad -A_k(f(y_{k})-f(x^*))\\
    &=\bar{b}_k D_h(x^*,z_{k+1})+\frac{\alpha_k B_k}{\beta_k}\langle\nabla h(z_{k+1})-\nabla h(x_{k+1}), x^*-z_{k+1}\rangle+\frac{B_k}{\beta_k}\langle\nabla_k, x^*-z_{k+1}\rangle-B_k D_h(z_{k+1},z_k)\\
    &\quad+A_{k+1}(f(y_{k+1})-f(x^*))-A_k(f(y_{k})-f(x^*))\\
    &\overset{\eqref{TPI}}{=} \left(\bar{b}_k-\frac{\alpha_k B_k}{\beta_k}\right)D_h(x^*,z_{k+1})+\frac{\alpha_k B_k}{\beta_k}\left(D_h(x^*,x_{k+1})-D_h(z_{k+1},x_{k+1})\right)+\frac{B_k}{\beta_k}\langle\nabla_k,x^*-z_k\rangle\\
    &\quad+\frac{B_k}{\beta_k}\langle\nabla_k,z_k-z_{k+1}\rangle-B_k D_h(z_{k+1},z_k)+A_{k+1}(f(y_{k+1})-f(x^*))-A_k(f(y_{k})-f(x^*))\\
    &\overset{\eqref{st_h}}{\leq}\left(\bar{b}_k-\frac{\alpha_k B_k}{\beta_k}\right)D_h(x^*,z_{k+1})+\frac{\alpha_k B_k}{\beta_k}D_h(x^*,x_{k+1})-\frac{\bar{\mu}\alpha_k B_k}{2\beta_k}\|z_{k+1}-x_{k+1}\|^2+\frac{B_k}{\beta_k}\langle\nabla_k,x^*-z_k\rangle\\
    &\quad+\frac{B_k}{\beta_k}\langle\nabla_k,z_k-z_{k+1}\rangle-\frac{\bar{\mu}B_k}{2}\|z_{k+1}-z_k\|^2+A_{k+1}(f(y_{k+1})-f(x^*))-A_k(f(y_{k})-f(x^*))\\
    &\leq \left(\bar{b}_k-\frac{\alpha_k B_k}{\beta_k}\right)D_h(x^*,z_{k+1}) + \frac{\alpha_k B_k}{\beta_k}D_h(x^*,x_{k+1})-\frac{\bar{\mu}\alpha_k B_k}{2\beta_k}\|z_{k+1}-x_{k+1}\|^2+\frac{B_k}{\beta_k}\langle\nabla_k,x^*-x_{k+1}\rangle\\
    &\quad+\frac{B_k}{\beta_k}\left(c(f_i(y_k)-f_i(x_{k+1}))+b\|x_{k+1}-z_k\|^2\right)+\frac{B_k}{\beta_k}\langle\nabla_k,z_k-z_{k+1}\rangle-\frac{\bar{\mu}B_k}{2}\|z_{k+1}-z_k\|^2\\
    &\quad+A_{k+1}(f(y_{k+1})-f(x^*))-A_k(f(y_{k})-f(x^*))\\
    &\leq \frac{\alpha_kB_k}{\beta_k}D_h(x^*,x_{k+1})+\frac{B_k}{\beta_k}\langle\nabla_k,z_k-z_{k+1}\rangle+\left(\frac{\bar{\mu}\alpha_kB_k}{2\beta_k}-\frac{\bar{\mu}B_k}{2}\right)\|z_{k+1}-z_k\|^2+\frac{B_k}{\beta_k}\langle\nabla_k,x^*-x_{k+1}\rangle\\
    &\quad+A_k(f(y_k)-f(x_{k+1}))+A_{k+1}(f(y_{k+1})-f(x^*))-A_k(f(y_{k})-f(x^*))
    \end{aligned}
\end{equation*}
The first equality and the third equality follows from Lemma \ref{Three-point identity}; the second equality follows from mirror descent.
The first inequality follows from the strong convexity of $h$, and the second inequality follows from \eqref{slackEOGa}
(Lemma \ref{Existenceofalpha}). With the choice of $\alpha_k$ and $\beta_k$, we have ${\alpha_kB_k}/{\beta_k}=\bar{b}_k$, which explains the last inequality. Moreover, with the choice of $B_k$ and Observation \ref{lowerboundofkappa}, we have
\begin{equation*}
    \frac{\alpha_k}{\beta_k}=\frac{\bar{b}_k}{B_k}=\frac{\gamma}{2\sqrt{\kappa}}\leq\frac{\gamma}{2}\sqrt{\frac{2-\gamma}{\gamma}}\leq\frac{\gamma}{2}\sqrt{\frac{1}{\gamma^2}}=\frac{1}{2}.
\end{equation*}
Combined with the initial bound and the relation above, we obtain the following bound:
\begin{equation*}
    \begin{aligned}
    E_{k+1}-E_k&\leq\frac{\alpha_kB_k}{\beta_k}D_h(x^*,x_{k+1})+\frac{B_k}{\beta_k}\langle\nabla_k,z_k-z_{k+1}\rangle+\left(\frac{\bar{\mu}\alpha_kB_k}{2\beta_k}-\frac{\bar{\mu}B_k}{2}\right)\|z_{k+1}-z_k\|^2+\frac{B_k}{\beta_k}\langle\nabla_k,x^*-x_{k+1}\rangle\\
    &\quad+A_k(f(y_k)-f(x_{k+1}))+A_{k+1}(f(y_{k+1})-f(x^*))-A_k(f(y_{k})-f(x^*))\\
    &\leq\frac{\alpha_kB_k}{\beta_k}D_h(x^*,x_{k+1})+\frac{B_k}{\beta_k}\langle\nabla_k,z_k-z_{k+1}\rangle-\frac{\bar{\mu}B_k}{4}\|z_{k+1}-z_k\|^2+\frac{B_k}{\beta_k}\langle\nabla_k,x^*-x_{k+1}\rangle\\
    &\quad+A_{k+1}(f(y_{k+1})-f(x_{k+1}))+\bar{a}_k(f(x_{k+1})-f(x^*))\\
    &\overset{\eqref{fenchelyoung}\eqref{uniform-quasar-convex-of-f}}{\leq}\frac{\alpha_kB_k}{\beta_k}D_h(x^*,x_{k+1})+\frac{B_k}{\bar{\mu}\beta_k^2}\|\nabla_k\|^2+\frac{\gamma B_k}{\beta_k}(f_i(x^*)-f_i(x_{k+1})-\mu D_h(x^*,x_{k+1}))\\
    &\quad+A_{k+1}(f(y_{k+1})-f(x_{k+1}))+\bar{a}_k(f(x_{k+1})-f(x^*))
    \end{aligned}
\end{equation*}
Taking the expectation, we obtain
\begin{equation*}
    \begin{aligned}
        \mathbb{E}[E_{k+1}-E_k]&\leq\frac{\alpha_kB_k}{\beta_k}D_h(x^*,x_{k+1})+\frac{B_k}{\bar{\mu}\beta_k^2}\mathbb{E}\left[\|\nabla_k\|^2\right]+\frac{\gamma B_k}{\beta_k}(f(x^*)-f(x_{k+1})-\mu D_h(x^*,x_{k+1})))\\
        &\quad+A_{k+1}\mathbb{E}[f(y_{k+1})-f(x_{k+1})]+\bar{a}_k(f(x_{k+1})-f(x^*))\\
        &\leq\frac{B_k}{\bar{\mu}\beta_k^2}\mathbb{E}\left[\|\nabla_k\|^2\right]+A_{k+1}\mathbb{E}[f(y_{k+1})-f(x_{k+1})]\\
        &\leq\left(\frac{\rho B_k}{\bar{\mu}\beta_k^2}-\frac{A_{k+1}}{2\rho L}\right)\mathbb{E}\left[\|\nabla_k\|^2\right]\leq 0
    \end{aligned}
\end{equation*}
\end{proof}
\section{Additional Simulation Results}
\label{additionalsimul}
\begin{figure*}[h]
\centering 
\subfigure{
\label{Dataset4.1}
\includegraphics[width=5.2cm,height = 3.2cm]{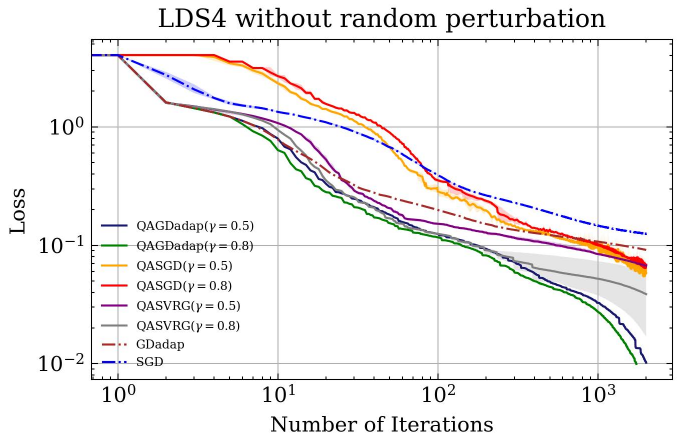}}\subfigure{
\label{Dataset4.2}
\includegraphics[width=5.2cm,height = 3.2cm]{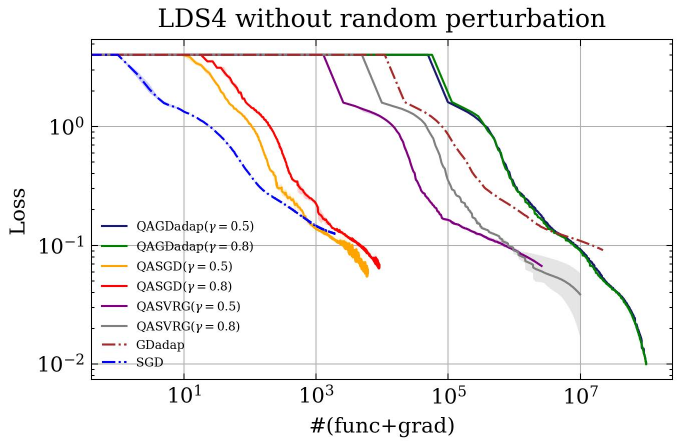}}
\subfigure{
\label{Dataset4.3}
\includegraphics[width=5.2cm,height = 3.2cm]{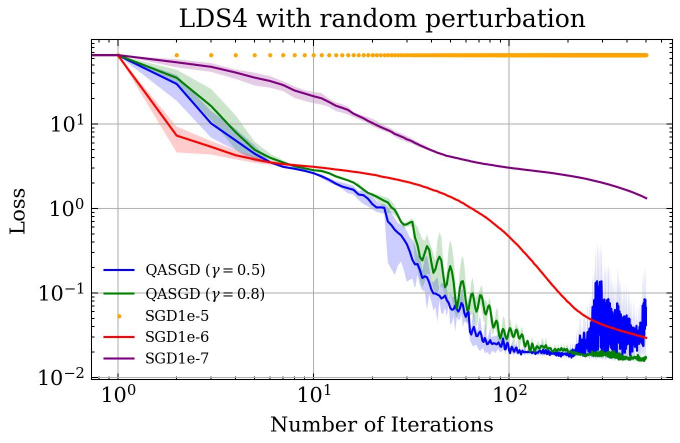}}
\subfigure{
\label{Dataset5.1}
\includegraphics[width=5.2cm,height = 3.2cm]{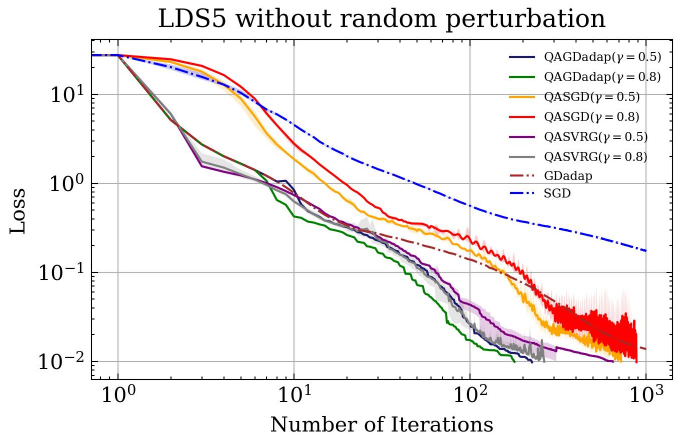}}\subfigure{
\label{Dataset5.2}
\includegraphics[width=5.2cm,height = 3.2cm]{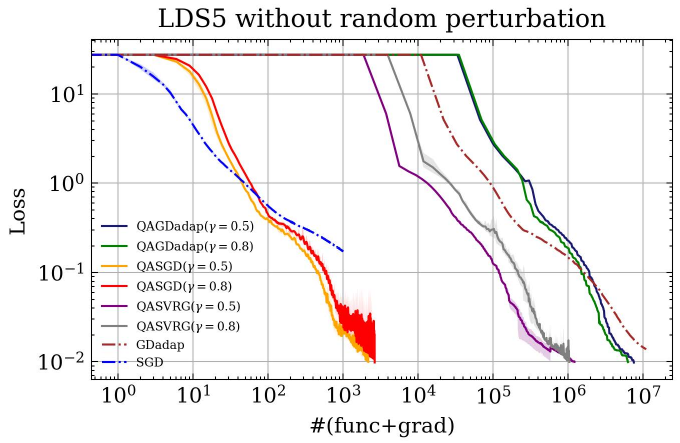}}
\subfigure{
\label{Dataset5.3}
\includegraphics[width=5.2cm,height = 3.2cm]{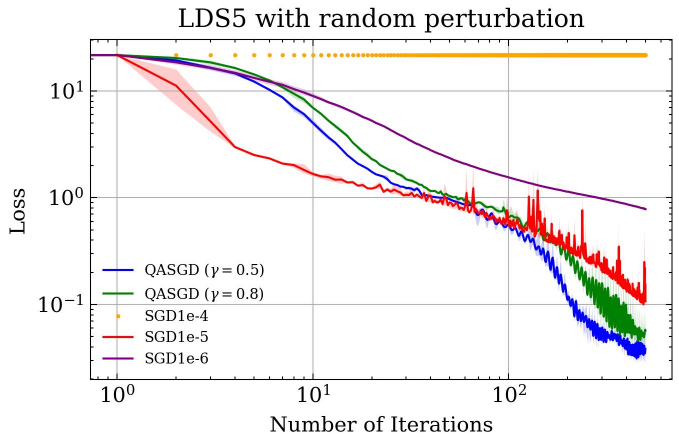}}
\caption{Evaluation on two different LDS instances with random seed in $\{12, 36\}$. We choose $\epsilon=10^{-2}$, the stepsize to be $1\times10^{-6},1\times10^{-5}$ for SGD, $L=1\times 10^7,5\times 10^6$ for QASGD and $L=3\times 10^6,1\times 10^5$ for QASVRG in LDS4 and LDS5. The flat line in the third column means the loss blows up to infinity with this choice of stepsize.}
\label{LDSadditional}
\end{figure*}

\begin{figure*}[ht]
\centering 
\subfigure{
\label{Dataset6.1}
\includegraphics[width=5.2cm,height = 3.2cm]{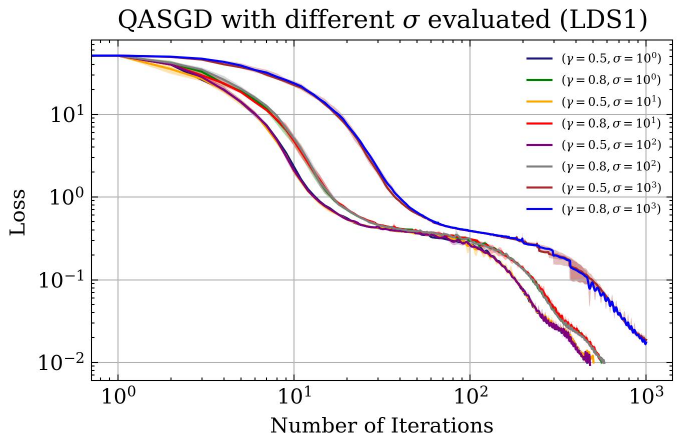}}\subfigure{
\label{Dataset6.2}
\includegraphics[width=5.2cm,height = 3.2cm]{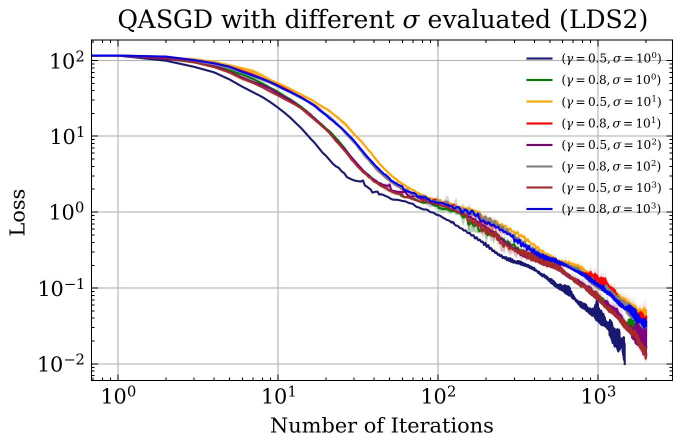}}
\subfigure{
\label{Dataset6.3}
\includegraphics[width=5.2cm,height = 3.2cm]{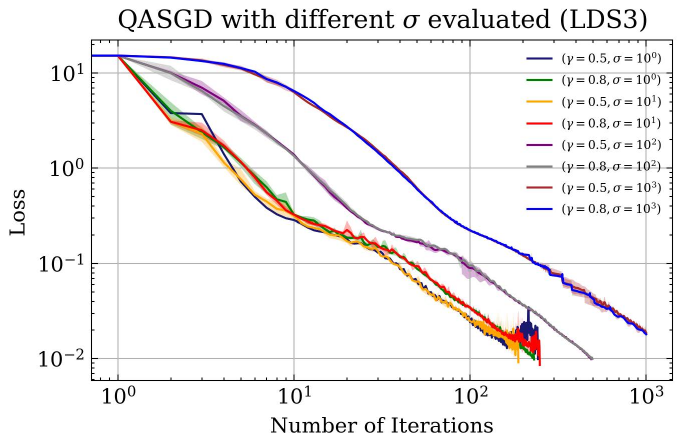}}
\subfigure{
\label{Dataset7.1}
\includegraphics[width=5.2cm,height = 3.2cm]{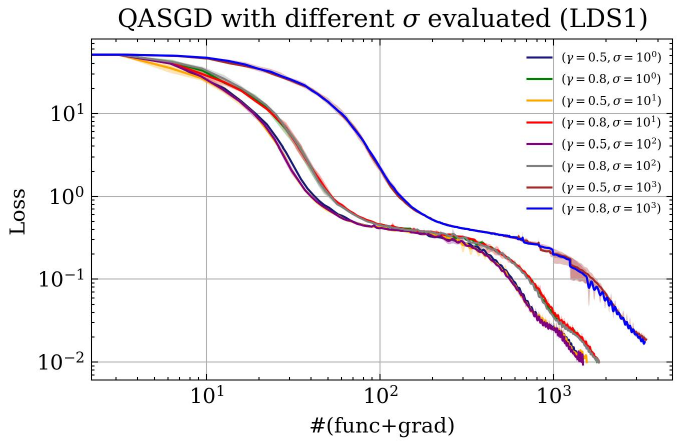}}\subfigure{
\label{Dataset7.2}
\includegraphics[width=5.2cm,height = 3.2cm]{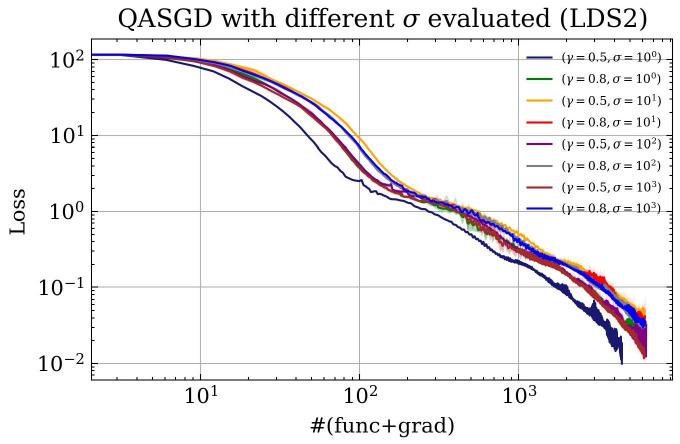}}
\subfigure{
\label{Dataset7.3}
\includegraphics[width=5.2cm,height = 3.2cm]{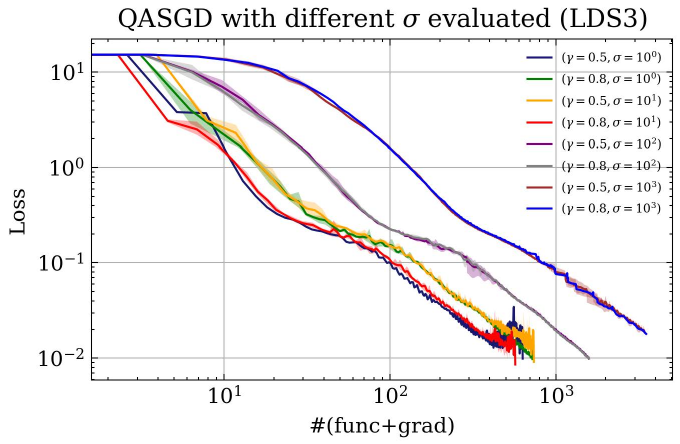}}
\caption{Evaluation of QASGD with different $\sigma$ on three LDS instances}
\label{LDSadditionalsigma}
\end{figure*}
We provide a contrived experiment by constructing an objective satisfying all the assumptions required. Consider the following optimization problem
\begin{equation}
\label{expproblem}
    \min_{x\in\mathbb{R}^d}\left[f(x) = \frac{1}{n}\sum_{i=1}^ng_{\gamma}(b_ia_i^{\mathsf{T}}x)+\frac{\mu}{2}\|x\|^2\right],\quad g_{\gamma}(x)= \begin{cases}
\frac{x^{\gamma}-1}{\gamma}+\frac{1}{2},\quad &x \geq 1, \\
\frac{x^2}{2},\quad &0 \leq x \leq 1,\\
0,\quad &x \leq 0,
\end{cases}
\end{equation}
where $(a_i,b_i)_{i=1,...,n}$ is training data with $a_i\in\mathbb{R}^d$ and $b_i\in\{+1,-1\}$; $\mu\geq 0$, and $f(x)$ satisfies Assumption \ref{boundedgrad}. $f(x)$ is $\mu$-strongly $\gamma$-quasar-convex and $L$-smooth by properties of quasar-convex functions introduced in (\citet{hinder}, D.3), where $L=\sum_{i=1}^n\|a_i\|/n+0.5\mu$. We choose $\gamma\in\{0.5,0.8\}$ and normalize each $a_i$ for simplicity so that $L=1+0.5\mu$. Note that each $g_{\gamma}(b_ia_i^{\mathsf{T}}x)$ has at least one common minimizer. Therefore, Assumption \ref{interpolation} is also satisfied by $f$. We use the following multi-classification dataset from \citet{Dua:2019}, which we treat as binary classification datasets.
We have $n=1372$ and $d=4$ according to the dataset. We set $\epsilon=10^{-2}$ when $\mu=0$, and set $\epsilon=10^{-3}$ when $\mu>0$. We generate the error bar the same way as simulations in section \ref{simulation}. Figure \ref{SVMadditional} shows that QASGD enjoys faster convergence than SGD while QASVRG enjoys fast convergence and lower complexity than QAGD and GD. When $\mu=0.002$,  Figure \ref{SVMadditional} also shows the superiority of QASVRG (Option \rom{1}) in terms of convergence speed and complexity given $\kappa\epsilon\approx0.5<1$.
\begin{figure*}[ht]
\centering 
\subfigure{
\label{Dataset8.1}
\includegraphics[width=4.0cm,height = 2.5cm]{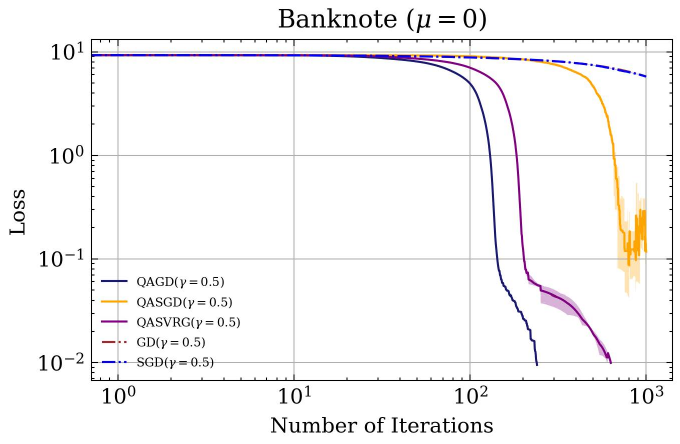}}\subfigure{
\label{Dataset8.2}
\includegraphics[width=4.0cm,height = 2.5cm]{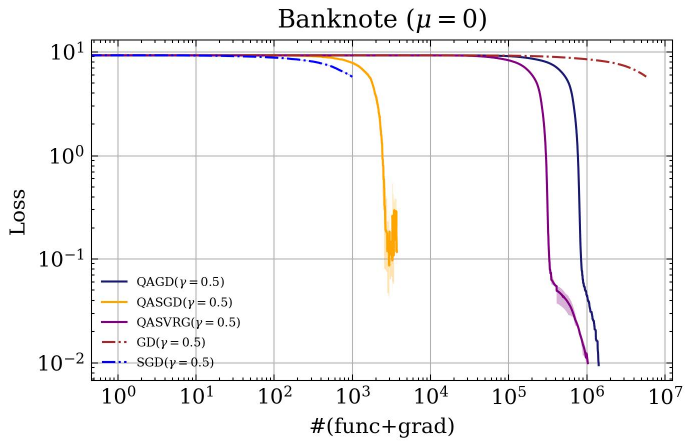}}
\subfigure{
\label{Dataset8.3}
\includegraphics[width=4.0cm,height = 2.5cm]{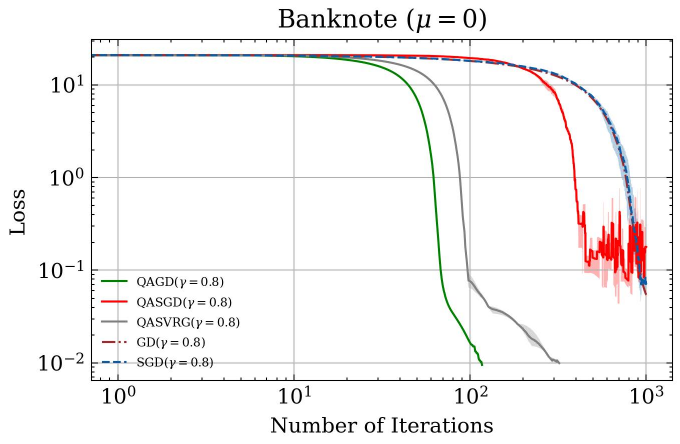}}
\subfigure{
\label{Dataset8.4}
\includegraphics[width=4.0cm,height = 2.5cm]{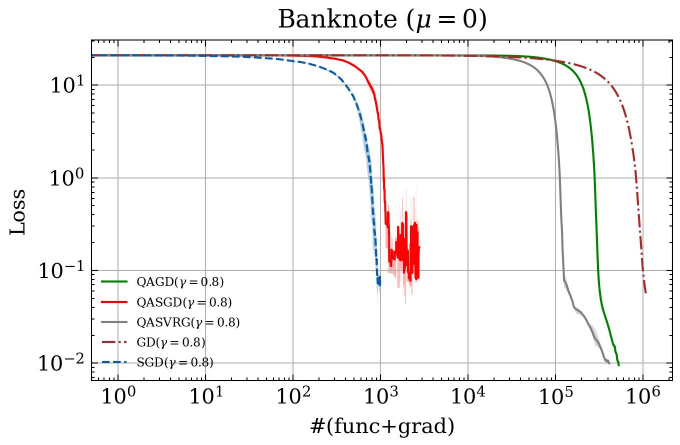}}\\
\subfigure{
\label{Dataset9.1}
\includegraphics[width=4.0cm,height = 2.5cm]{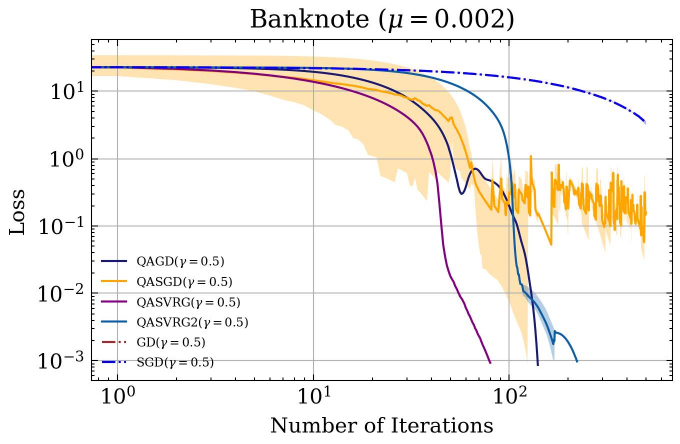}}
\subfigure{
\label{Dataset9.2}
\includegraphics[width=4.0cm,height = 2.5cm]{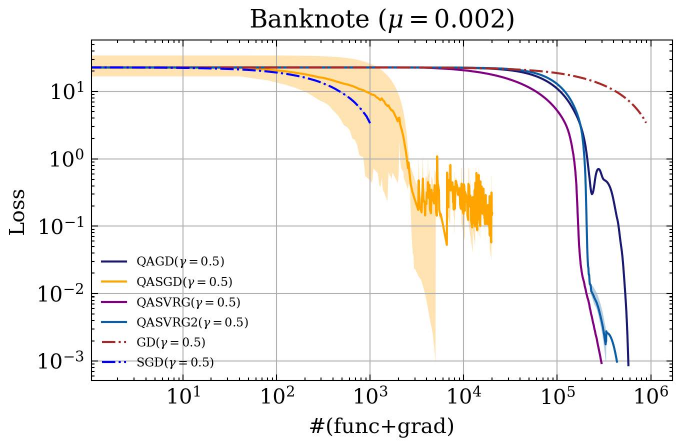}}
\subfigure{
\label{Dataset9.3}
\includegraphics[width=4.0cm,height = 2.5cm]{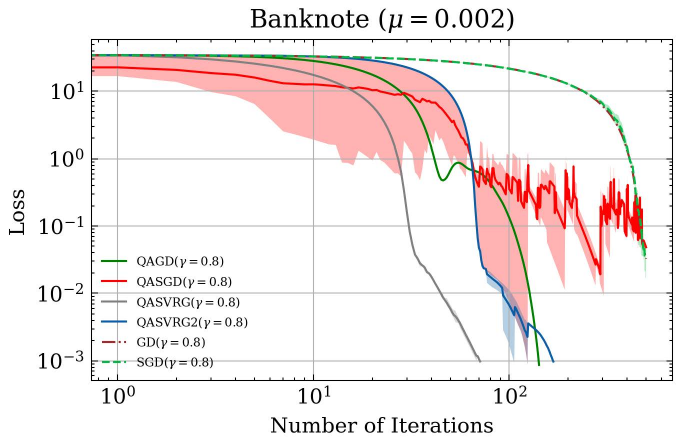}}
\subfigure{
\label{Dataset9.4}
\includegraphics[width=4.0cm,height = 2.5cm]{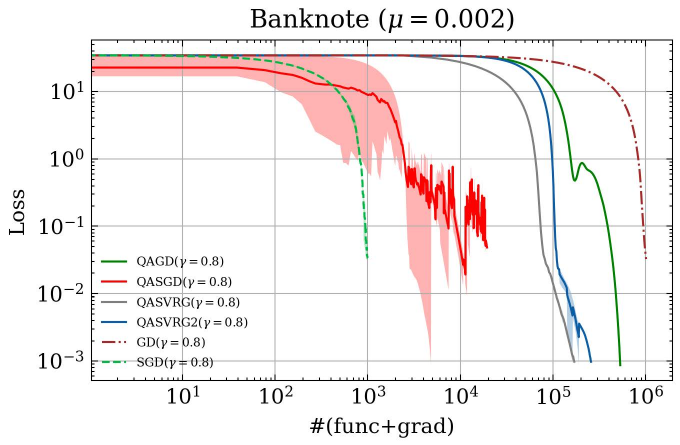}}
\caption{Evaluation of each algorithm on problem $\eqref{expproblem}$}
\label{SVMadditional}
\end{figure*}
We also compare our methods with GD, QAGD and SGD on solving empirical risks of GLM with logistic link function $\sigma(z)=(1+\exp(-z))^{-1}$. Consider the following optimization problem
\begin{equation}
\label{empiricalGLM}
    \min\left\{f(w)=\frac{1}{n}\sum_{i=1}^n\left[\left(\sigma\left(w^{\mathsf{T}}x_i\right)-y_i\right)^2\right]\right\},
\end{equation}
where $x_i\sim \mathcal{N}(0,I)$, $w^*\sim\mathcal{N}(0,I)$ and $y_i=\sigma\left(w_*^{\mathsf{T}}x_i\right)$ for each $i\in[n]$. In our experiment, we choose $n=5000$, $d=50$ and the initial iterate $w_0\sim\mathcal{N}(0,100I)$. Since it is intractable to compute the parameter of quasar-convexity $\gamma$ and smoothness $L$, we evaluate our methods with $\gamma=0.5$ and $L=10^5$ by extensive grid search. 
\begin{figure*}[h!]
\centering 
\subfigure{
\label{Dataset10.1}
\includegraphics[width=4.0cm,height = 2.5cm]{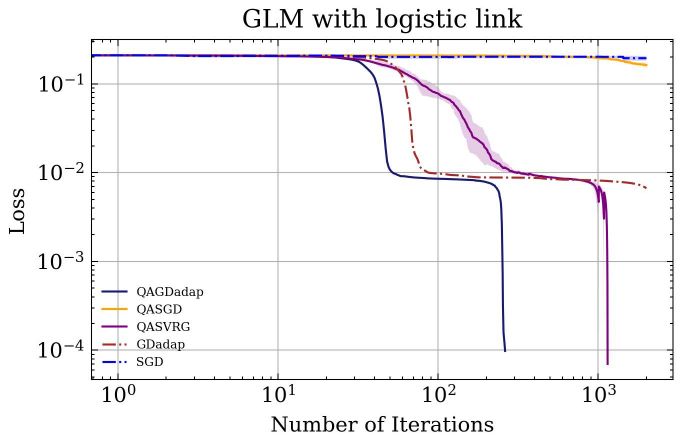}}\subfigure{
\label{Dataset10.2}
\includegraphics[width=4.0cm,height = 2.5cm]{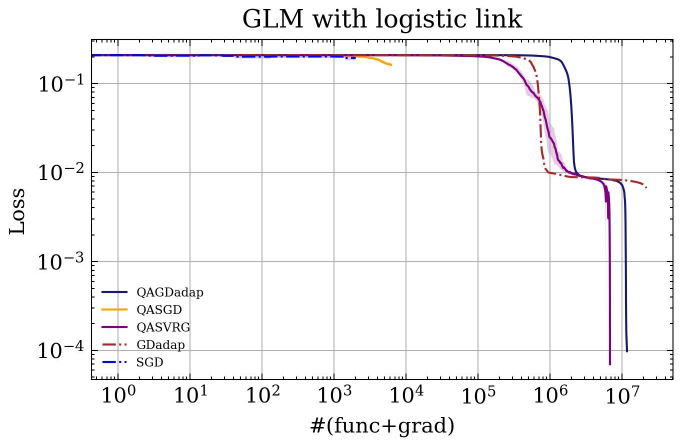}}
\caption{Evaluation of each algorithm on problem $\eqref{empiricalGLM}$}
\label{GLMsimulation}
\end{figure*}

\end{document}